\documentclass[12pt]{amsart}
\usepackage{mathptmx} 
\usepackage[shortlabels]{enumitem}
\usepackage{empheq}
\usepackage{amsfonts}
\usepackage{amsmath, amsrefs, amsthm, amssymb}
\usepackage[active]{srcltx}		
\usepackage{pdfsync}					
\usepackage{graphicx}
\usepackage [latin1]{inputenc}
\usepackage{bbm}
\usepackage{amsfonts}
\usepackage{amsthm}
\usepackage{graphicx}
\usepackage{framed}
\usepackage{multicol,multienum}
\usepackage{graphicx} 
\usepackage{booktabs} 
\usepackage{mathrsfs}
\usepackage{tcolorbox}
\usepackage{color}
\usepackage{xcolor}

\newtheorem{theorem}{Theorem}[section]

\newtheorem{lemma}[theorem]{Lemma}

\theoremstyle{definition}
\newtheorem{definition}[theorem]{Definition}

\newtheorem{remark}[theorem]{Remark}

\numberwithin{equation}{section}

\begin{document}

\title [The influence of noise for ISM]{{{\rmfamily  The influence of stochastic forcing on strong solutions to the Incompressible Slice Model in 2D bounded domain}}} \thanks{This work was partially supported by the National Natural Science Foundation of China (Projects \# 11701198, 11701161, 11901584 and 11971185), the Fundamental Research Funds for the Central Universities (Projects \# 5003011025 and WUT: 2020IA005).}

\author{Lei Zhang}
\address{School of Mathematics and Statistics, Hubei Key Laboratory of Engineering Modeling  and Scientific Computing, Huazhong University of Science and Technology  Wuhan 430074, Hubei, P.R. China.}
\email{lei\_zhang@hust.edu.cn}

\author{Yu Shi}
\address{School of Science, Wuhan University of Technology, Wuhan 430070, Hubei, P.R. China}
\email{shiyu87@whut.edu.cn}

\author{Chaozhu Hu}
\address{School of Science, Hubei University of Technology, Wuhan 430068, Hubei, P.R. China}
\email{huchaozhu0035@hbut.edu.cn}

\author{Weifeng Wang}
\address{School of Mathematics and Statistics, South-Central University for Nationalities, Wuhan, 430074,China}
\email{wwf87487643@163.com}

\author{Bin Liu}
\address{School of Mathematics and Statistics, Hubei Key Laboratory of Engineering Modeling  and Scientific Computing, Huazhong University of Science and Technology  Wuhan 430074, Hubei, P.R. China.}
\email{binliu@hust.edu.cn}


\date{\today}

\begin{abstract}
The Cotter-Holm Slice Model (CHSM) was introduced to study the behavior of whether and specifically the formulation of atmospheric fronts, whose prediction is fundamental in meteorology. Considered herein is the influence of stochastic forcing on the Incompressible Slice Model (ISM) in a
smooth 2D bounded domain, which can be derived by adapting the Lagrangian function in Hamilton's principle for CHSM to the Euler-Boussinesq
Eady incompressible case. First, we establish the existence and uniqueness of local pathwise solution (probability strong solution) to the
ISM perturbed by nonlinear multiplicative stochastic forcing in Banach spaces $W^{k,p}(D)$ with $k>1+1/p$ and $p\geq 2$. The solution
is obtained by introducing suitable cut-off operators applied to the $W^{1,\infty}$-norm of the velocity and temperature fields, using the
stochastic compactness method and the Yamada-Watanabe type argument based on the Gy\"{o}ngy-Krylov characterization of convergence in probability.
Then, when the ISM is perturbed by linear multiplicative stochastic forcing and the potential temperature does not vary linearly on the $y$-direction, we prove that the associated Cauchy problem admits a unique
global-in-time pathwise solution with high probability, provided that the initial data is sufficiently small or the diffusion parameter
is large enough. The results partially answer the problems left open in Alonso-Or{\'a}n et al. (Physica D 392:99--118, 2019, pp. 117).
\end{abstract}

\maketitle
 \newpage
 \tableofcontents
\section{Introduction}

In this paper, we consider the following stochastic Incompressible Slice Model (SISM) in a bounded domain $D\subset \mathbb{R}^2$ with smooth boundary $\partial D$ \cite{1,2}:
\begin{equation}
\left\{
\begin{aligned}\label{1.1}
&\mathrm{d}u_S+  (u_S\cdot \nabla )  u_S\mathrm{d}t+\nabla p  \mathrm{d}t -f u_T\widehat{x}\mathrm{d}t-\frac{g}{\theta_0}\theta _S\widehat{z} \mathrm{d}t=  \sigma_1(u_S,u_T,\theta_S)\mathrm{d}\mathcal {W},\\
&\nabla\cdot  u_S =0,\\
&\mathrm{d} u_T+ (u_S\cdot \nabla)   u_T\mathrm{d}t+fu_S\cdot\widehat{x} \mathrm{d}t +\frac{g}{\theta_0} zs\mathrm{d}t= \sigma_2(u_S,u_T,\theta_S)\mathrm{d}\mathcal {W}, \\
&\mathrm{d}\theta_S+ (u_S\cdot \nabla ) \theta_S\mathrm{d}t+u_Ts \mathrm{d}t  = \sigma_3(u_S,u_T,\theta_S)\mathrm{d}\mathcal {W},
\end{aligned}
\right.\quad  t>0, ~x\in D,
\end{equation}
 with the boundary condition
\begin{eqnarray}\label{1.2}
  u_S \cdot \vec{n}=0,\quad t>0, ~x\in \partial D,
\end{eqnarray}
and the initial conditions
\begin{eqnarray}\label{1.3}
u_S(0,x)=u_S^0(x),\quad  u_T(0,x)=u_T^0(x),\quad \theta_S(0,x)=\theta_T^0(x), \quad\quad x\in D.
\end{eqnarray}
Here, $\vec{n}$ stands for  the outward unit normal vector to the boundary $\partial D$, $g$ is the acceleration due to gravity, $\theta_0$ is the reference temperature, $f$ is the Coriolis parameter, which is assumed to be a constant, and $s$ is a constant which measures the variation of the potential temperature in the transverse direction. In \eqref{1.1}, the operator $\nabla$ denotes the 2D gradient in the slice, $p$ is the pressure obtained from incompressibility of the flow in the slice (div $u_S=0$), while $\widehat{x}$ and  $\widehat{z}$ denote horizontal and vertical unit vectors in the slice.
The process $\mathcal {W}$ is a cylindrical Brownian motion defined on a separable Hilbert space $\mathfrak{A}$, and $\sigma_i(\cdot)$, $i=1,2,3$ are diffusion coefficients satisfying proper boundedness and growth conditions, we refer to Subsection 1.1 for more details. The boundary condition \eqref{1.2} means that the fluid velocity component $u_S$ is taken to be tangent to the boundary, and the initial datum in \eqref{1.3} are random variables with suitable regularity.

In atmosphere and ocean science, the slice models are frequently used to describe the formation of fronts \cite{3,4,5}, and the slice framework is proved to be very useful in the study of whether fronts \cite{6}. These fronts arise when there is a strong nort-south temperature gradient (maintained by heating at the Equator and cooling at the Pole), which maintains a vertical shear flow in the east-west direction through geostrophic balance.  Recently, the Cotter-Holm Slice Model (CHSM) for oceanic and atmospheric fluid motions taking place in vertical slice domain $D\subset \mathbb{R}^2$ was introduced by using the Hamiltonian's variational principle \cite{2}. The fluid motion in the vertical slice is coupled dynamically to the flow velocity transverse to the slice, which is assumed to vary linearly with distance normal to the slice. The assumption about the transverse flow through the vertical slice in the CHSM simplifies its mathematics while still capturing important aspects of the 3D flow.

The deterministic Incompressible  Slice Model (ISM) (i.e., $\sigma_i(\cdot)\equiv 0$ in \eqref{1.1}, $i=1,2,3$) is a modified version of CHSM,  which can be obtained by adapting the Lagrangian function
$$
L(u_S,u_T,\theta_S,d,p)=\int_D\left(\frac{d}{2}(|u_S|^2+|u_T|^2))+dfu_Tx+\frac{g}{\theta_0}d\theta_Sz+(1-d)p\right)dV
$$
in Hamiltonian's principle for CHSM to the Euler-Boussinesq Eady incompressible case; see \cite{2,1} for more details. It should be emphasised that the ISM is potentially useful in numerical simulations of fronts and hence useful in the parameter-studies for numerical weather predictions, since the domain consists of two-dimensional slice, computer simulations of ISM take much less time to run a full three-dimensional model.  We refer the readers to the references \cite{4,6,7,8,9} for applications of this type of idealized models to predict and examine the formulation and evolution of weather fronts. The ISM enjoys some conservation laws due to its variational character.  For example, the ISM conserves the circulation of $v_S:=su_S-(u_T+fx)\nabla \theta _S$ on loops $c(u_S)$ carried by $u_S$:
$$
\frac{\mathrm{d}}{\mathrm{d} t} \oint_{c(u_S)} v_s\cdot \mathrm{d} s= \oint_{c(u_S)}  \mathrm{d} \pi =0.
$$
The ISM potential vorticity defined by $q:=\mbox{curl}~ (v_S)\cdot \hat{y}$ is also conserved on fluid parcels:
$$
\frac{\mathrm{D}q}{\mathrm{d} t} = \partial_t q +u_S \cdot \nabla q =0.
$$
It is worth pointing out that the potential vorticity above differs from the usual ones in fluid dynamics, since we have to take the transverse velocity $u_T$, the Coriocity force $f$ and the potential temperature $\theta_S$ into consideration. For any smooth function $\Phi$ of the potential vorticity $q$, the ISM conserves the generalized enstrophy
$$
C_\Phi= \int_D \Phi(q)\mathrm{d} V.
$$
Moreover, the ISM also conserves the energy
$$
\frac{\mathrm{d}}{\mathrm{d} t}E(t) = \frac{\mathrm{d}}{\mathrm{d} t}\int_D  \bigg(\frac{1}{2}(|u_S|^2+ u_T ^2)-\frac{g}{\theta_0} z \theta_S\bigg) \mathrm{d} V=0.
$$
As far as we know, few works are available concerning the qualitative analysis for the ISM. In \cite{1}, by virtue of the conserved quantities for ISM listed above, Alonso-Or{\'a}n and de Le{\'o}n studied the Arnold's stability (cf. \cite{10}) around a restricted class of equilibrium solutions by the Energy-Casimir method \cite{11}. Besides, they established the local well-posedness of solutions to the ISM in Hilbert spaces $H^s(D)$ (for integers $s>2)$ in spirit of the works by Temam \cite{12} and Kato and Lai \cite{13}. Moreover, they also constructed a blow-up criterion based on the analysis of $L^p$ norms for the gradients of $u_T$ and $\theta_S$, which is different with the well-known Beale-Kato-Majda criterion for the 3D Euler equation \cite{14}, that is, if $\int_0^{T^*}\|\mbox{curl}~ u_S(t)\|_{L^\infty} \mathrm{d} t<\infty$, then the corresponding solution stays regular on $[0,T^*]$. The main difficulty stems from the fact that it seems to be hard to control properly $u_T$ and $\theta_S$ in terms of the vorticity only; cf. Remark 6.3 in \cite{1}.

We also mention that the SISM resembles the 2D Boussinesq equations without viscosity, i.e., by taking $f=0$, $u_T\equiv0$ and $\sigma_i(\cdot)\equiv 0$ ($i=1,2,3$) in \eqref{1.1}, which has a wide range of applications in geophysics and fluid
mechanics, such as the modeling of large scale atmospheric and oceanic flows that are responsible for cold fronts and jet stream and the study of Rayleigh-B\'{e}nard convection; cf. \cite{15,16,17,18}. Nevertheless, the variable $u_T$ representing the transversal velocity to the slice, which naturally appears when deriving the ISM and gives rise to a more complex structure in the coupled system of equations. During past decades, the Boussinesq equations have attracted much attention and abundant references can be found concerning the well-posedness and global regularity \cite{19,20,21,23}. However, whether the classical solutions to the Boussinesq equations blow up in finite time is still a fundamental open problem \cite{22}, and we would like to refer the readers to the references \cite{24,25,26} for the recent important advances in the study of Boussinesq equations.

To our best knowledge, the influence of noises on dynamic behavior of solutions to the ISM and the existence of global solutions have not been studied, which were leaved as open questions in recent work \cite{1}. Taking into account the random environment surrounding the velocity field and the temperature filed, it is natural to consider the ISM perturbed by certain stochastic forcing. The purpose of this paper is to seek for probabilistically the local and global pathwise solution for the SISM with multiplicative noises. Notice that the introduction of the stochasticity in the ISM can help account for the uncertainty coming from the numerical methods used to resolve these equations. In fact, the dynamic behavior of fluid models perturbed by different kinds of noises has been widely studied in past decades, we refer to the stochastic fluid models \cite{27,28,29,30,31,32,33,34,52,53} and the stochastic dispersive PDEs \cite{35,36,37,54} to just mention a few. We also mention that the presence of stochastic perturbation in fluid equations can lead to new phenomena. For instance, while uniqueness may fail for the deterministic transport equation, Flandoli et al. \cite{39} proved that a multiplicative stochastic perturbation of Brownian type is enough to render the equation well-posed; see also \cite{38}. In \cite{40}, Brze{\'z}niak et al. proved that the 2D Navier-Stokes system driven by degenerate noise has a unique invariant measure and hence exhibits ergodic behavior in the sense that the time average of a solution is equal to the average over all possible initial data, which is quite different with the deterministic case.

The primary goal of this paper is to investigate the influence of noises on strong solutions to the ISM. First, we establish the local existence and uniqueness of pathwise solutions to the ISM perturbed by general multiplicative noise (cf. Theorem \ref{thm:1.2}), which improves the existed results in Hilbert spaces $H^s(D)$ for the deterministic counterpart \cite{1}. The main challenging we encounter in the proof is that it is impossible to prove our result by applying the classic energy compactness method, since the stochastic forcing  $\sigma_i(\cdot)\mathrm{d}\mathcal {W}$, $i=1,2,3$ have a bad influence on the $t$-regularity of solutions. Thereby some new techniques have to be introduced to overcome these difficulties. Second, we are looking for sufficient conditions which lead to global pathwise solutions to the ISM. In particular, we shall consider the ISM perturbed by linear multiplicative noise such that the associated SISM can be transformed into random PDEs. It is interesting to prove that, if the potential temperature do not vary linearly on the $y$-direction and the absolute value of the diffusion parameter $|\alpha|$ is large enough, then the local pathwise solutions to the SISM become global ones with high probability (cf. Theorem \ref{thm:1.3}). This fact shows that the linear multiplicative noise has a regularization effect on the pathwise solutions with respect to time variable.

\subsection{Notations and assumptions}

Now we recall some deterministic and stochastic integrants which will be widely used in the following paper. Given a smooth simply-connected bounded domain $D\subset \mathbb{R}^2$, we denote by $L^p(D;\mathbb{R}^d)$ be the usual Lebesgue space, and by $W^{k,p}(D;\mathbb{R}^d)$ the usual Sobolev space with the norm
$$
\|u\|_{W^{k,p}}^p =\sum_{|\alpha|\leq k}\|\partial^\alpha u\|_{L^p}^p.
$$
To treat the fluid velocity component, for any integers $d\geq1$, $k\geq0$ and $p\geq 2$, we introduce the following working spaces
$$
X^{k,p}(D;\mathbb{R}^d):=\{u\in W^{k,p}(D;\mathbb{R}^d);~~ \div~ u=0,~u|_{\partial D}\cdot \vec{n}=0\},
$$
which is endowed with the standard norm in $W^{k,p}$. When $p=2$, the corresponding spaces turn out to be Hilbert and are denoted by $X^{k}(D;\mathbb{R}^d)$, while the interior product is given by
$$
(u,v)_{X^k}=\sum_{|\alpha|\leq k}(\partial^\alpha u,\partial^\alpha v)_{L^2}.
$$
As the solutions in present paper is not first-order differentiable in time, we shall consider the time-fractional Sobolev spaces. Let $B$ be a Banach space, we define
$$
W^{\vartheta,p}([0,T];B)=\bigg\{f\in L^p([0,T];B);~[f]_{\vartheta,p,T}^p:=\int_0^T\int_0^T\frac{\|f(t)-f(\bar{t})\|_{B}^p}{|t-\bar{t}|^{1+\vartheta p}}\mathrm{d}t\mathrm{d}\bar{t}<\infty\bigg\},
$$
for any $\vartheta\in (0,1)$, $p\geq2$, which is endowed with the norm
$$
\|f\|_{W^{\vartheta,p}([0,T];B)}^p= \|f\|_{L^p([0,T];B)}^p+[f]_{\vartheta,p,T}^p.
$$
Moreover, in order to characterize the velocity fields $u_S$, $u_T$ and the temperature field $\theta_S$, we introduce the following  Cartesian product
$$
\mathcal {Z}^{k,p}(D):=X^{k,p}(D;\mathbb{R}^2)\times W^{k,p}(D;\mathbb{R})\times W^{k,p}(D;\mathbb{R}).
$$

Recalling the Helmholtz-Hodge decomposition theorem (cf. Section 3.3 in \cite{42}), we can decompose any function $v\in L^2(D;\mathbb{R}^d)$ as
$$
v=w+\nabla p,
$$
where $(w,\nabla p)_{L^2}=0$,  $w\in X^0(D;\mathbb{R}^d)$ and $p\in H^1 (D;\mathbb{R})$. The linear operator $P:v\mapsto w$ is called the Leray projector, and we have $Pv= (I- Q)v$ for any $v\in L^2(D;\mathbb{R}^d)$, where $Qv=-\nabla p$ and $p$ solves the following Neumann problem
\begin{equation} \label{1.4}
\begin{cases}
     -\Delta p= \div ~v,& x\in D,\\
     \frac{\partial p}{\partial \vec{n}}=v\cdot \vec{n},& x\in \partial D.
\end{cases}	
\end{equation}
If $v\in W^{k,p}(D;\mathbb{R}^d)$, then $v|_{\partial D}\cdot \vec{n}\in W^{k-1/p,p}(\partial D;\mathbb{R}^d)$, and the classic theory for elliptic Neumann problem \cite{51} gives
$$
\|Pv\|_{X^{k,p}}\leq C\|v\|_{W^{k,p}},
$$
which implies that $P:W^{k,p}(D;\mathbb{R}^d)\mapsto X^{k,p}(D;\mathbb{R}^d)$ is a linear continuous operator.

To make sense of the stochastic forcing, let $\mathcal {S}:=(\Omega,\mathcal {F},\mathbb{P},(\mathcal {F}_t)_{t\geq0})$ be a complete filtration probability space, and $(\beta_j )_{j\geq1}$ be mutually independent real-valued standard Wiener processes relative to $(\mathcal {F}_t )_{t\geq0}$. Let $(e_j )_{j\geq1}$ be a complete orthonormal system in a separate Hilbert space $\mathfrak{A}$, one can formally define the cylindrical Wiener process $\mathcal {W}$ in $\mathfrak{A}$ by
\begin{eqnarray}
\mathcal {W}(t,\omega) =\sum_{j\geq 1}  e_j\beta_j (t,\omega).\nonumber
\end{eqnarray}
To ensue the convergence of the last series, one need to introduce an auxiliary space
$$
\mathfrak{A}_0:=\Bigg\{u=\sum_{j\geq1} a_je_j;~\sum_{j\geq1} \frac{a_j^2}{j^2}<\infty\Bigg\}\supset \mathfrak{A},
$$
which is endowed with the norm $\|u\|_\mathfrak{A}^2=\sum_{j\geq1} a_j^2 j^{-2}$, for any $u=\sum_{j\geq1} a_je_j \in \mathfrak{A}$. Note that the canonical injection $\mathfrak{A}\hookrightarrow\mathfrak{A}_0$ is Hilbert-Schmidt, which implies that for any $T>0$  we have $\mathcal {W} \in C([0, T];\mathfrak{A}_0)$ almost surely.

Given another separate Hilbert space $X$, we denote by $L_2(\mathfrak{A},X)$ the collection of Hilbert-Schmidt operators from $\mathfrak{A}$ into $X$, namely,
$\|H\|_{L_2(\mathfrak{A};X)}^2=\sum_{j\geq 1} \|H e_j\|_{X}^2<\infty$ for any $H\in L_2(\mathfrak{A};X)$. Let $H$ be a $X$-valued predictable process in $ L^2(\Omega;L_{loc}^2([0,\infty);L_2(\mathfrak{A},X)))$, one can define the It\^{o}-type stochastic integration
\begin{eqnarray} \label{1.6}
\int_0^tH(r)\mathrm{d}\mathcal {W}_r=\sum_{j\geq 1} \int_0^tH(r)e_j\mathrm{d}\mathcal \beta_j(r) ,
\end{eqnarray}
which is indeed a $X$-valued square integrable martingale. We also remark that above definition of the stochastic integration does not depend on the choice of the space $\mathfrak{A}_0$ (cf. \cite{43}). Moreover, for any $p\geq1$ and $t>0$, there exists a positive constant depend only on $p$ such that the following Burkholder-Davis-Gundy (BDG) inequality holds:
\begin{eqnarray}\label{1.7}
\mathbb{E}\Bigg(\sup_{s\in [0,t]}\bigg\|\int_0^s H(r)\mathrm{d}\mathcal {W}_r\bigg\|_X^p\Bigg)\leq C\mathbb{E}\bigg( \int_0^t \|H(r)\|_{L_2(\mathfrak{A},X)}^2\mathrm{d} r \bigg)^{p/2},
\end{eqnarray}
which can be equivalently rewritten as
\begin{eqnarray*}
\mathbb{E}\Bigg(\sup_{s\in [0,t]}\bigg\|\sum_{j\geq 1} \int_0^tH_j(r) \mathrm{d}\mathcal \beta_j(r)\bigg\|_X^p\Bigg)\leq C\mathbb{E}\Bigg( \sum_{j\geq 1} \int_0^t \|H_j(r) \|_{X}^2\mathrm{d} r \Bigg)^{p/2},
\end{eqnarray*}
with $H_j(r)=H(r)e_j$, $j=1,2,...$.

By applying the Leray projector $P$ to the system \eqref{1.1}, one can delete the pressure term $\nabla p$ and absorbing the boundary condition \eqref{1.2}. In the following we shall mainly focus on the following system of PDEs in $\mathcal {Z}^{k,p}(D)$:
\begin{equation} \label{1.5}
\left\{
\begin{aligned}
&\mathrm{d}u_S+  P(u_S\cdot \nabla)   u_S\mathrm{d}t-P(f u_T\widehat{x}) \mathrm{d}t- P(\frac{g}{\theta_0}\theta _S\widehat{z}) \mathrm{d}t= P\sigma_1(u_S,u_T,\theta_S)\mathrm{d}\mathcal {W},\\
&Pu_S=u_S,\\
&\mathrm{d} u_T+ (u_S\cdot \nabla)   u_T\mathrm{d}t+fu_S\cdot\widehat{x} \mathrm{d}t +\frac{g}{\theta_0} zs\mathrm{d}t=\sigma_2(u_S,u_T,\theta_S)\mathrm{d}\mathcal {W}, \\
&\mathrm{d}\theta_S+ (u_S\cdot \nabla ) \theta_S\mathrm{d}t+u_Ts \mathrm{d}t  = \sigma_3(u_S,u_T,\theta_S)\mathrm{d}\mathcal {W},\\
&u_S(0,\cdot)=u_S^0,\quad  u_T(0,\cdot)=u_T^0,\quad \theta_S(0,\cdot)=\theta_T^0 .
\end{aligned}
\right.
\end{equation}

In  this paper, we shall consider the solutions evolving on $W^{k,p}(D;\mathbb{R}^d)$ for any $p\geq2$ and $k>1+\frac{2}{p}$, for any $k\geq0$, $p\geq2$ and $d\geq1$, we introduce the following spaces as \cite{48}:
\begin{equation*}
\begin{split}
 \mathbb{W}^{k,p} :=\Bigg\{\rho: D\mapsto L_2 ; ~\rho_j(\cdot)=\rho(\cdot)e_j\in W^{k,p}(D;\mathbb{R}^d),~\sum_{|\alpha|\leq k}\|\partial^\alpha\rho \|_{L^p(D;L_2(\mathfrak{A};\mathbb{R}^d))}^p<\infty\Bigg\},
 \end{split}
\end{equation*}
which is a Banach space endowed with the norm
$$
\|\rho\|_{\mathbb{W}^{k,p}}^p:=\sum_{|\alpha|\leq k}\int_D \Bigg(\sum_{j\geq1}|\partial^\alpha\rho_j(x)|^2\Bigg)^{p/2}dx.
$$
Let $X,Y,Z$ and $V$ be Banach spaces with $X,Y,Z$ continuously embedded into $ L^\infty(D;\mathbb{R}^d)$, we denote the space of locally bounded functionals by
\begin{equation*}
\begin{split}
\mbox{Bnd}_{loc}(X\times Y \times Z;V):=&\Big\{f\in C(X\times Y \times Z;V);\\
& \|f(x,y,z)\|_{V}\leq \kappa(\|(x,y,z)\|_{\mathcal {Z}^{0,\infty}})(1+\|x\|_{X}+\|y\|_{Y}+\|z\|_{Z} ),\\
&  \mbox{for all}~\forall x\in X,y\in Y,z\in Z\Big\},
 \end{split}
\end{equation*}
where $\kappa:\mathbb{R}^+\mapsto [1,\infty)$ is a locally bounded nondecreasing function. We also denote the space of locally Lipchitz functionals by
\begin{equation*}
\begin{split}
\mbox{Lip}_{loc}(X\times Y \times Z;V):=&\Bigg\{f\in \mbox{Bnd}_{loc}(X\times Y \times Z;V); ~~ \|f(x_1,y_1,z_1)-f(x_2,y_2,z_2)\|_{V}\\
& \leq \varsigma\Bigg(\sum_{i=1,2}\|(x_i,y_i,z_i)\|_{\mathcal {Z}^{0,\infty}}\Bigg)
(\|x_1-x_2\|_{X}+\|y_1-y_2\|_{Y}+\|z_1-z_2\|_{Z}),\\
&~ \mbox{for all}~ (x_i,y_i,z_i)\in X\times Y\times Z,~~i=1,2\Bigg\},
 \end{split}
\end{equation*}
where $\varsigma:\mathbb{R}^+\mapsto\mathbb{R}^+$ is a locally bounded and nondecreasing function.

With above nations, we provided the following assumptions for the stochastic forcing:
\begin{itemize} [leftmargin=1.42cm]
\item [\textbf{(A1)}]  For any fixed $p\geq2$ and an integer $k>1+\frac{2}{p}$, we assume that
$$
\sigma_i(\cdot)\in \mbox{Lip}_{loc}(\mathcal {Z}^{k-1,p}(D); \mathbb{W}^{k-1,p} )\cap \mbox{Lip}_{loc}(\mathcal {Z}^{k,p}(D); \mathbb{W}^{k,p} ),\quad i=1,2,3.
$$
and
$$
\sigma_i(\cdot)\in \mbox{Bnd}_{loc}(\mathcal {Z}^{k+1,p}(D);\mathbb{W}^{k+1,p} ),\quad i=1,2,3.
$$

\item [\textbf{(A2)}] In addition, in order to justify the construction of Galerkin approximations in sufficiently regular spaces, we assume that
$$
\sigma_i(\cdot)\in \mbox{Bnd}_{loc}(\mathcal {Z}^{k+4,p}(D);\mathbb{W}^{k+4,p} ),\quad i=1,2,3.
$$
\end{itemize}
 By virtue of the Sobolev embedding theorem, we infer that $\mathcal {Z}^{k+2,2}(D)$ is densely embedded into $ \mathcal {Z}^{k+1,p}(D)$, which plays an important role in constructing the local pathwise solutions in $\mathcal {Z}^{k,p}(D)$ via a density and stability argument (cf. Lemma \ref{lem:4.1}).

\subsection{Main results}

To formulate the main results in the paper, let us first provide the definition of pathwise solutions to the SISM \eqref{1.1}-\eqref{1.3}.
\begin{definition} \label{def:1.1}
Fix a stochastic basis $\mathcal {S}:= (\Omega,\mathcal {F},\mathbb{P},\{\mathcal {F}_t\}_{t\geq0},\mathcal {W})$. Assume that $k>1+\frac{2}{p}$, $p\geq2$, the triple $(u_S^0,u_T^0,\theta_S^0)$ are $\mathcal {Z}^{k,p}(D)$-valued $\mathcal {F}_0$-measurable random variables, and  $\sigma_i (\cdot)(i=1,2,3)$ satisfies the conditions (A1)-(A2).

(1) A quadruple $(u_S,u_T,\theta_S,\mathbbm{t})$ is called a \emph{local pathwise solution} provided
\begin{itemize}[leftmargin=1.45cm]
\item [$\bullet$] $\mathbbm{t}$ is a strictly positive stopping time almost surely, i.e., $\mathbb{P}\{\mathbbm{t}>0\}=1$.

\item [$\bullet$] $(u_S,u_T,\theta_S):D\times \mathbb{R}^+\mapsto X^{k,p}(D;\mathbb{R}^2)\times (W^{k,p}(D;\mathbb{R}))^2$ is a $\mathcal {F}_t$-progressively measurable process satisfying
    $$
    u_S(\cdot\wedge\mathbbm{t})\in C([0,\infty);X^{k,p}(D;\mathbb{R}^2)),\quad \mathbb{P}\mbox{-a.s.},
    $$
    $$
    u_T(\cdot\wedge\mathbbm{t}),~~~\theta_S(\cdot\wedge\mathbbm{t})\in C([0,\infty);W^{k,p}(D;\mathbb{R})),\quad \mathbb{P}\mbox{-a.s.};
    $$

\item [$\bullet$] for all $t\geq0$, there holds   $\mathbb{P}$-a.s.
\begin{equation*}
\begin{split}
\quad \quad\quad u_S(t\wedge\mathbbm{t})=& u_S^0-\int_0^{t\wedge\mathbbm{t}} \Big(P(u_S\cdot \nabla)   u_S+P(f u_T\widehat{x})-P(\frac{g}{\theta_0}\theta _S\widehat{z})\Big) \mathrm{d}r \\
      &+ \int_0^{t\wedge\mathbbm{t}}P\sigma_1(u_S,u_T,\theta_S)\mathrm{d}\mathcal {W},\\
u_T(t\wedge\mathbbm{t}) =& u_T^0-\int_0^{t\wedge\mathbbm{t}} \Big((u_S\cdot \nabla)   u_T + fu_S\cdot\widehat{x} - \frac{g}{\theta_0} zs\Big)\mathrm{d}r + \int_0^{t\wedge\mathbbm{t}}\sigma_2(u_S,u_T,\theta_S)\mathrm{d}\mathcal {W},\\
\theta_S(t\wedge\mathbbm{t})= & \theta_S^0-\int_0^{t\wedge\mathbbm{t}} \Big((u_S\cdot \nabla)  \theta_S + u_Ts\Big) \mathrm{d}r+ \int_0^{t\wedge\mathbbm{t}}\sigma_3(u_S,u_T,\theta_S)\mathrm{d}\mathcal {W}.
\end{split}
\end{equation*}
\end{itemize}

(2) The \emph{pathwise uniqueness} of the solution holds in the following sense: if $(u_S,u_T,\theta_S,\mathbbm{t})$, $(\tilde{u}_S,\tilde{u}_T,\tilde{\theta}_S,\tilde{\mathbbm{t}})$ are two local pathwise solutions with
    $$
    \mathbb{P}\{(u_S^0,u_T^0,\theta_S^0)
    =(\tilde{u}_S^0,\tilde{u}_T^0,\tilde{\theta}_S^0)\}=1,
    $$
then
    $$
    \mathbb{P}\{ (u_S(t),u_T(t),\theta_S(t))
    =(\tilde{u}_S(t),\tilde{u}_T(t),\tilde{\theta}_S(t)) ;~\forall t\in [0,\mathbbm{t}\wedge \tilde{\mathbbm{t}}]\}=1.
    $$

(3) A quintuple $(u_S,u_T,\theta_S, \{\mathbbm{t}_n\}_{n\geq1},\mathbbm{t})$ is called a \emph{maximal pathwise solution} to the system \eqref{1.1}-\eqref{1.3}, if each quadruple $(u_S,u_T,\theta_S ,\mathbbm{t}_n)$ is a local pathwise solution, $\mathbbm{t}^n$ increases with $\lim_{n\rightarrow\infty}\mathbbm{t}_n=\mathbbm{t}$ such that
$$
\lim_{n\rightarrow\infty}\sup_{t\in[0,\mathbbm{t}_n]}\|(u_S,u_T,\theta_S)\|_{\mathcal {Z}^{1,\infty}}=\infty \quad \mbox{on the set}~\{\mathbbm{t}<\infty\}.
$$
Moreover, if $\mathbbm{t}=\infty$ almost surely, then the solution is said to be \emph{global}.
\end{definition}

The first main result concerning  local well-posedness for the Cauchy problem \eqref{1.1}-\eqref{1.3} may now be enunciated by the following theorem.

\begin{theorem} [Local-in-time solutions] \label{thm:1.2}
 Assume that  $k>1+\frac{2}{p}$, $p\geq2$, $(u_S^0,u_T^0,\theta_S^0)$ is a  $\mathcal {Z}^{k,p}(D)$-valued $\mathcal {F}_0$ measurable random variables, and the diffusions $\sigma_i(\cdot)$, $i=1,2,3$  satisfy the conditions (A1)-(A2). Then the Cauchy problem \eqref{1.1}-\eqref{1.3} has a unique maximum pathwise solution $(u_S,u_T,\theta_S, \{\mathbbm{t}_n\}_{n\geq1},\mathbbm{t})$ in the sense of Definitions \ref{def:1.1}.
\end{theorem}

As we alluded to earlier, the stochastic forcing have a bad influence on the time regularity of solutions, which makes the a priori energy estimates and standard Arzel\`{a}-Ascoli type compactness results applied in \cite{1} to be invalid in present case. To overcome this, we first establish the existence of martingale solutions to the SISM following the approach in \cite{43}. However, the a priori estimates  (cf. Lemma \ref{lem:2.2}) hold only up to a stopping time, which allows us to derive some unform estimates for approximate solutions in the Galerkin scheme only up to a sequence of stopping times $\{\mathbbm{t}_n\}_{n\geq1}$. To our knowledge, there has no effective method to derive an uniform bound for the stopping times $\{\mathbbm{t}_n\}_{n\geq1}$ from below. We shall deal with this problem by properly introducing some new cut-off functions depending on the $W^{1,\infty}$-norm of approximations for the nonlinear terms in the system. Due to the appearance of the cut-off functions and the excessive number of derivatives in estimations, we have to first establish the (global) pathwise solutions in sufficiently regular Hilbert spaces $\mathcal {Z}^{k+2,2}(D)$, which is dense in the sharp spaces $\mathcal {Z}^{k,p}(D)$ with $k> 1+\frac{1}{p}$ and $p\geq2$. Here, we construct the pathwise solution in spirit of the classical Yamada-Watanabe theorem \cite{49,50}, which tells us that pathwise solutions exist whenever martingale solutions may be found, and pathwise uniqueness holds. Recently, a different proof of such result was developed in [23] which leans on an elementary characterization of convergence in probability (cf. Lemma 2.10 below). Due to the continuously embedding  $\mathcal {Z}^{k+2,2}(D)\hookrightarrow\mathcal {Z}^{k+1,p}(D)$, one can finally apply a density-stability argument to prove the existence of maximum pathwise solution in the shape cases. It should be pointed out that the coupling nature of the system and the $L^p$ estimates in the working space  $\mathcal {Z}^{k,p}(D)$ make the derivation of a serious of uniform bounds for approximations to be very complicated.

Our second goal in this paper is to seek for sufficient conditions which lead to global solutions. In this work, we focus on the case that the potential temperature do not vary linearly on the $y$-direction, i.e., $s=0$ (cf. \cite{1}), and consider the ISM with linear multiplicative stochastic forcing:
\begin{equation}\label{1.8}
\begin{split}
&\sigma_1(u_S,u_T,\theta_S) \mathrm{d} \mathcal {W}= \alpha u_S \mathrm{d} W, \\
& \sigma_2(u_S,u_T,\theta_S)  \mathrm{d} \mathcal {W}= \alpha u_T\mathrm{d} W,\\
&  \sigma_3(u_S,u_T,\theta_S)  \mathrm{d} \mathcal {W}= \alpha \theta_S\mathrm{d} W,
 \end{split}
\end{equation}
where $\alpha\in\mathbb{R}$ denotes the diffusion parameter, and $W$ is a single one-dimension Brownian motion. Our main result in this topic can be described by the following theorem.

\begin{theorem} [Global-in-time solutions] \label{thm:1.3}
Assume that $s=0$ in \eqref{1.1}, $k> 1+\frac{2}{p}$, $p\geq2$, $r\geq1$, and  $(u_S^0,u_T^0,\theta_S^0)$ is a $\mathcal {Z}^{k,p}(D)$-valued $\mathcal {F}_0$ measurable random variable. For $\alpha \in \mathbb{R}$, let $(u_S,u_T,\theta_S, \mathbbm{t})$
be the maximum pathwise solution to the SISM with linear multiplicative noises \eqref{1.8}.
If the initial data satisfies
\begin{equation}\label{1.9}
\begin{split}
\|(u_S^0, u_T^0, \theta_S^0) \|_{\mathcal {Z}^{k,p}}\leq \tilde{A}(|\alpha|,r),
 \end{split}
\end{equation}
for sufficiently large $|\alpha|>1 $, then the local pathwise solution exists globally in time with high probability. More precisely, we have
\begin{equation*}
\begin{split}
\mathbb{P}\{\mathbbm{t}=\infty\} \geq 1- r^{-\frac{1}{16}},
 \end{split}
\end{equation*}
where the upper bound $\tilde{A}(|\alpha|,r)$ in \eqref{1.9} depends only on the parameters $\alpha$ and $r$ such that, for any fixed $r\geq 1$,
$$
\lim_{|\alpha|\rightarrow \infty}\tilde{A}(|\alpha|,r)=\infty, \quad \lim_{|\alpha|\rightarrow0} \tilde{A}(|\alpha|,r)=0.
$$
\end{theorem}

Note that the forcing regime \eqref{1.8} is covered by the theory developed in Theorem \ref{thm:1.2}. Theorem \ref{thm:1.3} implies that the linear multiplicative noise in the type of \eqref{1.8} have a regularization effect at the pathwise level,
so that the local solutions to the ISM with large noise are global in time with high probability $1- r^{-1/16}$. The advantage of the noises \eqref{1.8} lies in the fact that one can reformulated  the SISM as random partial differential equations (PDEs) by using an exponential transformation, which provides a damping effect on the pathwise behavior of solutions. Due to the complex coupling nature of the system, it
turns out that the dynamic behavior of solutions to the SISM \eqref{5.1} is quite different with the 2D stochastic Euler equations \cite{41}, and hence it is challenging to prove the existence of global pathwise solution
in full probability, i.e., $\mathbbm{t}=\infty$ almost surely. Fortunately, for sufficiently large $|\alpha|>1$, one can prove by carefully introducing suitable stopping times that the vorticity term
$\omega= \mbox{curl}~u_S$ and the gradients for $u_T$ and $\theta_S$ on the R.H.S. of \eqref{5.9} can be controlled by the damping terms arising from the stochastic forcing. By applying the Beale-Kato-Majda
inequality, one can derive a bound for $\|(u_S, u_T, \theta_S)\|_{\mathcal {Z}^{k,p}}$ within a stochastic interval depending on $r$, provided the initial data
 $\|(u_S^0, u_T^0, \theta_S^0)\|_{\mathcal {Z}^{k,p}}$ is sufficiently small (which is determined by the parameters $\alpha$ and $r$). Moreover, it reveals that the global existence of
 solutions is closely related to the growth of a geometric Brownian motion $\exp\{\alpha W_t-\frac{\alpha^2}{32}t\}$. Finally, we shall prove that the solution is global in
 time on the event that the geometric Brownian motion stays always below the value $r$.

\begin{remark} Due to the complex coupling nature of system, the global existence problem with the linear multiplicative noises \eqref{1.8} in the case of $s\neq0$ becomes more challenging. Also, whether the ISM perturbed by additive noise admits global-in-time pathwise solution is still an open question. Moreover,
the influence of the viscosity on dynamic behavior of solutions to the SISM system is also a natural and interesting problem.
\end{remark}

The rest of the paper is organized as follows. In Section 2, we prove the existence of global martingale solutions to the truncated SISM with nonlinear multiplicative noise.
To this end, we first derive some a priori uniform bounds in Banach spaces $\mathcal {Z}^{k,p}(D)$ by introducing suitable stopping times. Then we introduce the Galerkin scheme and establish
the existence of martingale solutions by stochastic compactness method. In Section 3, we prove the pathwise uniqueness of solutions which together with the existence result lead to the existence of pathwise solution in smooth Hilbert
space $\mathcal {Z}^{k+4,2}(D)$. In Section 4, we establish the existence of solution in sharp case $\mathcal {Z}^{k,p}(D)$ with $k>1+1/p$ and $p\geq2$ by applying an abstract Cauchy theorem.
The Section 5 is devoted to the proof of Theorem \ref{thm:1.3} for the ISM perturbed by linear multiplicative stochastic forcing.

\section{Martingale solutions for truncated SISM}

The main task of this section is to prove the existence of global martingale solutions (probability weak solutions) for the truncated SISM  in sufficiently regular  spaces $\mathcal {Z}^{k',2}(D)$ with $k'=k+4$, $k\geq 1+\frac{2}{p}$ and  $p\geq2$. More precisely, the main result in this section can be formulated by the following lemma, whose proof will be provided in Subsection 2.3.

\begin{lemma} [Martingale solutions] \label{lem:2.1}
Fix $k'=k+4$, $k>1+\frac{2}{p}$, $\vartheta\in (0,\frac{1}{2})$. Assume that $\mu_0\in \mathcal {P}_r(\mathcal {Z}^{k',2}(D))$ is an initial  measure such that $\int_{\mathcal {Z}^{k',2}(D)}\|u\|_{\mathcal {Z}^{k',2}}\mu_0( d u)<\infty$. Then there exists a stochastic basis $\mathcal {S}=(\overline{\Omega},\overline{\mathcal {F}},\overline{\mathbb{P}},\{\overline{\mathcal {F}}_t\},\overline{\mathcal {W}})$ and a triple  $\mathcal {Z}^{k',2}(D)$ valued predictable process
\begin{equation*}
\begin{split}
(\overline{u}_S,\overline{u}_T,\overline{\theta}_S)\in L^q(\Omega; L^\infty_{loc}([0,\infty);\mathcal {Z}^{k',2}(D))) \bigcap L^q(\Omega; C([0,\infty);\mathcal {Z}^{k'-2,2}(D)))
 \end{split}
\end{equation*}
with $\mathbb{P}\circ(u_S^0,u_T^0,\theta_S^0)^{-1}
=\overline{\mathbb{P}}\circ(\overline{u}_S(0),\overline{u}_T(0),\overline{\theta}_S(0))^{-1}$ such that
\begin{equation}
\left\{
\begin{aligned}
& \overline{u}_S(t)+  \int_0^t\varpi_R(\|\overline{u}_S\|_{X^{1,\infty}})P(\overline{u}_S\cdot \nabla)   \overline{u}_S\mathrm{d}s-\int_0^tP(  \overline{u}_T\widehat{x}) \mathrm{d}s- \int_0^tP( \overline{\theta} _S\widehat{z}) \mathrm{d}s\\
&\quad\quad=  \overline{u}_S(0)+\int_0^t\mathcal {D}_{cut}^1(s)\mathrm{d}\mathcal {W}_s,\\
& \overline{u}_T(t)+ \int_0^t\varpi_R(\|(\overline{u}_S,\overline{u}_T)\|_{X^{1,\infty}\times W^{1,\infty}})(\overline{u}_S\cdot \nabla)   u_T\mathrm{d}s+ \int_0^tu_S\cdot\widehat{x} \mathrm{d}s +   zt \\
&\quad\quad=\overline{u}_T(0)+\int_0^t\mathcal {D}_{cut}^2(s)\mathrm{d}\mathcal {W}_s, \\
& \overline{\theta}_S(t)+\int_0^t\varpi_R(\|(\overline{u}_S,\overline{\theta}_S)\|_{X^{1,\infty}\times W^{1,\infty}})(\overline{u}_S\cdot \nabla ) \overline{\theta}_S\mathrm{d}s+\int_0^t\overline{u}_T \mathrm{d}s = \overline{\theta}_S(0)+\int_0^t\mathcal {D}_{cut}^3(s)\mathrm{d}\mathcal {W}_s,
\end{aligned}
\right.
\end{equation}
for all $t>0$, where the truncated diffusions are defined by
\begin{equation*}
\begin{split}
\mathcal {D}_{cut}^1(t)=\varpi_R(\|(u_S, u_T,\theta_S)\|_{\mathcal {Z}^{1,\infty}})
P\sigma_1(u_S,u_T,\theta_S),
 \end{split}
\end{equation*}
and
\begin{equation*}
\begin{split}
\mathcal {D}_{cut}^j(t)=\varpi_R(\|(u_S, u_T,\theta_S)\|_{\mathcal {Z}^{1,\infty}})
\sigma_j(u_S,u_T,\theta_S),\quad j=2,3.
 \end{split}
\end{equation*}
\end{lemma}

We shall construct the martingale solutions to above truncated system following the method in Chapter 8 of \cite{43}, where the main mathematical tools are the Prokhorov theorem and the Skorohod's representation theorem. Here and in the following, to simplify the exposition, we assume the parameters $f=g=\theta_0=s=1$. The general case is entirely similar.

\subsection{A priori estimates}

We shall derive some a priori estimates for solutions evolving in $\mathcal {Z}^{k,p}(D)$ with $k>1+\frac{2}{p}$ and $p\geq2$ by utilizing the stopping times techniques, which play an important role in  estimating approximate solutions in the next sections.

\begin{lemma} \label{lem:2.2}
Let $k\geq 1+\frac{2}{p}$,  $p\geq2$, and assume that  $(u_S,u_T,\theta_S)$ is a solution to the SISM system \eqref{1.1}-\eqref{1.3} with initial data $(u_S^0,u_T^0,\theta_S^0)$,
which is defined up to a maximum existence time $\mathbbm{t}>0$ almost surely. Then for any fixed $T>0$,
\begin{equation} \label{2.1}
\begin{split}
\mathbb{E}\left(\sup_{t\in[0, \mathbbm{t}_R \wedge T]} \|(u_S ,u_T,\theta_S)(t)\|_{\mathcal {Z}^{k,p}}^p\right) \leq Ce^{C(1+R+\kappa(R)^2)}\mathbb{E} \left(\|(u_S^0 ,u_T^0,\theta_S^0)\|_{\mathcal {Z}^{k,p}}^p\right),
 \end{split}
\end{equation}
with
\begin{equation} \label{2.2}
\begin{split}\mathbbm{t}_R:=\inf\left\{t\geq0;~ \|(u_S,u_T,\theta_S)(t)\|_{\mathcal {Z}^{1,\infty}}\geq R\right\},\quad \forall R>0,
\end{split}
\end{equation}
where the positive constant $C$ depends only on $k,T,\kappa$ and the domain $D$.
\end{lemma}

\begin{proof}
The proof of Lemma \ref{lem:2.2} will be divided into two parts.  In the first part we establish the a priori estimates in the Hilbert spaces $\mathcal {Z}^{k,2}(D)$. In the second part, we extend these a priori estimates to the sharp Banach spaces $\mathcal {Z}^{k,p}(D)$ with $k\geq 1+\frac{2}{p}$, $p\geq 2$.

Let $\alpha=(\alpha_1,\alpha_2)\in \mathbb{N}^+$ be the multi-index such that $|\alpha|=|\alpha_1|+|\alpha_2|\leq k$, by applying the It\^{o}'s formula to
 $\mathrm{d}\|u_S\|_{X^{k,2}}^2=\Sigma_{|\alpha|\leq k}\mathrm{d} \|\partial^\alpha u_S\|_{L^2}^2$, we get that for each multi-index $\alpha$
\begin{equation*}
\begin{split}
\mathrm{d} \|\partial^\alpha u_S\|_{L^2}^2=& -2( \partial^\alpha  u_S, \partial^\alpha [P(u_S\cdot \nabla)   u_S] )_{L^2}\mathrm{d}t +2( \partial^\alpha  u_S,  \partial^\alpha  P(u_T\widehat{x}) )_{L^2}\mathrm{d}t\\
&+2( \partial^\alpha  u_S,  \partial^\alpha  P(\theta _S\widehat{z}) )_{L^2}\mathrm{d}t  +  \| \partial^\alpha  P\sigma_1(u_S,u_T,\theta_S)\|_{\mathbb{X}^{0,2}}\mathrm{d}t\\
&+2( \partial^\alpha  u_S, \partial^\alpha P\sigma_1(u_S,u_T,\theta_S))_{L^2}\mathrm{d}\mathcal {W}.
 \end{split}
\end{equation*}
For any $T>0$ and stopping time $\tau\leq \mathbbm{t}\wedge T$, after integrating the above equality on $[0,t]$ with $t\in [0,\tau]$ and summering up the resulted equations with respect to $\alpha$, we find
\begin{equation} \label{2.3}
\begin{split}
 \|u_S(t)\|_{X^{k,2}}^2\leq& \|u_S^0\|_{X^{k,2}}^2 +2\sum_{|\alpha|\leq k}\int_0^t|( \partial^\alpha  u_S, \partial^\alpha [P(u_S\cdot \nabla)   u_S] )_{L^2}|\mathrm{d}s \\
&+2\sum_{|\alpha|\leq k}\bigg(\int_0^t|( \partial^\alpha  u_S,  \partial^\alpha  P(u_T\widehat{x}) )_{L^2}|\mathrm{d}s+ \int_0^t|( \partial^\alpha  u_S,  \partial^\alpha  P(\theta _S\widehat{z}) )_{L^2}|\mathrm{d}s\bigg)\\
&  + \sum_{|\alpha|\leq k}\int_0^t \| \partial^\alpha  P\sigma_1(u_S,u_T,\theta_S)\|_{\mathbb{X}^{0,2}}^2\mathrm{d}s\\
&+2\sum_{|\alpha|\leq k}\bigg|\int_0^t( \partial^\alpha  u_S,  \partial^\alpha  P\sigma_1(u_S,u_T,\theta_S))_{L^2}\mathrm{d}\mathcal {W}\bigg|\\
:=&\|u_S^0\|_{X^{k,2}}^2+\sum_{|\alpha|\leq k}(I_1^\alpha(t)+I_2^\alpha(t)+I_3^\alpha(t)+I_4^\alpha(t)).
 \end{split}
\end{equation}
Let us estimates the integral terms on the right hand side of \eqref{2.3} one by one. For $I_1^\alpha(t)$, we decompose the convection term as follows
\begin{equation}\label{2.4}
\begin{split}
 \partial^\alpha [P(u_S\cdot \nabla)   u_S]
= P([ \partial^\alpha ,u_S]\cdot \nabla u_S)+(u_S\cdot\nabla) \partial^\alpha  u_S-   (I-P)(u_S\cdot\nabla)  \partial^\alpha  u_S,
 \end{split}
\end{equation}
where $[ \partial^\alpha ,u_S]= \partial^\alpha u_S-u_S \partial^\alpha $ denotes the commutator. Direct calculation shows that
$\div (( \partial^\alpha  u_S\cdot \nabla)   u_S-( u_S\cdot \nabla)  \partial^\alpha u_S)=0$,
and thereby the uniqueness of solution to the Neumann problem \eqref{1.4} implies that
\begin{equation}\label{2.5}
\begin{split}
(I-P)( u_S\cdot \nabla)  \partial^\alpha   u_S=(I-P)( \partial^\alpha  u_S\cdot \nabla)   u_S.
 \end{split}
\end{equation}
Integration by parts leads to the cancelation property $( \partial^\alpha  u_S,(u_S \cdot \nabla) \partial^\alpha  u_S)_{L^2}$=0 since div$u_S=0$. By using the H\"{o}lder inequality, one can estimate $I_1^\alpha(t)$ as
\begin{equation}\label{2.6}
\begin{split}
I_1^\alpha(t)\leq&2\int_0^t |( \partial^\alpha  u_S, P([ \partial^\alpha ,u_S]\cdot \nabla u_S))_{L^2}|\mathrm{d}s+2\int_0^t|( \partial^\alpha  u_S,(u_S\cdot\nabla) \partial^\alpha  u_S)_{L^2}|\mathrm{d}s\\
&+2\int_0^t |( \partial^\alpha  u_S,(I-P)(u_S\cdot\nabla)  \partial^\alpha  u_S)_{L^2}|\mathrm{d}s    \\
\leq  &C\int_0^t\| u_S\|_{X^{k,2}} \|[ \partial^\alpha ,u_S]\cdot\nabla  u_S\|_{L^2} \mathrm{d}s
+ C\int_0^t\| u_S\|_{X^{k,2}}\|( \partial^\alpha  u_S\cdot \nabla)   u_S\|_{L^2}\mathrm{d}s \\
\leq  &C\int_0^t\| u_S\|_{X^{k,2}} (\|\nabla  u_S\|_{L^\infty}\|\nabla u_S\|_{X^{|\alpha|-1,2}}
+ \| u_S\|_{X^{|\alpha|,2}}\|\nabla u_S\|_{L^\infty}) \mathrm{d}s  \\
&+ C\int_0^t\| u_S\|_{X^{k,2}}\| \partial^\alpha  u_S\|_{L^2}\|\nabla u_S\|_{L^\infty}\mathrm{d}s \\
\leq & C\int_0^t\|\nabla  u_S\|_{L^\infty}\| u_S\|_{X^{k,2}}^2\mathrm{d}s ,
 \end{split}
\end{equation}
where  the last inequality used the following fundamental commutator estimate (cf. \cite{41}), which will be frequently used in the paper:
\begin{equation}\label{2.7}
\begin{split}
\sum_{0\leq|\alpha|\leq k}\|[\partial^\alpha,f]\cdot \nabla g\|_{L^p}\leq C(\|\nabla  f\|_{L^\infty}\| g\|_{W^{k,p} }+ \|f\|_{W^{k,p}}\| \nabla g\|_{L^\infty}),\quad 1<p<\infty,~~k\in \mathbb{N}.
 \end{split}
\end{equation}
For $I_2^\alpha(t)$, since $|\widehat{x}|=|\widehat{z}|=1$, it follows from the Cauchy-Schwartz inequality that
\begin{equation}\label{2.8}
\begin{split}
 I_2^\alpha(t) &\leq  C\int_0^t\| u_S\|_{X^{k,2}}(\| u_T\|_{W^{k,2}}+\| \theta_S\|_{W^{k,2}})\mathrm{d}s\\
 &\leq C\int_0^t(\| u_S\|_{X^{k,2}}^2+\| u_T\|_{W^{k,2}}^2+\| \theta_S\|_{W^{k,2}}^2)\mathrm{d}s.
 \end{split}
\end{equation}
In view of the condition (A1), the term $I_4^\alpha(t)$ can be estimated as
\begin{equation}\label{2.9}
\begin{split}
I_3^\alpha(t)\leq C\int_0^t \kappa(\|u_S,u_T,\theta_S\|_{L^\infty})^2(1+\|u_S\|_{X^{k,2}}^2+\|u_T\|_{W^{k,2}}^2+\|\theta_S\|_{W^{k,2}}^2 ) \mathrm{d}s.
 \end{split}
\end{equation}
For the stochastic term $I_4^\alpha(t)$, by using the BDG inequality (cf. \eqref{1.7}) and the condition (A1), we deduce that
\begin{equation}\label{2.10}
\begin{split}
\mathbb{E}\sup_{t\in[0,\tau]}|I_4^\alpha(t)|\leq &\mathbb{E} \left(\int_0^\tau\| u_S\|_{X^{k,2}}^2\|P\sigma_1(u_S,u_T,\theta_S)\|_{\mathbb{X}^{k,2}}^2\mathrm{d}t \right)^{1/2}\\
\leq &\mathbb{E} \bigg(\int_0^\tau\kappa(\|u_S,u_T,\theta_S\|_{L^\infty})^2\\
&\times\| u_S\|_{X^{k,2}}^2(1+\|u_S\|_{X^{k,2}}^2+\|u_T\|_{W^{k,2}}^2+\|\theta_S\|_{W^{k,2}}^2 )\mathrm{d}t \bigg)^{1/2}\\
\leq &\mathbb{E}\Bigg(\sup_{t\in[0,\tau]}\| u_S(t)\|_{X^{k,2}}\\
&   \times\int_0^\tau\kappa(\|u_S,u_T,\theta_S\|_{L^\infty})^2 (1+\|u_S\|_{X^{k,2}}^2+\|u_T\|_{W^{k,2}}^2+\|\theta_S\|_{W^{k,2}}^2 )\mathrm{d}t \Bigg)^{1/2} \\
\leq &\frac{1}{2}\mathbb{E}\sup_{t\in[0,\tau]}\| u_S(t)\|_{X^{k,2}}^2+C\mathbb{E} \int_0^t\kappa(\|u_S,u_T,\theta_S\|_{L^\infty})^2\\
&  \times \left(1+\|u_S\|_{X^{k,2}}^2+\|u_T\|_{W^{k,2}}^2+\|\theta_S\|_{W^{k,2}}^2\right )\mathrm{d}t  .
 \end{split}
\end{equation}
Combining the estimates \eqref{2.6}-\eqref{2.10}, taking a supremum over $t\in [0,\tau]$ on both sides of \eqref{2.3} and then taking the expected value, we deduce  that
\begin{equation}\label{2.11}
\begin{split}
 \mathbb{E}\sup_{t\in[0,\tau]}\|u_S(t)\|_{X^{k,2}}^2\leq& 2\mathbb{E}\|u_S^0\|_{X^{k,2}}^2 +C\mathbb{E} \int_0^t\left(1+\|\nabla  u_S\|_{L^\infty}+\kappa(\|u_S,u_T,\theta_S\|_{L^\infty})^2\right)\\
&  \times\left(1+\|u_S\|_{X^{k,2}}^2+\|u_T\|_{W^{k,2}}^2+\|\theta_S\|_{W^{k,2}}^2 \right)\mathrm{d}t.
 \end{split}
\end{equation}
To estimate $\|u_T(t)\|_{W^{k,2}}$, we apply the Ito's formula to the second equation of \eqref{1.5} and obtain
\begin{equation}\label{2.12}
\begin{split}
 \|u_T(t)\|_{W^{k,2}}^2 =&\|u_T^0\|_{W^{k,2}}^2-2\sum_{|\alpha|\leq k}\int_0^t( \partial^\alpha u_T,  \partial^\alpha (u_S\cdot \nabla)   u_T )_{L^2} \mathrm{d}s \\
 &-2\sum_{|\alpha|\leq k}\int_0^t( \partial^\alpha u_T,  \partial^\alpha (u_S\cdot\widehat{x}) )_{L^2}\mathrm{d}t-2\sum_{|\alpha|\leq k}\int_0^t( \partial^\alpha u_T,   \partial^\alpha  z )_{L^2}\mathrm{d}s\\
&+ \sum_{|\alpha|\leq k}\int_0^t\| \partial^\alpha \sigma_2(u_S,u_T,\theta_S)\|_{\mathbb{W}^{0,2}}^2\mathrm{d}s\\
&+2\sum_{|\alpha|\leq k}\int_0^t( \partial^\alpha u_T,  \partial^\alpha \sigma_2(u_S,u_T,\theta_S))_{L^2}\mathrm{d}\mathcal {W}_s\\
:=& \|u_T^0\|_{W^{k,2}}^2+\sum_{|\alpha|\leq k}(J_1^\alpha(t)+J_2^\alpha(t)+J_3^\alpha(t)+J_4^\alpha(t)+J_5^\alpha(t)).
 \end{split}
\end{equation}
For the second terms involved on the R.H.S of \eqref{2.12}, it follows from the cancelation property and the commutator estimate that
\begin{equation}\label{2.13}
\begin{split}
J_1^\alpha(t)&=-2\int_0^t( \partial^\alpha u_T, [ \partial^\alpha ,u_S]\cdot \nabla u_T )_{L^2} \mathrm{d}s-2\int_0^t( \partial^\alpha u_T, u_S\cdot\nabla \partial^\alpha    u_T )_{L^2} \mathrm{d}s\\
&\leq \int_0^t\| \partial^\alpha u_T\|_{L^2} \|[ \partial^\alpha ,u_S]\cdot \nabla u_T \|_{L^2} \mathrm{d}s\\
&\leq \int_0^t\|u_T\|_{W^{k,2}}  \left(\|\nabla u_S\|_{L^\infty}\| u_T\|_{W^{|\alpha|,2}} + \| \partial^\alpha u_S\|_{L^2}\|\nabla u_T\|_{L^\infty} \right) \mathrm{d}s\\
&\leq \int_0^t  (\|\nabla u_S\|_{L^\infty}+\|\nabla u_T\|_{L^\infty})\left(\|u_T\|_{W^{k,2}}^2 + \| u_S\|_{W^{k,2}}^2 \right) \mathrm{d}s.
 \end{split}
\end{equation}
Since $|\widehat{x}|=1$ and the domain $D\subset \mathbb{R}^2$ is bounded, the third and forth terms on  the R.H.S of \eqref{2.12} can be estimated as
\begin{equation}\label{2.14}
\begin{split}
J_2^\alpha(t)+J_3^\alpha(t)\leq  2\int_0^t|( \partial^\alpha u_T,  \partial^\alpha (u_S\cdot\widehat{x}) )_{L^2}|\mathrm{d}t\leq C \int_0^t\|u_S\|_{X^{k,2}}\|u_T\|_{W^{k,2}}\mathrm{d}t.
 \end{split}
\end{equation}
By condition (A1), we have
\begin{equation}\label{2.15}
\begin{split}
J_4^\alpha(t)\leq C\int_0^t \kappa(\|u_S,u_T,\theta_S\|_{L^\infty})^2(1+\|u_S\|_{X^{k,2}}^2+\|u_T\|_{W^{k,2}}^2+\|\theta_S\|_{W^{k,2}}^2 ) \mathrm{d}s.
 \end{split}
\end{equation}
For $J_5^\alpha(t)$, after taking the supremum on $[0,\tau]$ for any stopping time $\tau\leq \mathbbm{t}\wedge T$, we deduce from the BDG inequality and the Young inequality that
\begin{equation}\label{2.16}
\begin{split}
\mathbb{E}\sup_{t\in[0,\tau]}|J_5^\alpha(t)|\leq& C\mathbb{E} \Bigg(\sup_{t\in [0,\tau]}\| u_T(t)\|_{W^{k,2}}^2\int_0^\tau\| \sigma_2(u_S,u_T,\theta_S)\|_{\mathbb{W}^{k,2}}^2\mathrm{d}t \Bigg)^{1/2}\\
\leq& \frac{1}{2}\mathbb{E} \sup_{t\in [0,\tau]}\| u_T(t)\|_{W^{k,2}}^2 \\
&+C\mathbb{E}\int_0^t \kappa(\|u_S,u_T,\theta_S\|_{L^\infty})^2\left(1+\|u_S\|_{X^{k,2}}^2+\|u_T\|_{W^{k,2}}^2
+\|\theta_S\|_{W^{k,2}}^2 \right) \mathrm{d}s.
 \end{split}
\end{equation}
Plugging the estimates \eqref{2.13}-\eqref{2.16} into \eqref{2.12}, we get that for any stopping time $\tau\leq \mathbbm{t}\wedge T$
\begin{equation}\label{2.17}
\begin{split}
&\mathbb{E}\sup_{t\in[0,\tau]} \|u_T(t)\|_{W^{k,2}}^2 \leq  2\mathbb{E}\|u_T^0\|_{W^{k,2}}^2\\
&\quad +C\mathbb{E}\int_0^t \left(1+\|\nabla u_S\|_{L^\infty}+\|\nabla u_T\|_{L^\infty}+\kappa(\|u_S,u_T,\theta_S\|_{L^\infty})^2 \right)\\
&\quad \times \left(1+\|u_S\|_{X^{k,2}}^2+\|u_T\|_{W^{k,2}}^2+\|\theta_S\|_{W^{k,2}}^2\right ) \mathrm{d}s.
 \end{split}
\end{equation}
To estimate the norm $\|\theta_S(t)\|_{W^{k,2}}$ with respect to the third equation of \eqref{1.5}, we apply the It\^{o}'s formula to $\mathrm{d}\|\partial^\alpha\theta_S(t)\|_{L^2}^2$ and derive that
\begin{equation*}
\begin{split}
\mathbb{E}\sup_{t\in[0,\tau]}\|\theta_S(t)\|_{W^{k,2}}^2\leq& \mathbb{E}\|\theta_S^0\|_{W^{k,2}}^2+ \sum_{|\alpha|\leq k}\mathbb{E}\int_0^\tau |( \partial^\alpha  \theta_S,  \partial^\alpha (u_S\cdot \nabla ) \theta_S )_{L^2}|\mathrm{d}t\\
&+\sum_{|\alpha|\leq k}\mathbb{E}\int_0^\tau |( \partial^\alpha  \theta_S,  \partial^\alpha  u_T )_{L^2}|\mathrm{d}t\\
&+\sum_{|\alpha|\leq k}\mathbb{E}\sup_{t\in[0,\tau]}\left|\int_0^t( \partial^\alpha  \theta_S,  \partial^\alpha \sigma_3(u_S,u_T,\theta_S)  )_{L^2}\mathrm{d}\mathcal {W}\right|\\
 \leq& \mathbb{E}\|\theta_S^0\|_{W^{k,2}}^2+ \sum_{|\alpha|\leq k}\mathbb{E}\int_0^\tau \| \theta_S\|_{W^{k,2}}  \|[ \partial^\alpha ,u_S]\cdot \nabla \theta_S \|_{L^2}\mathrm{d}t\\
 &+\sum_{|\alpha|\leq k}\mathbb{E}\int_0^\tau \| \theta_S\|_{W^{k,2}} \| u_T\|_{W^{k,2}} \mathrm{d}t\\
&+\sum_{|\alpha|\leq k}\mathbb{E} \left(\int_0^\tau\| \theta_S\|_{W^{k,2}}^2\| \partial^\alpha \sigma_3(u_S,u_T,\theta_S)  \|_{\mathbb{W}^{0,2}}^2\mathrm{d}t\right)^{1/2}.
 \end{split}
\end{equation*}
By using the commutator estimate \eqref{2.7}, the condition (A1) and the similar manner as we did in the estimation for $\|u_T(t)\|_{W^{k,2}}$, one can deduce from the last inequality that
\begin{equation}\label{2.18}
\begin{split}
&\mathbb{E}\sup_{t\in[0,\tau]}\|\theta_S(t)\|_{W^{k,2}}^2 \\
&\quad \leq 2\mathbb{E}\|\theta_S^0\|_{W^{k,2}}^2+C\mathbb{E}\int_0^\tau \left(1+\|\nabla \theta_S\|_{L^\infty}+\|\nabla u_S\|_{L^\infty}+\kappa(\|u_S,u_T,\theta_S\|_{L^\infty})^2 \right)\\
&\quad \quad \times \left(1+\|u_S\|_{X^{k,2}}^2+\|u_T\|_{W^{k,2}}^2+\|\theta_S\|_{W^{k,2}}^2\right ) \mathrm{d}s.
 \end{split}
\end{equation}
To deal with the $L^\infty$-norm involved in above estimates, we use the definition of the stopping time $\mathbbm{t}_R$ to find
$$
 1+\|\nabla u_S\|_{L^\infty}+\|\nabla u_T\|_{L^\infty}+\|\nabla \theta_S\|_{L^\infty}+\kappa(\|u_S,u_T,\theta_S\|_{L^\infty})^2  \leq 1+R+\kappa(R)^2.
$$
Define
$$
F(t):=\|u_S(t)\|_{X^{k,2}}^2+\|u_T(t)\|_{W^{k,2}}^2+\|\theta_S(t)\|_{W^{k,2}}^2.
$$
Combining the estimates \eqref{2.11}, \eqref{2.17}, \eqref{2.18}, and taking $\tau=\mathbbm{t}_R \wedge t$ for any $t>0$, we get
\begin{equation}\label{2.19}
\begin{split}
\mathbb{E}\sup_{t\in[0, \mathbbm{t}_R \wedge t]}F(t)\leq&  \mathbb{E} F(0)+C\mathbb{E}\int_0^{ t} (1+R+\kappa(R)^2) \left(1+F(\mathbbm{t}_R \wedge s)\right ) \mathrm{d}s \\ \leq&  \mathbb{E} F(0)+C \int_0^{t}(1+R+\kappa(R)^2) \Bigg(1+\mathbb{E}\sup_{r\in[0, \mathbbm{t}_R \wedge s]}F(r)\Bigg) \mathrm{d}s.
 \end{split}
\end{equation}
By  using the Gronwall inequality to \eqref{2.19}, we get
\begin{equation}\label{2.20}
\begin{split}
\mathbb{E}\sup_{t\in[0, \mathbbm{t}_R \wedge T]}F(t)\leq e^{C(1+R+\kappa(R)^2) T}(1+\mathbb{E} F(0)),
 \end{split}
\end{equation}
which proves the desired estimate in the case of $p=2$.

Now let us establish the estimates in the case of $p>2$. To this end, we first apply the It\^{o}'s formula in $L^p$ space to $\mathrm{d}\| \partial^\alpha u_S \|_{L^p}^p=\mathrm{d}(\| \partial^\alpha u_S \|_{L^p}^2)^{p/2}$ related to the first equation in \eqref{1.5},
and then integrate the resulted identity on $D\times [0,t]$ for any $t\in [0,\tau]$, it follows from  the stochastic Fubini theorem (cf. Theorem 4.33 in \cite{43}) that
\begin{equation}\label{2.21}
\begin{split}
 \|u_S(t)\|_{X^{k,p}} ^p=& \|u_S^0\|_{X^{k,p}} ^p- p\sum_{|\alpha|\leq k}\int_0^t\int_D| \partial^\alpha u_S |^{p-2} \partial^\alpha u_S \cdot \partial^\alpha P(u_S\cdot \nabla)   u_S\mathrm{d}V\mathrm{d}s\\
 &+p\sum_{|\alpha|\leq k}\int_0^t\int_D | \partial^\alpha u_S |^{p-2}  \partial^\alpha u_S \cdot   \partial^\alpha P( u_T\widehat{x}  +   \theta _S\widehat{z})  \mathrm{d}V\mathrm{d}s\\
 &+  \sum_{|\alpha|\leq k} \sum_{j\geq1}\int_0^t\int_D\bigg(\frac{p}{2}| \partial^\alpha u_S |^{p-2}( \partial^\alpha P\sigma_1(u_S,u_T,\theta_S)e_j)^2  \\
 &\quad \quad\quad\quad\quad+ \frac{p(p-2)}{2}| \partial^\alpha u_S |^{p-4}( \partial^\alpha u_S \cdot \partial^\alpha P\sigma_1(u_S,u_T,\theta_S)e_j)^2\bigg)\mathrm{d}V\mathrm{d}s\\
&+ p\sum_{|\alpha|\leq k}\sum_{j\geq1} \int_0^t\int_D| \partial^\alpha u_S |^{p-2} \partial^\alpha u_S \cdot \partial^\alpha P\sigma_1(u_S,u_T,\theta_S)e_j \mathrm{d}V\mathrm{d}\beta_j \\
:=& \|u_S^0\|_{X^{k,p}} ^p+\sum_{|\alpha|\leq k}(L_1^\alpha(t)+L_2^\alpha(t)+L_3^\alpha(t)+L_4^\alpha(t)).
 \end{split}
\end{equation}
By virtue of the incompressible condition $\nabla\cdot u_S=0$, we deduce from the divergence theorem that
\begin{equation}\label{2.22}
\begin{split}
 &\int_0^t\int_D| \partial^\alpha u_S |^{p-2} \partial^\alpha u_S \cdot(u_S\cdot\nabla) \partial^\alpha  u_S\mathrm{d}V\mathrm{d}t  \\
  &\quad=\frac{1}{2}\sum_{i}\int_0^t\int_D| \partial^\alpha u_S |^{p-2} u_S^i\partial_i| \partial^\alpha  u_S^j|^2 \mathrm{d}V\mathrm{d}t\\
 &\quad=\frac{1}{p} \int_0^t\int_D u_S \cdot \nabla | \partial^\alpha  u_S|^p \mathrm{d}V\mathrm{d}t=\frac{1}{p} \int_0^t\int_D | \partial^\alpha  u_S|^p\nabla \cdot u_S \mathrm{d}V\mathrm{d}t=0.
 \end{split}
\end{equation}
Using \eqref{2.22}, the commutator estimate and the H\"{o}lder inequality, the term $L_1^\alpha(t)$ can be treated as
\begin{equation}\label{2.23}
\begin{split}
 L_1^\alpha(t) =&p\int_0^t\int_D| \partial^\alpha u_S |^{p-2} \partial^\alpha u_S \cdot\Big(P([ \partial^\alpha ,u_S]\cdot \nabla u_S)+(u_S\cdot\nabla) \partial^\alpha  u_S\\
 &-   (I-P)( \partial^\alpha  u_S\cdot\nabla) u_S\Big)\mathrm{d}V\mathrm{d}t \\
 \leq&C\int_0^t\| \partial^\alpha u_S \|_{L^p}^{ p-1 }\left(\|P([ \partial^\alpha ,u_S]\cdot \nabla u_S)\| _{L^p} + \|(I-P)( \partial^\alpha  u_S\cdot\nabla) u_S\| _{L^p}\right)\mathrm{d}s \\
 \leq& C\int_0^t\| u_S \|_{W^{k,p}}^{ p-1 }(\| [ \partial^\alpha ,u_S]\cdot \nabla u_S\| _{L^p}+\|( \partial^\alpha  u_S\cdot\nabla) u_S\| _{L^p}) \mathrm{d}t \\
 \leq& C\int_0^t\|\nabla u_S\| _{L^\infty}\| u_S \|_{X^{k,p}}^{ p}\mathrm{d}t.
 \end{split}
\end{equation}
For $L_2^\alpha(t)$, one deduce from the definition of Leray projection and the Young inequality that
\begin{equation}\label{2.24}
\begin{split}
 L_2^\alpha(t) &\leq C\sup_{s\in [0,t]}\|u_S(s) \|_{X^{k,p}}^{p-1}\int_0^t (\|u_T \|_{W^{k,p}}+\|\theta_S \|_{W^{k,p}})\mathrm{d}s\\
 &\leq \frac{1}{4}\sup_{s\in [0,t]}\|u_S \|_{X^{k,p}}^{p}+C\int_0^t (\|u_T \|_{W^{k,p}}^p+\|\theta_S \|_{W^{k,p}}^p)\mathrm{d}s.
 \end{split}
\end{equation}
By using the condition (A1) and the H\"{o}lder inequality, one can estimate the third term as
\begin{equation}\label{2.25}
\begin{split}
 L_3^\alpha(t)  \leq & C\int_0^t\int_D| \partial^\alpha u_S |^{p-2}\sum_{j\geq1}| \partial^\alpha P\sigma_1(u_S,u_T,\theta_S)e_j|^2  \mathrm{d}V\mathrm{d}s\\
 \leq& C \int_0^t\| \partial^\alpha u_S\|_{L^p}^{p-2} \| \partial^\alpha P\sigma_1(u_S,u_T,\theta_S)\|_{\mathbb{W}^{0,p}}^2\mathrm{d}s\\
\leq& C \int_0^t \kappa(\|u_S,u_T,\theta_S\|_{L^\infty})^2\| u_S\|_{X^{k,p}}^{p-2} \left(1+\|u_S\|_{X^{k,p}}^2+\|u_T\|_{W^{k,p}}^2+\|\theta_S\|_{W^{k,p}}^2 \right)\mathrm{d}s.
 \end{split}
\end{equation}
For the stochastic term $L_3^\alpha(t)$, by applying the BDG inequality, the Minkowski inequality, the H\"{o}lder inequality as well as the condition (A1), we get for any stopping time $\tau\leq \mathbbm{t}\wedge T$ that
\begin{equation}\label{2.26}
\begin{split}
\mathbb{E}\sup_{t\in [0,\tau]}|L_4^\alpha(t)|\leq&C\mathbb{E} \Bigg( \int_0^\tau\bigg\|\bigg(\int_D| \partial^\alpha u_S |^{p-2} \partial^\alpha u_S \cdot \partial^\alpha P\sigma_1(u_S,u_T,\theta_S)e_j \mathrm{d}V\bigg)_{j\geq 1}\bigg\|_{l^2}^2\mathrm{d}s\Bigg)^{1/2}\\
\leq & C \mathbb{E} \Bigg( \int_0^\tau\bigg(\int_D| \partial^\alpha u_S |^{ p-1}\| (  \partial^\alpha P\sigma_1(u_S,u_T,\theta_S)e_j )_{j\geq 1}\|_{l^2} \mathrm{d}V\bigg)^2\mathrm{d}s\Bigg)^{1/2}\\
\leq & C \mathbb{E} \Bigg( \int_0^\tau\| \partial^\alpha u_S \|^{2p-2}_{L^p}\Big\|\|  (  \partial^\alpha P\sigma_1(u_S,u_T,\theta_S)e_j  )_{j\geq 1}\|_{l^2}\Big\|_{L^p}^2\mathrm{d}s\Bigg)^{1/2}\\
\leq & C \mathbb{E} \Bigg( \int_0^\tau\| u_S \|^{2p-2}_{X^{k,p}}\bigg(\int_D \| \partial^\alpha P\sigma_1(u_S,u_T,\theta_S)\|_{L_2(\mathfrak{A},\mathbb{R}^2)}^p\mathrm{d}V\bigg)^{2/p}   \mathrm{d}s\Bigg)^{1/2}\\
\leq & C \mathbb{E} \Bigg( \sup_{t\in [0,\tau]}\| u_S \|^{ p }_{X^{k,p}}\int_0^\tau\| u_S \|^{p-2}_{X^{k,p}}\| \partial^\alpha P\sigma_1(u_S,u_T,\theta_S)\|_{\mathbb{X}^{0,p}}^2 \mathrm{d}s\Bigg)^{1/2}\\
\leq & \frac{1}{4} \mathbb{E} \sup_{t\in [0,\tau]}\| u_S \|^{p}_{X^{k,p}}\\
&  +C\mathbb{E} \int_0^\tau \kappa(\|u_S,u_T,\theta_S\|_{L^\infty})^2 \| u_S\|_{X^{k,p}}^{p-2} \left(1+\|u_S\|_{X^{k,p}}^2+\|u_T\|_{W^{k,p}}^2+\|\theta_S\|_{W^{k,p}}^2 \right)\mathrm{d}s.
 \end{split}
\end{equation}
Define
\begin{equation*}
\begin{split}
\Xi^{k,p}(t):=\|u_S(t)\|_{X^{k,p}}^p+\|u_T(t)\|_{W^{k,p}}^p+\|\theta_S(t)\|_{W^{k,p}}^p.
 \end{split}
\end{equation*}
Combining the estimates \eqref{2.23}-\eqref{2.26}, taking the supremum over $t\in [0,\tau]$ on both sides of \eqref{2.22} and then taking the expectation, we obtain
\begin{equation}\label{2.27}
\begin{split}
&\mathbb{E}\sup_{t\in [0,\tau]}\|u_S(t)\|_{X^{k,p}} ^p\leq 2\mathbb{E}\|u_S^0\|_{X^{k,p}} ^p \\
&\quad \quad\quad+C\mathbb{E} \int_0^\tau \left(1+\|\nabla u_S\| _{L^\infty}+\kappa(\|u_S,u_T,\theta_S\|_{L^\infty})^2\right) (1+\Xi^{k,p}(s))\mathrm{d}s.
 \end{split}
\end{equation}
 We apply the It\^{o}'s formula to
$\mathrm{d}\| \partial^\alpha u_T \|_{L^p}^p=\mathrm{d}(\| \partial^\alpha u_T \|_{L^p}^2)^{p/2}$ with respect to
the second equation in \eqref{1.5}. After integration by parts, we get
\begin{equation}\label{2.28}
\begin{split}
\mathrm{d}\| \partial^\alpha u_T \|_{L^p}^p\leq&p\int_D| \partial^\alpha u_T |^{p-1}|[ \partial^\alpha ,u_S]\cdot \nabla   u_T|\mathrm{d}V\mathrm{d}t+p\int_D| \partial^\alpha u_T |^{p-1}| \partial^\alpha u_S| \mathrm{d}V\mathrm{d}t\\
&+ \sum_{j\geq1}\frac{p (p-1)}{2}\int_D| \partial^\alpha u_T |^{p-2} ( \partial^\alpha \sigma_2(u_S,u_T,\theta_S)e_j)^2\mathrm{d}V\mathrm{d}t\\
 &+ p\sum_{j\geq1}\bigg|\int_D| \partial^\alpha u_T |^{p-2}
  \partial^\alpha u_T \partial^\alpha \sigma_2(u_S,u_T,\theta_S)e_j\mathrm{d}V\mathrm{d}\beta_j\bigg|,
 \end{split}
\end{equation}
where we have used the fact of
$$
\int_D| \partial^\alpha u_T |^{p-2} \partial^\alpha u_T (u_S\cdot \nabla ) \partial^\alpha  u_T\mathrm{d}V=-\frac{1}{2}\int_D| \partial^\alpha  u_T|^p \div u_S\mathrm{d}V=0.
$$
Integrating the both sides of \eqref{2.28} over $[0,t]$, and then taking the supremum over $t\in[0,\tau]$ for any stopping time $\tau$. After taking the expected values and using the BDG inequality, we get
\begin{equation}\label{2.29}
\begin{split}
 &\mathbb{E}\sup_{t\in [0,\tau]}\|u_T(t) \|_{W^{k,p}}^p\\
 &\leq \mathbb{E}\| u_T^0 \|_{W^{k,p}}^p+C\sum_{|\alpha|\leq k}\mathbb{E}\int_0^\tau\| u_T \|_{W^{k,p}}^{p-1}(\|[ \partial^\alpha ,u_S]\cdot \nabla   u_T\|_{L^p}+\| u_S\|_{W^{k,p}}) \mathrm{d}s\\
&\quad + C\sum_{|\alpha|\leq k}\mathbb{E}\int_0^\tau\int_D| \partial^\alpha u_T |^{p-2} \sum_{j\geq1}( \partial^\alpha \sigma_2(u_S,u_T,\theta_S)e_j)^2\mathrm{d}V\mathrm{d}s\\
 &\quad + C\sum_{|\alpha|\leq k}\mathbb{E}\Bigg(\int_0^\tau \sum_{j\geq1}\bigg(\int_D| \partial^\alpha u_T |^{p-2} \partial^\alpha u_T \partial^\alpha \sigma_2(u_S,u_T,\theta_S)e_j\mathrm{d}V\bigg)^2\mathrm{d}s\Bigg)^{1/2}\\
& :=\mathbb{E}\| u_T^0 \|_{W^{k,p}}^p+\sum_{|\alpha|\leq k}( M_1^\alpha(\tau)+ M_2^\alpha(\tau)+ M_3^\alpha(\tau)).
 \end{split}
\end{equation}
The term $M_1^\alpha(\tau)$ can be treated by using the commutator estimate as
\begin{equation}\label{2.30}
\begin{split}
M_1^\alpha(\tau)\leq& C\mathbb{E}\int_0^\tau\| u_T \|_{W^{k,p}}^{p-1}( \|\nabla  u_S\|_{L^\infty}\| u_T\|_{W^{k,p}}+ \| u_S\|_{X^{k,p}}\|\nabla u_T\|_{L^\infty}+\| u_S\|_{W^{k,p}}) \mathrm{d}s\\
\leq& C\mathbb{E}\int_0^\tau (1+\|\nabla  u_S\|_{L^\infty}+\|\nabla u_T\|_{L^\infty})(\| u_S\|_{X^{k,p}}^p+\| u_T\|_{W^{k,p}}^p) \mathrm{d}s.
 \end{split}
\end{equation}
For $M_2^\alpha(\tau)$, we deduce from the condition (A1) and the Young inequality that
\begin{equation}\label{2.31}
\begin{split}
M_2^\alpha(\tau)\leq& C\mathbb{E}\int_0^\tau\| \partial^\alpha u_T \|^{p-2}_{L^p}
\bigg( \int_D| \partial^\alpha  \sigma_2(u_S,u_T,\theta_S)|_{L_2(\mathfrak{A};\mathbb{R})}^pdV\bigg)^{2/p}\mathrm{d}s\\
\leq& C\mathbb{E}\int_0^\tau\| \partial^\alpha u_T \|^{p-2}_{L^p} \kappa(\|(u_S,u_T,\theta_S)\|_{L^\infty})^2 (1+\|(u_S,u_T,\theta_S)\|_{\mathcal {Z}^{k,p}}^2)\mathrm{d}s \\
\leq& C\mathbb{E}\int_0^\tau\kappa(\|u_S,u_T,\theta_S\|_{L^\infty})^2 (1+\Xi^{k,p}(s))\mathrm{d}s.
 \end{split}
\end{equation}
For $M_3^\alpha(\tau)$, we have
\begin{equation}\label{2.32}
\begin{split}
M_3^\alpha(\tau)\leq&  C\mathbb{E}\Bigg(\int_0^\tau\bigg(\int_D| \partial^\alpha u_T |^{p-1}\bigg(\sum_{j\geq1}( \partial^\alpha \sigma_2(u_S,u_T,\theta_S)e_j)^2\bigg)^{1/2}\mathrm{d}V\bigg)^2\mathrm{d}s\Bigg)^{1/2}\\
\leq& C\mathbb{E}\Bigg(\int_0^\tau\bigg(\int_D| \partial^\alpha u_T |^{ p }\mathrm{d}V\bigg)^{(2p-2)/p} \bigg(\int_D \bigg(\sum_{j\geq1}( \partial^\alpha \sigma_2(u_S,u_T,\theta_S)e_j)^2\bigg)^{p/2}\mathrm{d}V\bigg)^{2}\mathrm{d}s\Bigg)^{1/2}\\
\leq& C\mathbb{E}\Bigg(\sup_{t\in [0,\tau]}\| u_T \|_{W^{k,p}}^{p/2} \bigg(\int_0^\tau \| u_T \|_{W^{k,p}}^{p-2}\| \partial^\alpha  \sigma_2\|_{\mathbb{W}^{0,p}}^2 \mathrm{d}s\bigg)^{1/2}\Bigg)\\
\leq& \frac{1}{2}\mathbb{E}\sup_{t\in [0,\tau]}\| u_T \|_{W^{k,p}}^p+C\mathbb{E}\int_0^\tau \kappa(\|u_S,u_T,\theta_S\|_{L^\infty})^2 (1+\Xi^{k,p}(s))\mathrm{d}s.
 \end{split}
\end{equation}
Plugging the estimates \eqref{2.30}-\eqref{2.32} into \eqref{2.29} yields that
\begin{equation}\label{2.33}
\begin{split}
&\mathbb{E}\sup_{t\in [0,\tau]}\|u_T(t) \|_{W^{k,p}}^p\leq 2\mathbb{E}\|u_T^0\|_{X^{k,p}} ^p \\
&\quad +C\mathbb{E} \int_0^\tau (1+\|\nabla u_S\| _{L^\infty}+\|\nabla u_T\|_{L^\infty}+\kappa(\|u_S,u_T,\theta_S\|_{L^\infty})^2) (1+\Xi^{k,p}(s))\mathrm{d}s.
 \end{split}
\end{equation}
By applying the It\^{o}'s formula to  $\mathrm{d}\| \partial^\alpha \theta_S\|_{L^{p}}^p$ with respect to the scalar equation in \eqref{1.5}, one can arrive at
\begin{equation}\label{2.34}
\begin{split}
 \| \partial^\alpha \theta_S(t)\|^p_{L^p}&=\| \partial^\alpha \theta_S^0\|^p_{L^p}-p\int_0^t\int_D| \partial^\alpha \theta_S|^{p-2} \partial^\alpha \theta_S  \partial^\alpha  (u_S\cdot \nabla )  \theta_S\mathrm{d} V\mathrm{d}s\\
 &\quad-p\int_0^t\int_D| \partial^\alpha \theta_S|^{p-2} \partial^\alpha \theta_S  \partial^\alpha u_T\mathrm{d} V\mathrm{d}s\\
&\quad+ \frac{p(p-1)}{2}\sum_{j\geq 1} \int_0^t\int_D| \partial^\alpha \theta_S|^{p-2}(  \partial^\alpha \sigma_3(u_S,u_T,\theta_S)e_j)^2\mathrm{d}V\mathrm{d} s\\
 &\quad+ p\sum_{j\geq 1}\int_0^t\int_D| \partial^\alpha \theta_S|^{p-2} \partial^\alpha \theta_S  \partial^\alpha \sigma_3(u_S,u_T,\theta_S)e_j\mathrm{d} V\mathrm{d}\beta_j(s)\\
 &=\| \partial^\alpha \theta_S^0\|^p_{L^p}+\int_0^t(N_1^\alpha(s)+N_2^\alpha(s)+N_3^\alpha(s))\mathrm{d} s+\int_0^t N_4^\alpha(s) \mathrm{d}\mathcal {W}.
 \end{split}
\end{equation}
After integration by parts and using the free divergence condition $\div u_S=0$, we have
\begin{equation}\label{2.35}
\begin{split}
&\int_0^t\int_D| \partial^\alpha \theta_S|^{p-2} \partial^\alpha \theta_S  \partial^\alpha  (u_S\cdot \nabla )  \theta_S\mathrm{d} V\mathrm{d}t\\
&\quad =\int_0^t\int_D| \partial^\alpha \theta_S|^{p-2} \partial^\alpha \theta_S [ \partial^\alpha ,u_S]\cdot\nabla  \theta_S\mathrm{d} V\mathrm{d}t+\frac{1}{p}\int_0^t\int_Du_S\cdot \nabla| \partial^\alpha \theta_S|^{p}\mathrm{d} V\mathrm{d}t\\
&\quad =\int_0^t\int_D| \partial^\alpha \theta_S|^{p-2} \partial^\alpha \theta_S [ \partial^\alpha ,u_S]\cdot\nabla  \theta_S\mathrm{d} V\mathrm{d}t.
 \end{split}
\end{equation}
For any stopping time $\tau$, one can estimate the integration for $N_1^\alpha$ as
\begin{equation}\label{2.36}
\begin{split}
\mathbb{E}\sup_{t\in [0,\tau]}\bigg |\int_0^tN_1^\alpha(s)\mathrm{d}s\bigg| &\leq C\mathbb{E} \int_0^\tau\int_D| \partial^\alpha \theta_S|^{p-1}|[ \partial^\alpha ,u_S]\cdot\nabla  \theta_S|\mathrm{d} V\mathrm{d}s\\
& \leq C\mathbb{E} \int_0^\tau \| \partial^\alpha \theta_S\|_p^{p-1}\|[ \partial^\alpha ,u_S]\cdot\nabla  \theta_S\|_p\mathrm{d}s\\
&  \leq C\mathbb{E} \int_0^\tau\| \partial^\alpha \theta_S\|_p^{p-1}( \|\nabla u_S\|_{L^\infty} \| \nabla \theta_S\|_{W^{|\alpha|-1,p}}\\
&\quad + \| u_S\|_{W^{|\alpha|,p}}\|\nabla \theta_S\|_{L^\infty}) \mathrm{d}s\\
&  \leq C\mathbb{E} \int_0^\tau (\|\nabla u_S\|_{L^\infty}+\|\nabla \theta_S\|_{L^\infty})(\|u_S\|_{W^{|\alpha|,p}}^{p}+\|\theta_S\|_{W^{|\alpha|,p}}^{p}) \mathrm{d}s.
 \end{split}
\end{equation}
Using the H\"{o}lder inequality we have
\begin{equation}\label{2.37}
\begin{split}
\mathbb{E}\sup_{t\in [0,\tau]}\bigg |\int_0^tN_2^\alpha(s)\mathrm{d}s\bigg|
&\leq C\mathbb{E} \int_0^\tau\int_D| \partial^\alpha \theta_S|^{p-1}| \partial^\alpha u_T|\mathrm{d} V\mathrm{d}s\\
& \leq C\mathbb{E} \int_0^\tau (\| \partial^\alpha \theta_S\|_p^p+\| \partial^\alpha u_T\|_p^p)\mathrm{d}s.
 \end{split}
\end{equation}
By applying the Minkowski inequality and the condition (A1), the integral with respect to $N_3^\alpha$ can be estimated as
\begin{equation}\label{2.38}
\begin{split}
\mathbb{E}\sup_{t\in [0,\tau]}\bigg |\int_0^tN_3^\alpha(s)\mathrm{d}s\bigg|
&\leq C\mathbb{E}\sum_{j\geq 1} \int_0^\tau\int_D
| \partial^\alpha \theta_S|^{p-2}(  \partial^\alpha \sigma_3(u_S,u_T,\theta_S)e_j)^2\mathrm{d}V\mathrm{d} s\\
&\leq C\mathbb{E} \int_0^\tau
\| \partial^\alpha \theta_S\|^{p-2}_{L^p}\Bigg(\int_D\bigg(\sum_{j\geq 1}
(  \partial^\alpha \sigma_3(u_S,u_T,\theta_S)e_j)^2\bigg)^{p/2}\Bigg)^{2/p}\mathrm{d}V\mathrm{d} s\\
&\leq C\mathbb{E} \int_0^\tau
\| \partial^\alpha \theta_S\|^{p-2}_{L^p} \| \partial^\alpha \sigma_3(u_S,u_T,\theta_S))\|^{2}_{\mathbb{W}^{0,p}}\mathrm{d} s\\
&\leq C\mathbb{E} \int_0^\tau \kappa(\|u_S,u_T,\theta_S\|_{L^\infty})^2
\| \partial^\alpha \theta_S\|^{p-2}_{L^p}(1+\Xi^{k,p-2}(s))\mathrm{d} s\\
&\leq C\mathbb{E} \int_0^\tau \kappa(\|u_S,u_T,\theta_S\|_{L^\infty})^2
(1+\Xi^{k,p}(s))\mathrm{d} s.
 \end{split}
\end{equation}
For the stochastic term, one can deduce from the BDG inequality and \eqref{2.38} that
\begin{equation}\label{2.39}
\begin{split}
&\mathbb{E}\sup_{t\in [0,\tau]}\bigg |\int_0^t N_4^\alpha(s)\mathrm{d}\mathcal {W}\bigg|\\
&\quad\leq C\mathbb{E}\Bigg(\int_0^\tau\bigg(\sum_{j\geq 1}
\int_D| \partial^\alpha \theta_S|^{p-2} \partial^\alpha \theta_S  \partial^\alpha \sigma_3(u_S,u_T,\theta_S)e_j\mathrm{d} x\bigg)^2
\mathrm{d}s\Bigg)^{1/2}\\
&\quad\leq C\mathbb{E}\Bigg(\sup_{s\in [0,\tau]}\| \partial^\alpha \theta_S\|_{L^p}^{p/2}
\int_0^\tau \| \partial^\alpha \theta_S\|_{L^p}^{p-2} \kappa(\|u_S,u_T,\theta_S\|_{L^\infty})^2
(1+\Xi^{k,2}(s))    \mathrm{d}s\Bigg)^{1/2}\\
&\quad\leq \frac{1}{2}\mathbb{E} \sup_{t\in [0,\tau]} \| \partial^\alpha \theta_S\|_{L^p}^p+C\mathbb{E}\int_0^\tau \kappa(\|u_S,u_T,\theta_S\|_{L^\infty})^2
(1+\Xi^{k,p}(s))\mathrm{d} s.
 \end{split}
\end{equation}
Plugging  above estimates \eqref{2.35}-\eqref{2.39} into \eqref{2.34} and then  summing up with respect to the index $|\alpha|\leq k$, we deduce that
\begin{equation}\label{2.40}
\begin{split}
&\mathbb{E}\sup_{t\in [0,\tau]}\|\theta_S(t) \|_{W^{k,p}}^p\leq  2\mathbb{E}\|\theta_S^0\|_{X^{k,p}} ^p \\
&\quad+C\mathbb{E} \int_0^\tau \left(1+\|\nabla u_S\| _{L^\infty}+\|\nabla \theta_S\|_{L^\infty}+\kappa(\|u_S,u_T,\theta_S\|_{L^\infty})^2\right) (1+\Xi^{k,p}(s))\mathrm{d}s.
 \end{split}
\end{equation}
Combining the estimates \eqref{2.27}, \eqref{2.29} and \eqref{2.40}, and utilizing the stopping time $\mathbbm{t}_R$ defined in \eqref{2.2}, we obtain
\begin{equation*}
\begin{split}
1+\mathbb{E}\sup_{s\in [0,\mathbbm{t}_R \wedge t]}\Xi^{k,p}(s)\leq& 2\mathbb{E}\Xi^{k,p}(0) +C\mathbb{E} \int_0^{\mathbbm{t}_R \wedge t} \Big(1+\|\nabla u_S\| _{L^\infty}\\
&+\|\nabla u_T\|_{L^\infty}+\|\nabla \theta_S\| _{L^\infty}+\kappa(\|u_S,u_T,\theta_S\|_{L^\infty})^2\Big)(1+\Xi^{k,p}(s))\mathrm{d}s\\
\leq&2\mathbb{E}\Xi^{k,p}(0) +C\mathbb{E} \int_0^{t} (1+R +\kappa(R)^2)(1+\Xi^{k,p}(s))\mathrm{d}s\\
\leq&2\mathbb{E}\Xi^{k,p}(0) +C \int_0^{t}(1+R +\kappa(R)^2 )\Bigg(1+\mathbb{E}\sup_{r\in [0,\mathbbm{t}_R \wedge s]}\Xi^{k,p}(r)\Bigg)\mathrm{d}s.
 \end{split}
\end{equation*}
Applying the Gronwall inequality  to above inequality leads to the desired inequality \eqref{2.1}. This completes the proof of Lemma \ref{lem:2.2}.
\end{proof}

\subsection{Galerkin approximation and uniform bounds}

Let $(\phi_j)_{j\geq1}$ be a $C^\infty$ complete orthonormal system of $X^0(D;\mathbb{R}^d)$ made up of eigenfunctions of the Stokes operator $P(-\Delta)$ (cf. \cite{42}). For all $n\geq1$, we consider the orthogonal projection operator $P_n$ from $X^0(D;\mathbb{R}^d)$ onto span$\{\phi_1,\phi_2,...,\phi_n\}$, given by
$$
P_n: v\longmapsto v_n=\sum_{j=1}^n(v,\phi_j)\phi_j,\quad \mbox{for all}~   v\in X^0(D;\mathbb{R}^d).
$$
As we mentioned before, in order to obtain some uniform a priori estimates for the approximate solutions, we would like to multiply the nonlinear terms by some cut-off function $\varpi_R(\cdot)$ depending only on the size of $\|u_S\|_{X^{1,\infty}} $, $\|u_T\|_{W^{1,\infty}} $ and $\|\theta_S\|_{W^{1,\infty}} $, where $\omega:[0,\infty)\mapsto [0,1]$ is a smooth non-increasing truncation function defined by
\begin{equation*}
\varpi (x):=
\begin{cases}
1,&0\leq x \leq 1,\cr
0,& x > 2,
\end{cases}	\quad \mbox{with}\quad \sup_{x\in [0,\infty)}|\varpi'(x)| \leq C<\infty.
\end{equation*}
For each  $R>0$, we define the cut-off function
$$
\varpi_R (x):=\varpi \left(\frac{x}{R}\right),\quad \mbox{for }~\forall x\in [0,\infty).
$$
By introducing appropriate cut-off functions to nonlinear terms, the SISM \eqref{1.5} can be approximated by the following truncated system taking values in $P_n\mathcal {Z}^{k',2}(D):=P_nX^{m',2}\times P_nW^{m',2}\times P_nW^{m',2}$:
\begin{equation}\label{2.41}
\left\{
\begin{aligned}
&\mathrm{d}u_S^n+  \varpi_R(\|u_S^n\|_{W^{1,\infty}})P_nP(u_S^n\cdot \nabla)   u_S^n\mathrm{d}t-P_nP(  u_T^n\widehat{x}) \mathrm{d}t\\
&\quad - P_nP( \theta _S^n\widehat{z}) \mathrm{d}t=  \mathcal {D}_{cut,n}^1(t)\mathrm{d}\mathcal {W}_t,\\
&\mathrm{d} u_T^n+ \varpi_R(\|(u_S^n,u_T^n)\|_{X^{1,\infty}\times W^{1,\infty}})P_n(u_S^n\cdot \nabla)   u_T^n\mathrm{d}t+ P_nu_S^n\cdot\widehat{x} \mathrm{d}t +  z \mathrm{d}t =\mathcal {D}_{cut,n}^2(t)\mathrm{d}\mathcal {W}_t, \\
&\mathrm{d}\theta_S^n+\varpi_R(\|(u_S^n,\theta_S^n)\|_{X^{1,\infty}\times W^{1,\infty}})P_n (u_S^n\cdot \nabla ) \theta_S^n\mathrm{d}t+P_nu_T^n  \mathrm{d}t = \mathcal {D}_{cut,n}^3(t)\mathrm{d}\mathcal {W}_t,\\
&u_S^n|_{t=0}=P_n u_S^0,\quad  u_T^n|_{t=0}=P_nu_T^0,\quad \theta_S^n|_{t=0}=P_n\theta_T^0,
\end{aligned}
\right.
\end{equation}
where the diffusions are defined by
\begin{equation*}
\begin{split}
\mathcal {D}_{cut,n}^1(t)=&\varpi_R(\|(u_S^n, u_T^n,\theta_S^n)\|_{\mathcal {Z}^{1,\infty}})
P_nP\sigma_1(u_S^n,u_T^n,\theta_S^n),
 \end{split}
\end{equation*}
and
\begin{equation*}
\begin{split}
\mathcal {D}_{cut,n}^j(t)=&\varpi_R(\|(u_S^n, u_T^n,\theta_S^n)\|_{\mathcal {Z}^{1,\infty}})
P_n\sigma_j(u_S^n,u_T^n,\theta_S^n),\quad j=2,3.
 \end{split}
\end{equation*}
In view of the condition (A1) on $\sigma_i$, $i=1,2,3$, \eqref{2.41} can be regarded as stochastic differential equations (SDEs) in finite dimensional spaces with locally Lipschitz drift and globally Lipschitz diffusion. According to the classic existence theory for SDEs (cf. \cite{43}), there exists a time $T_n>0$ and a triple
$$
(u_S^n,u_T^n,\theta_S^n)\in C([0,T_n); P_n\mathcal {Z}^{k',2}(D))
$$
satisfying the system \eqref{2.41} for a.e. $t\in [0,T_n]$. Since $u_S$ is a divergence free vector with boundary condition $u_S|_{\partial D}\cdot \vec{n}=0$, by using the symmetry property for the operators $P$ and $P_n$, we get by integration by parts that
$$
(P_nP(u_S^n\cdot \nabla)   u_S^n,u_S^n)_{L^2}=(P_nP(u_S^n\cdot \nabla)   u_T^n,u_T^n)_{L^2}=(P_nP(u_S^n\cdot \nabla)   \theta_S^n,\theta_S^n)_{L^2}=0.
$$
Moreover, similar to the proof of Lemma 1.2, we see that for each fixed $n$,
$$
T_n<\infty ~\Longrightarrow~\limsup_{t\rightarrow T_n}\|(u_S^n, u_T^n,\theta_S^n)(t)\|_{\mathcal {Z}^{1,\infty}}=\infty, \quad \mathbb{P}\mbox{-a.s.}
$$
Due to the appearance of cut-off functions, the norm $\|(u_S^n, u_T^n,\theta_S^n)\|_{\mathcal {Z}^{1,\infty}}$ is uniformly bounded by a positive constant depending only on $R$, and hence one can infer that the approximations $(u_S^n,u_T^n,\theta_S^n)$ are exactly global-in-time solutions, i.e., $(u_S^n,u_T^n,\theta_S^n)\in C([0,\infty); P_n\mathcal {Z}^{k',2}(D))$, $\mathbb{P}$-a.s., for each $n\geq 1$.

In order to prove the tightness of the probability measures with respect to the approximations $(u_S^n,u_T^n,\theta_S^n)$, one need to establish some a priori uniform bounds for approximations in sufficiently regular spaces.

\begin{lemma}\label{lem:2.3}
Fix $k'=k+4$, $k>1+\frac{2}{p}$ and  $\vartheta\in (0,\frac{1}{2})$. Assume that $(u^0_S,u^0_T,\theta^0_S)\in L^q(\Omega;X^{k',2}\times W^{k',2}\times W^{k',2})$ is a $\mathcal {F}_0$-measurable random variable for some $q\geq2$, and the diffusion coefficients satisfy the conditions (A1)-(A3). Then we have
\begin{equation*}
\begin{split}
(u_S^n,u_T^n,\theta_S^n)_{n\geq1}&\in L^q(\Omega; L^\infty_{loc}([0,\infty);\mathcal {Z}^{k',2}(D))) \bigcap L^q(\Omega; W^{\vartheta,q}_{loc}([0,\infty);\mathcal {Z}^{k'-1,2}(D))).
 \end{split}
\end{equation*}
Moreover, there hold
\begin{equation*}
\begin{split}
&\sup_{n\geq 1}\mathbb{E} \bigg\|u_S^n(\cdot)-\int_0^\cdot\mathcal {D}_{cut,n}^1\mathrm{d}\mathcal {W} \bigg\|_{W^{1,q}([0,T];X^{k'-1,2})}^q  <\infty,\\
&\sup_{n\geq 1}\mathbb{E} \bigg\|u_T^n(\cdot)-\int_0^\cdot\mathcal {D}_{cut,n}^2\mathrm{d}\mathcal {W} \bigg\|_{W^{1,q}([0,T];W^{k'-1,2})}^q  <\infty,\\
&\sup_{n\geq 1}\mathbb{E} \bigg\|\theta_S^n(\cdot)-\int_0^\cdot\mathcal {D}_{cut,n}^3\mathrm{d}\mathcal {W} \bigg\|_{W^{1,q}([0,T];W^{k'-1,2})}^q  <\infty,
 \end{split}
\end{equation*}
and
\begin{equation*}
\begin{split}
\sup_{n\geq 1}\mathbb{E} \bigg\|\int_0^\cdot\mathcal {D}_{cut,n}^1\mathrm{d}\mathcal {W} \bigg\|_{W^{\vartheta,q}([0,T];X^{k'-1,2})}^q +\sum_{j\in\{2,3\}}\sup_{n\geq 1}\mathbb{E}\bigg \|\int_0^\cdot\mathcal {D}_{cut,n}^j\mathrm{d}\mathcal {W} \bigg\|_{W^{\vartheta,q}([0,T];W^{k'-1,2})}^q  <\infty.
 \end{split}
\end{equation*}
\end{lemma}

\begin{proof}
By  applying the It\^{o}'s formula in Hilbert space $X^{k',2}$ to $\mathrm{d}\|u_S^n(t)\|_{X^{k',2}}^q=\sum_{|\alpha|\leq k'}\mathrm{d}\|\partial^ \alpha u_S^n(t)\|_{L^2}^q$ and then using the fact that the projection operator $P_n$ is self-adjoint, we find
\begin{equation*}
\begin{split}
 &\|\partial^ \alpha u_S^n(t)\|_{L^2}^q=\|P_n \partial^ \alpha u_S^0\|_{L^2}^q\\
 &\quad-q\int_0^t \varpi_R(\|u_S^n\|_{W^{1,\infty}})\|u_S^n\|_{X^{k',2}}^{q-2} ( \partial^ \alpha  u_S^n, \partial^ \alpha P(u_S^n\cdot \nabla)   u_S^n )_{L^2}\mathrm{d}s \\
 &\quad+q\int_0^t \|u_S^n\|_{X^{k',2}}^{q-2} ( \partial^ \alpha  u_S^n,  \partial^ \alpha P(u_T^n\widehat{x}) + \partial^ \alpha  P(\theta _S^n\widehat{z}) )_{L^2}\mathrm{d}s\\
&\quad+  \int_0^t \bigg(\frac{q}{2}\|u_S^n\|_{X^{k',2}}^{q-2}\| \partial^ \alpha  \mathcal {D}_{cut,n}^1\|_{\mathbb{X}^{0,2}}^2 + \frac{q(q-2)}{2}\|u_S^n\|_{X^{k',2}}^{q-4}( \partial^ \alpha  u_S^n,  \partial^ \alpha  \mathcal {D}_{cut,n}^1)_{L^2}^2\bigg)\mathrm{d}s\\
&\quad+q \int_0^t  \|u_S^n\|_{X^{k',2}}^{q-2}( \partial^ \alpha  u_S^n,  \partial^ \alpha  \mathcal {D}_{cut,n}^1)_{L^2}\mathrm{d}\mathcal {W}_s.
 \end{split}
\end{equation*}
For any $K>0$, we introduce the stopping time
$$
\mathbbm{t}_K:= \inf\bigg\{t>0;~ \sup_{s\in [0,t]} (\|u_S^n(t)\|_{X^{k',2}}+\|u_T^n(t)\|_{X^{k',2}}+\|\theta_S^n(t)\|_{X^{k',2}} )\geq K \bigg \}.
$$
One can use the similar calculations as we did for the a priori estimates in Subsection 2.1, and it follows from the definition of the cut-off function $\varpi_R(\cdot)$ that
\begin{equation}\label{2.42}
\begin{split}
\mathbb{E}\sup_{s\in [0,\mathbbm{t}_K \wedge t]} \|u_S^n(s)\|_{X^{k',2}}^q\leq&\mathbb{E}\| u_S^0\|_{X^{k',2}}^q+C\int_0^{\mathbbm{t}_K \wedge t}\varpi_R(\|u_S^n\|_{W^{1,\infty}})\|\nabla u_S\| _{L^\infty}\| u_S^n  \|_{W^{k',2}}^{ q}\mathrm{d}s\\
&+\frac{1}{2}\mathbb{E}\sup_{s\in [0,\mathbbm{t}_K \wedge t]}\|u_S \|_{X^{k',2}}^{q}+C\mathbb{E}\int_0^{\mathbbm{t}_K \wedge t} (\|u_T^n \|_{W^{k',2}}^q+\|\theta_S^n \|_{W^{k',2}}^q)\mathrm{d}s\\
 &+C\mathbb{E} \int_0^\tau \varpi_R(\|u_S^n, u_T^n,\theta_S^n\|_{W^{1,\infty}})^2\kappa(\|u_S^n,u_T^n,\theta_S^n\|_{L^\infty})^2  \\
&\quad\quad \times \left(1+\|u_S^n\|_{X^{k',2}}^q+\|u_T^n\|_{W^{k',2}}^q+\|\theta_S\|_{W^{k',2}}^q \right)\mathrm{d}s\\
\leq& \mathbb{E}\| u_S^0\|_{X^{k',2}}^q+\frac{1}{2}\mathbb{E}\sup_{s\in [0,\mathbbm{t}_K \wedge t]}\|u_S (s)\|_{X^{k',2}}^{q}\\
&+C\int_0^t (1+R+\kappa(R)^2 )\mathbb{E}\sup_{r\in[0,\mathbbm{t}_K \wedge s]}\Xi^{n,k',2}(r)\mathrm{d}s,
 \end{split}
\end{equation}
where $C$ is a positive constant depends only on $D,k'$ and $q$, and
\begin{equation*}
\begin{split}
\Xi^{n,k',2}(t):=\|u_S^n(t)\|_{X^{k',2}}^q+\|u_T^n(t)\|_{W^{k',2}}^q+\|\theta_S^n(t)\|_{W^{k',2}}^q.
 \end{split}
\end{equation*}
After rearranging the terms, we get from \eqref{2.42} that
\begin{equation}\label{2.43}
\begin{split}
\mathbb{E}\sup_{s\in [0,\mathbbm{t}_K \wedge t]} \|u_S^n(s)\|_{X^{k',2}}^q
\leq& 2\mathbb{E}\| u_S^0\|_{X^{k',2}}^q+C(1+R+\kappa(R)^2 )\int_0^t\mathbb{E}\sup_{r\in[0,\mathbbm{t}_K \wedge s]}\Xi^{n,k',2}(r)\mathrm{d}s.
 \end{split}
\end{equation}
In order to close the estimate in \eqref{2.43}, we shall apply the It\^{o}'s formula to the functions $\|u_T^n(s)\|_{W^{k',2}}^q$ and $\|\theta_S^n(s)\|_{W^{k',2}}^q$ with respect to the second and third equation in \eqref{2.41}, respectively. By using the property of cut-off functions, we can deduce that
\begin{equation}\label{2.44}
\begin{split}
\mathbb{E}\sup_{s\in [0,\mathbbm{t}_K \wedge t]} \|u_T^n(s)\|_{X^{k',2}}^q
\leq& 2\mathbb{E}\| u_T^0\|_{X^{k',2}}^q+C(1+R+\kappa(R)^2 )\int_0^t\mathbb{E}\sup_{r\in[0,\mathbbm{t}_K \wedge s]}\Xi^{n,k',2}(r)\mathrm{d}s.
 \end{split}
\end{equation}
and
\begin{equation}\label{2.45}
\begin{split}
\mathbb{E}\sup_{s\in [0,\mathbbm{t}_K \wedge t]} \|\theta_S^n(s)\|_{X^{k',2}}^q
\leq& 2\mathbb{E}\| \theta_S^0\|_{X^{k',2}}^q+C(1+R+\kappa(R)^2 )\int_0^t\mathbb{E}\sup_{r\in[0,\mathbbm{t}_K \wedge s]}\Xi^{n,k',2}(r)\mathrm{d}s.
 \end{split}
\end{equation}
Summing up the estimates \eqref{2.43}-\eqref{2.45} leads to
 \begin{equation*}
\begin{split}
\mathbb{E}\sup_{s\in [0,\mathbbm{t}_K \wedge t]} \Xi^{n,k',2}(s)
\leq& 2\mathbb{E}\Xi^{n,k',2}(0)+C\int_0^t\mathbb{E}\sup_{r\in[0,\mathbbm{t}_K \wedge s]}\Xi^{n,k',2}(r)\mathrm{d}s,
 \end{split}
\end{equation*}
which combined with the Gronwall inequality yield that
\begin{equation*}
\begin{split}
\mathbb{E}\sup_{s\in [0,\mathbbm{t}_K \wedge t]} \Xi^{n,k',2}(s)
\leq C,
 \end{split}
\end{equation*}
for some positive constant $C$ depending only on $D,k',q,R$ and the initial data $(u_S^0,u_T^0,\theta_S^0)$. As the stopping time $\mathbbm{t}_K$ is increasing with respect to the parameter $K$, by taking $K\rightarrow\infty$, we deduce from the monotone convergence theorem and the definition of $\Xi^{n,k',2}(t) $ that, for any $T>0$,
\begin{equation}\label{2.46}
\begin{split}
\sup_{n\geq1}\Bigg(\mathbb{E}\sup_{s\in [0,  T]} \left(\|u_S^n(s)\|_{X^{k',2}}^q+\|u_T^n(s)\|_{W^{k',2}}^q+\|\theta_S^n(s)\|_{W^{k',2}}^q\right)\Bigg)
\leq C,
 \end{split}
\end{equation}
which implies that $(u_S^n,u_T^n,\theta_S^n)\in L^{q}(\Omega;L^\infty([0,T];X^{k',2}\times W^{k',2}\times W^{k',2}))$ with the uniform bound $C>0$ independent of $n$.

Note that the Sobolev space $W^{1,q}([0,T];X^{k',2}) $ is continuously embedded into $W^{\vartheta,q}([0,T];X^{k',2})$, for any $\vartheta\in (0,1)$. By virtue of the first equation with respect to  $u_S$ and the inequality $(a+b+c)^p\leq 3^{q-1}(a^p+b^p+c^p)$, we have for any $\eta\in (0,1)$
\begin{equation}\label{2.47}
\begin{split}
&\mathbb{E}\|u_S^n \|_{W^{\vartheta,q}([0,T];X^{k'-1,2})}^q\\
&\quad \leq 3^{q-1}\mathbb{E}\left\|P_nu_S^0+\int_0^\cdot\varpi_R(\|u_S^n\|_{W^{1,\infty}})P_nP(u_S^n\cdot \nabla)   u_S^n\mathrm{d}s\right\|_{W^{1,q}([0,T];X^{k'-1,2})}^q\\
&\quad\quad +3^{q-1}\mathbb{E}\left\|\int_0^\cdot(P_nP(  u_T^n\widehat{x})  + P_nP( \theta _S^n\widehat{z}) )\mathrm{d}s\right\|_{W^{1,q}([0,T];X^{k'-1,2})}^q\\
& \quad\quad +3^{q-1}\mathbb{E}\left\|\int_0^\cdot\mathcal {D}_{cut,n}^1(s)\mathrm{d}\mathcal {W}_s\right\|_{W^{\vartheta,q}([0,T];X^{k'-1,2})}^q\\
& \quad :=N_1+ N_2+ N_3.
 \end{split}
\end{equation}
Let us verify that the terms $N_i,i=1,2,3$ are uniform bounded. From the definition of the Sobolev space $W^{1,q}([0,T];X^{k'-\eta,2})$ and the Moser type estimate (cf. \cite{1,13}), we have
\begin{equation*}
\begin{split}
 N_1 \leq & C\mathbb{E}\left\|P_nu_S^0\right\|_{X^{k',2}}^q+  C \mathbb{E}\left\|\int_0^t\varpi_R(\|u_S^n\|_{W^{1,\infty}})P_nP(u_S^n\cdot \nabla)   u_S^n\mathrm{d}s\right\|_{L^q([0,T];X^{k'-1,2})}^q\\
 & + C\mathbb{E}\left\| \varpi_R(\|u_S^n\|_{W^{1,\infty}})P_nP(u_S^n\cdot \nabla)   u_S^n \right\|_{L^q([0,T];X^{k'-1,2})}^q \\
 \leq& C \mathbb{E}\left\|P_nu_S^0\right\|_{X^{k',2}}^q+ C \mathbb{E}\int_0^t\varpi_R(\|u_S^n\|_{W^{1,\infty}})^q\left\|P_nP(u_S^n\cdot \nabla)   u_S^n\right\|_{X^{k'-1,2}}^q\mathrm{d}s \\
\leq& C \mathbb{E}\left\|P_nu_S^0\right\|_{X^{k',2}}^q+ C \mathbb{E}\int_0^t\varpi_R(\|u_S^n\|_{W^{1,\infty}})^q\\
&\quad \quad \times
(\|u_S^n\|_{L^{\infty}}\|\nabla u_S^n\|_{X^{k'-1,2}}+\|u_S^n\|_{X^{k'-1,2}}\|\nabla u_S^n\|_{L^{\infty}})^q\mathrm{d}s\\
\leq& C \mathbb{E}\left\| u_S^0\right\|_{X^{k',2}}^q+ C \mathbb{E}\sup_{s\in[0,T]}
\|u_S^n(s)\|_{X^{k',2}}^q <\infty,
 \end{split}
\end{equation*}
where the last inequality used the uniform bound \eqref{2.46}.

For $N_2$, we have
\begin{equation*}
\begin{split}
 N_2\leq& C\mathbb{E}\left\|\int_0^t(P_nP(  u_T^n\widehat{x})  + P_nP( \theta _S^n\widehat{z}) )\mathrm{d}s\right\|_{L^q([0,T];X^{k'-1,2})}^q\\
 &+C\mathbb{E}\left\|P_nP(  u_T^n\widehat{x})  + P_nP( \theta _S^n\widehat{z})\right\|_{L^q([0,T];X^{k'-1,2})}^q \\
\leq &C\mathbb{E}\Big( \left\|P(  u_T^n\widehat{x})\right\|_{L^q([0,T];X^{k',2})}^q  + \left\|P( \theta _S^n\widehat{z})\right\|_{L^q([0,T];X^{k',2})}^q\Big) \\
\leq&  C\mathbb{E}\sup_{s\in[0,T]} \Big( \left\| u_T^n(s)\right\|_{X^{k',2}}^q  + \left\| \theta _S^n(s)\right\|_{W^{k',2}}^q \Big)<\infty.
 \end{split}
\end{equation*}
For any given function $f\in W^{\vartheta,q}([0,T];X^{k',2})$ and any $\delta>0$, there exits a pair $(t_\delta,\overline{t_\delta})\in [0,T]\times [0,T]$ with $t_\delta\neq\overline{t_\delta}$ such that
\begin{equation*}
\begin{split}
&\sup_{t\neq\bar{t}\in [0,T]}
|t-\bar{t}|^{-\frac{1}{q}-\vartheta}\left\|f(t)-f(\bar{t})\right\|_{X^{k',2}} \leq |t_\delta-\bar{t}_\delta|^{-\frac{1}{q}-\vartheta}\left\|f(t)-f(\bar{t})\right\|_{X^{k',2}}+\delta.
 \end{split}
\end{equation*}
By applying the condition (A2) and the uniform bound \eqref{2.46}, we deduce that
\begin{equation}\label{2.48}
\begin{split}
 &\mathbb{E}\int_0^T\int_0^T
|t-\bar{t}|^{-1-q\vartheta}\bigg\|\int_t^{\bar{t}}\mathcal {D}_{cut,n}^1(s)\mathrm{d}\mathcal {W}_s\bigg\|_{X^{k',2}} ^qdtd\bar{t}\\
&\quad \leq C |t_\delta-\bar{t}_\delta|^{-1-q\vartheta}\mathbb{E}\left(\int_{t_\delta}^{\bar{t}_\delta} \|
\mathcal {D}_{cut,n}^1(s)\|_{\mathbb{X}^{k',2}} ^2 \mathrm{d} s\right)^{q/2} +C\delta^q\\
&\quad \leq C\kappa(R) |t_\delta-\bar{t}_\delta|^{-1-q\vartheta+\frac{q}{2}}\mathbb{E} \sup_{s\in[t_\delta,\bar{t}_\delta]}\left(1+\|u_S^n\|_{X^{k',2}}^q+\|u_T^n\|_{W^{k',2}}^q+\|\theta_S\|_{W^{k',2}} ^q\right)+C\delta^q\\
&\quad \leq C |t_\delta-\bar{t}_\delta|^{-1-q\vartheta+\frac{q}{2}}+C\delta^q  \\
&\quad \leq C T^{-1-q\vartheta+\frac{q}{2}}+C\delta^q .
 \end{split}
\end{equation}
By assumption of $q-2q\vartheta-2\geq0$, the estimate \eqref{2.48} implies that $[\int_t^{\bar{t}}\mathcal {D}_{cut,n}^j(s)\mathrm{d}\mathcal {W}_s]_{\vartheta,q,T}<\infty$. Moreover, an application of the BDG inequality leads to
\begin{equation*}
\begin{split}
&\mathbb{E}\left\|\int_0^t\mathcal {D}_{cut,n}^1(s)\mathrm{d}\mathcal {W}_s\right\|_{L^q([0,T];X^{k'-1,2})}^q\\
&\quad\leq \mathbb{E}\sup_{t\in [0,T]}\left\|\int_0^t\mathcal {D}_{cut,n}^1(s)\mathrm{d}\mathcal {W}_s\right\|_{X^{k'-1,2}}^q\\
&\quad \leq \mathbb{E}\bigg(\int_0^T\varpi_R(\|u_S^n, u_T^n,\theta_S^n\|_{W^{1,\infty}})^2\kappa(\|u_S^n,u_T^n,\theta_S^n\|_{L^\infty})^2 \\
&\quad\quad  \times \left(1+\|u_S^n\|_{X^{k',2}}^2+\|u_T^n\|_{W^{k',2}}^2+\|\theta_S\|_{W^{k',2}}^2 \right)\mathrm{d}t\bigg)^{q/2}\\
&\quad\leq C\mathbb{E} \sup_{t\in [0,T]} \left(1+\|u_S^n\|_{X^{k',2}}^q+\|u_T^n\|_{W^{k',2}}^q+\|\theta_S\|_{W^{k',2}}^q \right) <\infty,
 \end{split}
\end{equation*}
which combined with \eqref{2.48} yields that
$$
 \left\{\int_0^t\mathcal {D}_{cut,n}^1(s)\mathrm{d}\mathcal {W}_s \right\}_{n\geq1}\quad \mbox{is uniformly bounded in} \quad W^{\vartheta,q}([0,T];X^{k'-1,2}).
$$
Moreover,  it follows from \eqref{2.46} and the uniform bounds for $N_j,j=1,2,3 $ that the sequence $(u_S^n)_{n\geq1}$ is uniformly bounded in $W^{\vartheta,q}([0,T];X^{k'-1,2})$, which also implies that $u_S^n-\int_0^t\mathcal {D}_{cut,n}^1(s)\mathrm{d}\mathcal {W}_s$ is unifirmly bounded in $W^{1,q}([0,T];X^{k'-1,2})$.

By applying the same argument, we can also prove that the sequences $(u_T^n)_{n\geq1}$ and $(\theta_S^n)_{n\geq1}$ are both uniformly bounded in $W^{\vartheta,q}([0,T];W^{k'-1,2})$, and we shall omit the details here for simplicity. This completes the proof of Lemma \ref{lem:2.3}.
\end{proof}

\subsection{Global martingale solutions}

Based on the uniform bounds in Subsection 2.2, we are ready to prove the tightness of the measures with resect to the approximations and construct the global martingales by applying the stochastic compactness method.

To begin with, let us define the path space
$$
\mathscr{X}= \mathscr{X}_\mathbb{S} \times \mathscr{X}_\mathcal {W}, \quad
\mathscr{X}_\mathbb{S}=C([0,T]; \mathcal {Z}^{k'-2,2}(D)), \quad\mathscr{X}_\mathcal {W}=C([0,T]; \mathfrak{A}_0).
$$
And we define the probability measures on $\mathscr{X}$ as
\begin{equation}\label{2.49}
\begin{split}
\mu^n=\mu^n_\mathbb{S}\times \mu^n_\mathcal {W},\quad\mu^n_\mathbb{S}=\mathbb{P}\circ (u_S,u_T,\theta_S)^{-1},\quad \mu^n_\mathcal {W}=\mathbb{P}\circ \mathcal {W}^{-1}.
 \end{split}
\end{equation}

Recalling that a sequence $(\nu_n)_{n\geq1} \subset \mathscr{X}$ is said to converge weakly to an element $\nu\in\mathcal {P}_r(\mathscr{X}) $, if $\int_\mathscr{X} f(x)  \mathrm{d}\nu_n \rightarrow \int_\mathscr{X} f(x)  \mathrm{d}\nu$ for all bounded continuous function $f$ on $\mathscr{X}$ as $n\rightarrow\infty$. A set $\Gamma\in \mathcal {P}_r(\mathscr{X}) $ is said to be weakly compact, if every sequence of $\Gamma$ has a weakly convergent subsequence. Moreover, a collection $\mathscr{O}\subset\mathcal {P}_r(\mathscr{X})$ (which is the set of all Borel probability measure on $\mathscr{X}$) is tight if, for every $\gamma>0$, there exists a compact set $K_\gamma\subset \mathscr{X}$ such that, $\nu(K_\gamma)\geq 1-\gamma$ for all $\nu\in \mathscr{O}$. The following compactness result for fractional Sobolev spaces is useful.

\begin{lemma} [\cite{44}] \label{lem:2.4}
Suppose that $B_0, B$ are Banach spaces, and the embedding from $B_0$ into $  B$ is compact. Let $p \in(1,\infty)$ and $\vartheta\in (0, 1]$ be such that $\vartheta p>1$. Then $W^{\vartheta,p}(0, T;B_0) \subset  C ([0,T];B)$ and the embedding is compact.
\end{lemma}

Based on the Lemma \ref{lem:2.4} and the a priori estimates established in Subsection 2.1-2.2, we can now prove Lemma \ref{lem:2.1}.

\begin{proof} [\bf{\emph{Proof of Lemma \ref{lem:2.1} (Martingale solutions)}}.]

We first claim that the set of probability measures $\{\mu^n;~ n\geq1\} $ defined by \eqref{2.49} is tight on $\mathscr{X}= \mathscr{X}_\mathbb{S} \times \mathscr{X}_\mathcal {W}$.

Indeed, since $D$ is a bounded domain with smooth boundary, the Sobolev embedding from $\mathcal {Z}^{k'-1,2}(D)$ into
$\mathcal {Z}^{k'-2,2}(D)$ is compact. Fix a $\vartheta\in (0,1)$ such that $\vartheta q>1$,
it follows from the Lemma \ref{lem:2.3} that $W^{1,q}([0,T];\mathcal {Z}^{k'-1,2}(D))$
and $W^{\vartheta,q}([0,T];\mathcal {Z}^{k'-1,2}(D))$ are both compactly embedded
into $\mathscr{X}_\mathbb{S}=C([0,T]; \mathcal {Z}^{k'-2,2}(D))$, which implies that the following ball with radius $r>0$
\begin{equation*}
\begin{split}
B_n(r):=&\left\{(u_S,u_T,\theta_S)\in W^{\vartheta,q}([0,T];\mathcal {Z}^{k'-1,2});
~  \|(u_S,u_T,\theta_S)\|_{W^{\vartheta,q}([0,T];\mathcal {Z}^{k'-1,2})}\leq r\right\}\\
&\quad \bigcap \left\{(u_S,u_T,\theta_S)\in W^{1,q}([0,T];\mathcal {Z}^{k'-1,2}); ~ \|(u_S,u_T,\theta_S)\|_{W^{1,q}([0,T];\mathcal {Z}^{k'-1,2})}\leq r\right\}
\end{split}
\end{equation*}
is compact in $\mathscr{X}_\mathbb{S}$. Define
\begin{equation*}
\begin{split}
B_{1,n}(r):=\Bigg\{(u_S,u_T,\theta_S);~~&\bigg\|\int_0^t\mathcal {D}_{cut,n}^1(s)\mathrm{d}\mathcal {W}_s \bigg\|_{W^{\vartheta,q}([0,T];X^{k'-1,2})}^q \\
&  +\sum_{j=2}^3 \bigg\|\int_0^t\mathcal {D}_{cut,n}^j(s)\mathrm{d}\mathcal {W}_s \bigg\|_{W^{\vartheta,q}([0,T];W^{k'-1,2})}^q \leq r   \Bigg\},
\end{split}
\end{equation*}
and
\begin{equation*}
\begin{split}
B_{2,n}(r):=\Bigg\{(u_S,u_T,\theta_S);~~&\bigg\|u_S^n(\cdot)-\int_0^\cdot\mathcal {D}_{cut,n}^1(s)\mathrm{d}\mathcal {W}_s \bigg\|_{W^{1,q}([0,T];X^{k'-1,2})}^q\\
+&\bigg\|u_T^n(\cdot)-\int_0^\cdot\mathcal {D}_{cut,n}^2(s)\mathrm{d}\mathcal {W}_s \bigg\|_{W^{1,q}([0,T];W^{k'-1,2})}^q \\
+&\bigg\|\theta_S^n(\cdot)-\int_0^\cdot\mathcal {D}_{cut,n}^3(s)\mathrm{d}\mathcal {W}_s \bigg\|_{W^{1,q}([0,T];W^{k'-1,2})}^q \leq r \Bigg\}.
\end{split}
\end{equation*}
It is clear that $B_{1,n}(r) \cap B_{2,n}(r)\subset B_{n}(r)$,  by using uniform bounds obtained in Lemma \ref{lem:2.2} and the Chebyshev inequality, we have
\begin{equation*}
\begin{split}
\mu^n_\mathbb{S} (B_n(r)^c) &\leq\mathbb{P} \{(u_S^n,u_T^n,\theta_S^n)\in B_{1,n}(r)^c\}+\mathbb{P} \{(u_S^n,u_T^n,\theta_S^n)\in B_{2,n}(r)^c \}\\
&\leq \frac{C}{r} \sum_{j=1}^3 \sup_{n\geq1}\mathbb{E} \bigg\|\int_0^t\mathcal {D}_{cut,n}^j(s)\mathrm{d}\mathcal {W}_s \bigg\|_{W^{\vartheta,q}([0,T];X^{k'-1,2})}^q\\
&\quad+ \frac{C}{r} \sup_{n\geq1}\mathbb{E}\bigg\|u_S^n(t)-\int_0^t\mathcal {D}_{cut,n}^1(s)\mathrm{d}\mathcal {W}_s \bigg\|_{W^{1,q}([0,T];X^{k'-1,2})}^q\\
 &\quad +\frac{C}{r}\sup_{n\geq1}\mathbb{E}\bigg\|u_T^n(t)-\int_0^t\mathcal {D}_{cut,n}^2(s)\mathrm{d}\mathcal {W}_s \bigg\|_{W^{1,q}([0,T];W^{k'-1,2})}^q \\
& \quad+\frac{C}{r}\sup_{n\geq1}\mathbb{E}\bigg\|\theta_S^n(t)-\int_0^t\mathcal {D}_{cut,n}^3(s)\mathrm{d}\mathcal {W}_s \bigg\|_{W^{1,q}([0,T];W^{k'-1,2})}^q\\
&\leq \frac{C}{r},
\end{split}
\end{equation*}
where $C$ is a positive constant independent of $n$ and $r$. In other words, for any $\gamma>0$, there exists $r>0$ large enough such that
$$
\overline{B_n(r)}\in \mathscr{X}_\mathbb{S} \quad \mbox{and}\quad  \mu^n_\mathbb{S} (\overline{B_n(r)})>1-\gamma,\quad \forall n \in \mathbb{N},
$$
which shows that the sequence of measures $(\mu^n_\mathbb{S} )_{n\geq 1}$ is tight on $\mathscr{X}_\mathbb{S}$. Moreover,  $\mu_\mathcal {W}$ is trivially weakly compact and so is tight on $\mathscr{X}_\mathcal {W}$. As a consequence, the sequence $(\mu^n)_{n\geq 1}$ is tight on $\mathscr{X}$. Since the path space $\mathscr{X}$ is a polish space, the Prokhorov theorem (cf. \cite{45}) asserts that the sequence $(\mu^n )_{n\geq 1}$ is indeed weakly compact on  $\mathscr{X}$.

With the weakly compactness in hand, one can apply the Skorokhod representation theorem (cf. \cite{43}) to a weakly convergent subsequence of $(\mu^n)_{n\geq 1}$. There is a sequence of random elements $\{(\overline{u}_S^n,\overline{u}_T^n,\overline{\theta}_S^n,\overline{\mathcal {W}}^n)\}_{n\geq1}$ defined on a new probability space $(\overline{\Omega},\overline{\mathcal {F}},\overline{\mathbb{P}})$, which converges almost surely in $\mathscr{X}$ to an element $(\overline{u}_S,\overline{u}_T,\overline{\theta}_S,\overline{\mathcal {W}})$, that is,
\begin{equation}\label{2.50}
\begin{split}
\overline{\mathcal {W}}^n\rightarrow \overline{\mathcal {W}} \quad \mbox{in}~~C([0,T]; \mathfrak{A}_0)~~ \mbox{almost surely},
\end{split}
\end{equation}
and
\begin{equation}\label{2.51}
\begin{split}
 (\overline{u}_S^n,\overline{u}_T^n,\overline{\theta}_S^n)\rightarrow (\overline{u}_S,\overline{u}_T,\overline{\theta}_S) \quad \mbox{in}~~C([0,T]; \mathcal {Z}^{k'-2,2}(D))~~ \mbox{almost surely}.
\end{split}
\end{equation}
Following the similar idea in \cite{46}, one can verify that the triple satisfies the Garlekin approximation relative
to the stochastic basis $\mathcal {S}^n:=(\overline{\Omega},\overline{\mathcal {F}},\overline{\mathbb{P}},
\{\overline{\mathcal {F}}_t^n\},\overline{\mathcal {W}}^n)$.
By the assumption $k'=k+4>5$ we get $\mathcal {Z}^{k'-2,2}(D)\hookrightarrow (W^{1,\infty}(D))^3$,
which implies that, as $n\rightarrow\infty$,
\begin{equation}\label{2.52}
\begin{split}
\varpi_R(\|\overline{u}_S^n\|_{X^{1,\infty}})&\longrightarrow \varpi_R(\|\overline{u}_S\|_{X^{1,\infty}}),\\
 \varpi_R(\|(\overline{u}_S^n,\overline{u}_T^n)\|_{X^{1,\infty}\times W^{1,\infty}})&\longrightarrow
\varpi_R(\|(\overline{u}_S,\overline{u}_T)\|_{X^{1,\infty}\times W^{1,\infty}}),\\
 \varpi_R(\|(\overline{u}_S^n,\bar{\theta}_T^n)\|_{X^{1,\infty}\times W^{1,\infty}})&\longrightarrow
\varpi_R(\|(\overline{u}_S,\bar{\theta}_T)\|_{X^{1,\infty}\times W^{1,\infty}}).
\end{split}
\end{equation}
Using the uniform bounds in Lemma \ref{2.2}, the almost sure convergence \eqref{2.50}-\eqref{2.51} and \eqref{2.52}, we can infer that $(\overline{u}_S,\overline{u}_T,\overline{\theta}_S,\overline{\mathcal {W}})$ solves the truncated SISM system. The proof of Lemma \ref{lem:2.1} is now completed.
\end{proof}

\section{Smooth pathwise solutions}
Being inspired by the classic Yamada-Wannabe theorem, we shall show that the local pathwise solution to the SISM system \eqref{1.1}-\eqref{1.3} exists once the global martingale solution $(\overline{u}_S,\overline{u}_T,\overline{\theta}_S)$ in Lemma \ref{lem:2.1} is proved to be ``pathwise uniqueness", which can be stated by the following result.

\begin{lemma} [Pathwise uniqueness] \label{lem:3.1}
Assume that the conditions of the  Lemma \ref{lem:2.1} hold, and $\Phi^{(j)}:=(u_S^{(j)},u_T^{(j)},\theta_S^{(j)})$, $j=1,2$, are two global martingale solutions to the truncated SISM associated to the same stochastic basis $\mathbb{S}=(\Omega,\mathcal {F},\mathbb{P},\{\mathcal {F}_t\}_{t\geq0},\mathcal {W})$. If
$$
\mathbb{P} \left\{\Phi^{(1)}(0)
=\Phi^{(2)}(0)=(u_S^0,u_T^0,\theta_S^0) \right\}=1
$$
with $\mathbb{E}\|(u_S^0,u_T^0,\theta_S^0)\|_{\mathcal {Z}^{k',2}}^r<\infty$ for some  $r>2$,
then the solutions $\Phi^{(1)}$ and $\Phi^{(2)}$ are indistinguishable, that is,
$
\mathbb{P} \{\Phi^{(1)}(t)=\Phi^{(2)}(t);\forall t \geq0 \}=1.
$
\end{lemma}

\begin{proof}
For every $T>0$, it follows from Lemma \ref{lem:2.2} that
\begin{equation}\label{3.1}
\begin{split}
\sup_{n\geq1}\mathbb{E}\sup_{s\in [0,  T]} \left(\|\Phi^{(1)}(s)\|_{\mathcal {Z}^{k'-2,2}(D)}+\|\Phi^{(2)}(s)\|_{\mathcal {Z}^{k'-2,2}(D)}  \right)
\leq C<\infty,
 \end{split}
\end{equation}
where $C>0$ is an universal constant independent of $n$. Notice that the continuity in time is only ensured for $\mathcal {Z}^{k'-2,2}(D)$-norm of $\Phi^{(1)}$ and $ \Phi^{(2)}$, we define the following collection of stopping times
\begin{eqnarray*}
\mathbbm{t}_K=\inf\left\{t>0;~\|\Phi^{(1)}(s)\|_{\mathcal {Z}^{k+1,p}(D)}+\|\Phi^{(2)}(s)\|_{\mathcal {Z}^{k+1,p}(D)} \geq K \right\},
\end{eqnarray*}
for any $K>0$. Since $\mathcal {Z}^{k',2}(D)$ is continuously embedded into $\mathcal {Z}^{k+1,p}(D)$, we see that $\mathbbm{t}_K\rightarrow\infty$ almost surely as $K\rightarrow\infty$.

Setting
\begin{equation*}
U=u_S^{(1)}-u_S^{(2)}, \quad V=u_T^{(1)}-u_T^{(2)},\quad Z=\theta_S^{(1)}-\theta_S^{(2)}.
\end{equation*}
By using the equations with respect to $u_S^{(1)}$ and $u_S^{(2)}$ in \eqref{1.5}, we deduce that
\begin{equation}\label{3.2}
\begin{split}
\mathrm{d}  \partial^\alpha U +  \partial^\alpha \triangle (u_S^{(1)},u_S^{(2)})\mathrm{d}t -  \partial^\alpha P( V\widehat{x}) \mathrm{d}t-   \partial^\alpha P( Z\widehat{z}) \mathrm{d}t = \partial^\alpha \triangle^1 (\Phi^{(1)},\Phi^{(2)})\mathrm{d}\mathcal {W },
 \end{split}
\end{equation}
where
\begin{equation*}
\begin{split}
\triangle (u_S^{(1)},u_S^{(2)})=&\varpi_R(\|u_S^{(1)}\|_{X^{1,\infty}})P(u_S^{(1)}\cdot \nabla)   u_S^{(1)}-\varpi_R(\|u_S^{(2)}\|_{X^{1,\infty}})P(u_S^{(2)}\cdot \nabla)   u_S^{(2)},\\
\triangle^{1} (\Phi^{(1)},\Phi^{(2)})=&\varpi_R(\|\Phi^{(1)}\|_{\mathcal {Z}^{1,\infty}})
P\sigma_1(\Phi^{(1)})-\varpi_R(\|\Phi^{(2)}\|_{\mathcal {Z}^{1,\infty}})
P\sigma_1(\Phi^{(2)}).
 \end{split}
\end{equation*}
Define $\triangle^1_j (\Phi^{(1)},\Phi^{(2)}):=\triangle^1  (\Phi^{(1)},\Phi^{(2)})e_j$, for all $j\geq 1$. After applying the It\^{o}'s formula to $\mathrm{d}| \partial^\alpha U(t)|^p=\mathrm{d}(| \partial^\alpha U(t)|^2)^{p/2}$ and then integrating the resulted equality on $D$, we find
\begin{equation}\label{3.3}
\begin{split}
\mathrm{d}\|U(t)\|^{p} _{X^{k,p}}=&- p\sum_{|\alpha|\leq k}\int_D| \partial^\alpha U(t)|^{p-2} \partial^\alpha U(t)\cdot \partial^\alpha \triangle   (u_S^{(1)},u_S^{(2)})\mathrm{d}V\mathrm{d}t\\
&  + p\sum_{|\alpha|\leq k}\int_D| \partial^\alpha U(t)|^{p-2} \partial^\alpha U(t)\cdot \partial^\alpha (P( V\widehat{x}) + P( Z\widehat{z}) )\mathrm{d}V \mathrm{d}t\\
& +\sum_{|\alpha|\leq k}\sum_{j\geq 1}\bigg(\frac{p}{2}\int_D| \partial^\alpha U(t)|^{p-2}\left( \partial^\alpha \triangle_j ^1 (\Phi^{(1)},\Phi^{(2)})\right)^2\mathrm{d}V
 \\
 &\quad\quad \quad\quad \quad+ \frac{p(p-2)}{2} | \partial^\alpha U(t)|^{p-4} \left( \partial^\alpha U(t)\cdot \partial^\alpha \triangle_j ^1  (\Phi^{(1)},\Phi^{(2)})\right)^2 \mathrm{d}V\bigg)\mathrm{d}t\\
& +p\sum_{|\alpha|\leq k}\sum_{j\geq 1}\int_D| \partial^\alpha U(t)|^{p-2} \partial^\alpha U(t)\cdot \partial^\alpha \triangle _j ^1 (\Phi^{(1)},\Phi^{(2)})\mathrm{d}V\mathrm{d}\beta_j \\
:= & \sum_{|\alpha|\leq k}(\mathcal {J}^{V,\alpha}_1(t)+\mathcal {J}^{V,\alpha}_2(t)+\mathcal {J}^{V,\alpha}_3(t))\mathrm{d} t+\sum_{|\alpha|\leq k}\mathcal {J}^{V,\alpha}_4(t)\mathrm{d}\mathcal {W}.
 \end{split}
\end{equation}
For $\mathcal {J}^{V,\alpha}_1(t)$, we first use the mean value theorem and the property of $\varpi_R(\cdot)$ to derive that
\begin{equation*}
\begin{split}
\left|\varpi_R(\|u_S^{(1)}\|_{X^{1,\infty}}) -\varpi_R(\|u_S^{(2)}\|_{X^{1,\infty}})\right|\leq C\|U \|_{X^{1,\infty}}\leq C\|U \|_{X^{k,p}},
 \end{split}
\end{equation*}
where we used the fact of $\mathcal {Z}^{k,p}(D)\hookrightarrow \mathcal {Z}^{1,\infty}(D)$. It then follows from the Moser-type estimates, the H\"{o}lder inequality and the definition of the operator $P$ that
\begin{equation}\label{3.4}
\begin{split}
|\mathcal {J}^{V,\alpha}_1(t)|\leq& C\|U\|_{W^{1,\infty}}\int_D| \partial^\alpha U|^{p-1}|  \partial^\alpha P(u_S^{(1)}\cdot \nabla)   u_S^{(1)}|\mathrm{d}V\\
& + p\varpi_R(\|u_S^{(2)}\|_{X^{1,\infty}})\int_D| \partial^\alpha U|^{p-2}| \partial^\alpha U\cdot  \partial^\alpha P(U\cdot \nabla)   u_S^{(1)}|\mathrm{d}V\\
& +p\varpi_R(\|u_S^{(2)}\|_{X^{1,\infty}})\int_D| \partial^\alpha U|^{p-2}| \partial^\alpha U\cdot  \partial^\alpha P(u_S^{(2)}\cdot \nabla)U|\mathrm{d}V\\
\leq& C\|U\|_{X^{1,\infty}}\| \partial^\alpha U\|_{L^p}^{p-1}\|  \partial^\alpha P(u_S^{(1)}\cdot \nabla)   u_S^{(1)}\|_{L^p}\\
& + p\varpi_R(\|u_S^{(2)}\|_{X^{1,\infty}})\| \partial^\alpha U\|_{L^p}^{p-1}\|  \partial^\alpha P(U\cdot \nabla)   u_S^{(1)}\|_{L^p}\\
& + p\varpi_R(\|u_S^{(2)}\|_{X^{1,\infty}})\| \partial^\alpha U\|_{L^p}^{p-1}\|  \partial^\alpha P(u_S^{(2)}\cdot \nabla)U\|_{L^p}\\
\leq& C \| U\|_{X^{k,p}}^{p}(\|u_S^{(1)}\|_{L^{\infty}}\| u_S^{(1)}\|_{X^{k+1,p}}+\| u_S^{(1)}\|_{X^{1,\infty}}\| u_S^{(1)}\|_{X^{k,p}})\\
& + C\| U\|_{X^{k,p}}^{p-1}(\|U\|_{L^{\infty}}\| u_S^{(1)}\|_{X^{k+1,p}}+\| u_S^{(1)}\|_{X^{1,\infty}}\| U\|_{X^{k,p}})\\
& + p\varpi_R(\|u_S^{(2)}\|_{X^{1,\infty}})\| U\|_{X^{k,p}}^{p-1}(\|u_S^{(2)}\|_{W^{1,\infty}}\|U\|_{X^{k,p}}+\|\nabla U\|_{L^{\infty}}\| u_S^{(2)}\|_{X^{k,p}})\\
\leq & C \| U\|_{X^{k,p}}^{p} \Big(1+\| u_S^{(1)}\|_{X^{k+1,p}} \| u_S^{(1)}\|_{X^{k,p}}+ \| u_S^{(1)}\|_{X^{k+1,p}}\\
& +\| u_S^{(1)}\|_{X^{k,p}}+\| u_S^{(2)}\|_{X^{k,p}}\Big).
 \end{split}
\end{equation}
Using the Young inequality, we can estimate $\mathcal {J}^{V,\alpha}_2(s)$ as
\begin{equation}\label{3.5}
\begin{split}
&\mathbb{E}\sup_{s\in [0,t\wedge \mathbbm{t}_K]}\int_0^s|\mathcal {J}^{V,\alpha}_2(r)|dr\\
&\quad\leq C\mathbb{E} \int_0^{t\wedge \mathbbm{t}_K}\| U\|_{W^{k,p}}^{p-1} \| \partial^\alpha (P( V\widehat{x}) + P( Z\widehat{z}) )\|_{L^p}\mathrm{d}s\\
&\quad\leq \frac{1}{4}\mathbb{E}\sup_{s\in [0,t]}\| U(s)\|_{X^{k,p}}^p+C\mathbb{E}\int_0^{t\wedge \mathbbm{t}_K}(\| V(t)\|_{W^{k,p}}^p+\| Z(t)\|_{W^{k,p}}^p)\mathrm{d}s.
 \end{split}
\end{equation}
Due to the locally boundedness and Lipschitz conditions for diffusions, we can estimate the third term as
\begin{equation}\label{3.6}
\begin{split}
|\mathcal {J}^{V,\alpha}_3(t)| \leq & C \| \partial^\alpha U(t)\|_{L^p}^{p-2}  \sum_{j\geq 1}\| \partial^\alpha \triangle_j ^1 (\Phi^{(1)},\Phi^{(2)})\|_{\mathbb{W}^{0,p}}^2\\
\leq& C \| \partial^\alpha U(t)\|_{L^p}^{p-2}\Big( \|\Phi^{(1)}-\Phi^{(2)}\|_{W^{1,\infty}} ^2\|
P\sigma_1(\Phi^{(1)})\|_{\mathbb{X}^{k,p}}^2\\
&+ \varpi_R(\|\Phi^{(2)}\|_{W^{1,\infty}}) \|
P\sigma_1(\Phi^{(1)})-P\sigma_1(\Phi^{(2)})\|_{\mathbb{X}^{k,p}}^2\Big)\\
\leq& C \| U\|_{W^{k,p}}^{p-2}\|
\Phi^{(1)}-\Phi^{(2)}\|_{\mathcal {Z}^{k,p}}^2\Big( \kappa(\|\Phi^{(1)}\|_{L^\infty})^2 (1+\|\Phi^{(1)}\|_{\mathcal {Z}^{k,p}}^2)\\
& +  \varsigma(\|\Phi^{(1)}\|_{L^\infty}+\|\Phi^{(2)}\|_{L^\infty})^2\Big).
 \end{split}
\end{equation}
By utilizing the BDG inequality and the definition of the stopping time $\mathbbm{t}_K$, we can estimate the  stochastic term as
\begin{equation}\label{3.7}
\begin{split}
&\mathbb{E}\sup_{s\in [0,t\wedge \mathbbm{t}_K]}\int_0^s|\mathcal {J}^{V,\alpha}_4(r)|\mathrm{d}\mathcal {W}_r\\
 &\quad \leq C\mathbb{E}\Bigg(\int_0^{t\wedge \mathbbm{t}_K}\sum_{j\geq 1}\bigg(\int_D| \partial^\alpha U(t)|^{p-2} \partial^\alpha U(t)\cdot \partial^\alpha \triangle _j  ^1 (\Phi^{(1)},\Phi^{(2)})\mathrm{d}V\bigg)^2\mathrm{d}  s\Bigg)^{1/2}\\
 &\quad \leq C\mathbb{E}\bigg(\int_0^{t\wedge \mathbbm{t}_K}\| \partial^\alpha U(s)\|_{L^p}^{2p-2} \|  \partial^\alpha \triangle ^1  (\Phi^{(1)},\Phi^{(2)})\|_{\mathbb{X}^{0,p}}^2 \mathrm{d}  s\bigg)^{1/2}\\
 &\quad \leq C\mathbb{E}\Bigg(\sup_{s\in [0,t]}(\|U(s)\|_{X^{k,p}}^{p-1}\|
\Phi^{(1)}-\Phi^{(2)}\|_{\mathcal {Z}^{k,p}})\int_0^{t\wedge \mathbbm{t}_K}  \|U(s)\|_{X^{k,p}}^{p-1}\|
\Phi^{(1)}-\Phi^{(2)}\|_{\mathcal {Z}^{k,p}} \\
& \quad\quad \times\Big( \kappa(\|\Phi^{(1)}\|_{L^\infty})^2 (1+\|\Phi^{(1)}\|_{\mathcal {Z}^{k,p}}^2)+  \varsigma(\|\Phi^{(1)}\|_{L^\infty}+\|\Phi^{(2)}\|_{L^\infty})^2\Big) \mathrm{d}  s\Bigg)^{1/2}\\
 &\quad \leq \frac{1}{4}\mathbb{E} \sup_{s\in [0,t\wedge \mathbbm{t}_K]}(\|U(s)\|_{X^{k,p}}^{p-1}\|
\Phi^{(1)}-\Phi^{(2)}\|_{\mathcal {Z}^{k,p}})\\
&\quad\quad+ C\mathbb{E}\int_0^{t\wedge \mathbbm{t}_K}  \|U(s)\|_{X^{k,p}}^{p-1}\|
\Phi^{(1)}-\Phi^{(2)}\|_{\mathcal {Z}^{k,p}} \mathrm{d}  s.
 \end{split}
\end{equation}
Plugging the estimates \eqref{3.4}-\eqref{3.7} into \eqref{3.3}, we find that for any fixed $K>0$
\begin{equation}\label{3.8}
\begin{split}
&\frac{3}{4}\mathbb{E}\sup_{s\in [0,t\wedge \mathbbm{t}_K]}\|U(s)\|^{p} _{X^{k,p}}\\
&\quad \leq  \frac{1}{4}\mathbb{E} \sup_{s\in [0,t\wedge \mathbbm{t}_K]}\|
\Phi^{(1)}-\Phi^{(2)}\|_{\mathcal {Z}^{k,p}}^p + C\mathbb{E}\int_0^{t\wedge \mathbbm{t}_K}\|
\Phi^{(1)}-\Phi^{(2)}\|_{\mathcal {Z}^{k,p}}^p \mathrm{d}  s.
 \end{split}
\end{equation}
In order to close above inequality, we have to derive the $L^p$ bounds  for $V(s)$ and $Z(s)$, respectively. Indeed, the component $V(s)$ satisfies the following equation
\begin{equation}\label{3.9}
\begin{split}
V(t)=-\int_0^t\triangle (u_S^{(1)},u_S^{(2)},u_T^{(1)},u_T^{(2)})\mathrm{d}s- \int_0^tU(s)\cdot\widehat{x} \mathrm{d}s +\int_0^t\triangle ^2(\Phi^{(1)},\Phi^{(2)})\mathrm{d}\mathcal {W}_s,
 \end{split}
\end{equation}
where
\begin{equation*}
\begin{split}
\triangle (u_S^{(1)},u_S^{(2)},u_T^{(1)},u_T^{(2)})=&\left(\varpi_R(\|(u_S^{(1)},u_T^{(1)})\|_{X^{1,\infty}\times W^{1,\infty}})
-\varpi_R(\|(u_S^{(2)},u_T^{(2)})\|_{X^{1,\infty}\times W^{1,\infty}})\right)(u_S^{(1)}\cdot \nabla)   u_T^{(1)}\\
& +\varpi_R(\|(u_S^{(2)},u_T^{(2)})\|_{X^{1,\infty}\times W^{1,\infty}})(U\cdot \nabla) u_T^{(1)} \\ &+\varpi_R(\|(u_S^{(2)},u_T^{(2)})\|_{X^{1,\infty}\times W^{1,\infty}})(u_S^{(2)}\cdot \nabla) V,\\
\triangle ^2(\Phi^{(1)},\Phi^{(2)})=&\left(\varpi_R(\|\Phi^{(1)}\|_{W^{1,\infty}})
 -\varpi_R(\|\Phi^{(2)}\|_{\mathcal {Z}^{1,\infty}})\right)
\sigma_2(\Phi^{(1)})\\
& +\varpi_R(\|\Phi^{(2)}\|_{\mathcal {Z}^{1,\infty}})\left(
\sigma_2(\Phi^{(1)})-\sigma_2(\Phi^{(2)})\right).
 \end{split}
\end{equation*}
Define $\triangle^2_j (\Phi^{(1)},\Phi^{(2)}):=\triangle^2  (\Phi^{(1)},\Phi^{(2)})e_j$, for all $j\geq 1$. Applying the It\^{o}'s formula in $L^p$ space to  \eqref{3.9}, we get
\begin{equation}\label{3.10}
\begin{split}
\|V(t) \|_{W^{k,p}}^p=& -p \sum_{|\alpha|\leq k}\int_0^t\int_D| \partial^\alpha V|^{p-2} \partial^\alpha V \partial^\alpha \triangle (u_S^{(1)},u_S^{(2)},u_T^{(1)},u_T^{(2)})\mathrm{d} V\mathrm{d} s\\
&-p\sum_{|\alpha|\leq k} \int_0^t\int_D| \partial^\alpha V|^{p-2} \partial^\alpha V \partial^\alpha U(s)\cdot\widehat{x}\mathrm{d} V\mathrm{d} s\\
&+\frac{p(p-2)}{2} \sum_{|\alpha|\leq k}\sum_{j\geq 1}\int_0^t\int_D| \partial^\alpha V|^{p-2}( \partial^\alpha \triangle_j ^2(\Phi^{(1)},\Phi^{(2)}))^2\mathrm{d} V\mathrm{d} s\\
&+p \sum_{|\alpha|\leq k}\sum_{j\geq 1}\int_0^t\int_D| \partial^\alpha V|^{p-2} \partial^\alpha V \partial^\alpha \triangle_j ^2(\Phi^{(1)},\Phi^{(2)})\mathrm{d} V\mathrm{d}\mathcal \beta_j\\
&:=\sum_{|\alpha|\leq k}(\mathcal {J}^{U,\alpha}_1(t)+\mathcal {J}^{U,\alpha}_2(t)+\mathcal {J}^{U,\alpha}_3(t)+\mathcal {J}^{U,\alpha}_4(t)).
 \end{split}
\end{equation}
An application of the Lagrange's mean value theorem to the functions $\varpi_R(\|u_S^{(i)},u_T^{(i)}\|_{W^{1,\infty}})$, $i=1,2,$ we deduce from the property of the cut-off function that
\begin{equation*}
\begin{split}
\left|\varpi_R(\|(u_S^{(1)},u_T^{(1)})\|_{X^{1,\infty}\times W^{1,\infty}})
-\varpi_R(\|(u_S^{(2)},u_T^{(2)})\|_{X^{1,\infty}\times W^{1,\infty}})\right|&\leq \sup_{x\in [0,\infty)}|\varpi'(x)|\|(U,V)\|_{X^{1,\infty}\times W^{1,\infty}}\\
&\leq C(\|U\|_{X^{k,p}}+\|V\|_{W^{k,p}}),
 \end{split}
\end{equation*}
\begin{equation*}
\begin{split}
\left|\varpi_R(\|\Phi^{(1)}\|_{\mathcal {Z}^{1,\infty}})
 -\varpi_R(\|\Phi^{(2)}\|_{\mathcal {Z}^{1,\infty}})\right|\leq C\|\Phi^{(1)}-\Phi^{(2)}\|_{\mathcal {Z}^{1,\infty}}\leq C\|\Phi^{(1)}-\Phi^{(2)}\|_{\mathcal {Z}^{k,p}}.
 \end{split}
\end{equation*}
It follows from the Young inequality and the Moser-type estimate  that
\begin{equation}\label{3.11}
\begin{split}
\mathcal {J}^{U,\alpha}_1({t\wedge \mathbbm{t}_K} ) \leq &C\int_0^{t\wedge \mathbbm{t}_K} \| V\|_{W^{k,p}}^{p-1}(\|U\|_{X^{k,p}}+\|V\|_{W^{k,p}})\|(u_S^{(1)}\cdot \nabla)   u_T^{(1)}\|_{W^{k,p}}\mathrm{d} s\\
&+C\int_0^{t\wedge \mathbbm{t}_K} \| V\|_{W^{k,p}}^{p-1}\|(U\cdot \nabla) u_T^{(1)}\|_{W^{k,p}}\mathrm{d} s\\
&+C\int_0^{t\wedge \mathbbm{t}_K} \varpi_R(\|u_S^{(2)},u_T^{(2)}\|_{W^{1,\infty}})\| V\|_{W^{k,p}}^{p-1}\|(u_S^{(2)}\cdot \nabla) V\|_{W^{k,p}}\mathrm{d} s\\
\leq &C\int_0^{t\wedge \mathbbm{t}_K}  (\|U\|_{X^{k,p}}^p+\|V\|_{W^{k,p}}^p)(\| u_S^{(1)} \|_{L^\infty}\| u_T^{(1)} \|_{W^{k+1,p}}+\| \nabla u_T^{(1)} \|_{L^\infty}\| u_S^{(1)} \|_{W^{k,p}})\mathrm{d} s\\
&+C\int_0^{t\wedge \mathbbm{t}_K}  \| V\|_{W^{k,p}}^{p-1}( \| V\|_{W^{k,p}}+\|V \|_{L^\infty}\| u_S^{(2)} \|_{W^{k,p}})\mathrm{d} s\\
&+C\int_0^{t\wedge \mathbbm{t}_K} \| V\|_{W^{k,p}}^{p-1}(\| U \|_{L^\infty}\| u_T^{(1)} \|_{W^{k+1,p}}+\| \nabla u_T^{(1)} \|_{L^\infty}\| U \|_{W^{k,p}})\mathrm{d} s\\
\leq& C\int_0^{t\wedge \mathbbm{t}_K}  (\|U(s)\|_{X^{k,p}}^p+\|V(s)\|_{W^{k,p}}^p) \mathrm{d} s.
 \end{split}
\end{equation}
The term $\mathcal {J}^{U,\alpha}_2(t)$ can be treated as $\mathcal {J}^{V,\alpha}_2(t)$, and we obtain
\begin{equation}\label{3.12}
\begin{split}
\mathcal {J}^{U,\alpha}_2(t)\leq&C \int_0^t\| V\|_{W^{k,p}}^{p-1}\|U\|_{W^{k,p}} \mathrm{d} s
\leq \frac{1}{4} \sup_{s\in [0,t]}\| V(s)\|_{W^{k,p}} +\int_0^t\|U(s)\|_{W^{k,p}}^p \mathrm{d} s.
 \end{split}
\end{equation}
By using the locally boundness and the locally Lipschitz condition (A1), we can estimate the third term  $\mathcal {J}^{U,\alpha}_3$ as
\begin{equation}\label{3.13}
\begin{split}
\mathcal {J}^{U,\alpha}_3(t\wedge \mathbbm{t}_K) \leq& C \int_0^{t\wedge \mathbbm{t}_K}\|  \partial^\alpha V\|_{L^p}^{p-2}\Bigg(\int_D\bigg(\sum_{j\geq 1}( \partial^\alpha \triangle_j ^2(\Phi^{(1)},\Phi^{(2)}))^2\bigg)^{p/2}\mathrm{d} V\Bigg)^{2/p}\mathrm{d} s\\
\leq & C \int_0^{t\wedge \mathbbm{t}_K}\| V\|_{W^{k,p}}^{p-2}\|\Phi^{(1)}-\Phi^{(2)}\|_{W^{k,p}}^2\| \sigma_2(\Phi^{(1)})\|^{2}_{\mathbb{W}^{k,p}}\mathrm{d} s\\
&+C \int_0^{t\wedge \mathbbm{t}_K} \| V\|_{W^{k,p}}^{p-2}\|
\sigma_2(\Phi^{(1)})-\sigma_2(\Phi^{(2)})\|^{2}_{\mathbb{W}^{k,p}} \mathrm{d} s\\
\leq & C \int_0^{t\wedge \mathbbm{t}_K}\| V\|_{W^{k,p}}^{p-2}\|\Phi^{(1)}-\Phi^{(2)}\|_{\mathcal {Z}^{k,p}} ^2 \kappa(\|\Phi^{(1)}\|_{\mathcal {Z}^{0,\infty}})^{2}(1+\|\Phi^{(1)}\|_{\mathcal {Z}^{k,p}})\mathrm{d} s\\
&+C \int_0^{t\wedge \mathbbm{t}_K}\| V\|_{W^{k,p}}^{p-2}\varsigma(\|\Phi^{(1)}\|_{\mathcal {Z}^{0,\infty}}+\|\Phi^{(2)}\|_{\mathcal {Z}^{0,\infty}})^2
\|\Phi^{(1)}-\Phi^{(2)}\|^{2}_{\mathcal {Z}^{k,p}} \mathrm{d} s\\
\leq & C \int_0^{t\wedge \mathbbm{t}_K}\| V\|_{W^{k,p}}^{p-2}
\|\Phi^{(1)}-\Phi^{(2)}\|^{2}_{\mathcal {Z}^{k,p}} \mathrm{d} s.
 \end{split}
\end{equation}
For the stochastic term, we make use of the BDG inequality and argue as in \eqref{3.13} to deduce that
\begin{equation}\label{3.14}
\begin{split}
\mathbb{E}\sup_{s\in [0,t\wedge \mathbbm{t}_K]}|\mathcal {J}^{U,\alpha}_4(s)|\leq & C\mathbb{E}\bigg(\int_0^{t\wedge \mathbbm{t}_K}\| \partial^\alpha V(s)\|_{L^p}^{2p-2} \|  \partial^\alpha \triangle ^2  (\Phi^{(1)},\Phi^{(2)})\|_{\mathbb{W}^{0,p}}^2 \mathrm{d}  s\bigg)^{1/2}\\
\leq & C \mathbb{E}\int_0^{t\wedge \mathbbm{t}_K}\| V\|_{W^{k,p}}^{2p-2}
\|\Phi^{(1)}-\Phi^{(2)}\|^{2}_{\mathcal {Z}^{k,p}} \mathrm{d} s\\
\leq & \frac{1}{4} \mathbb{E}\sup_{s\in [0,t\wedge \mathbbm{t}_K]}(\| V\|_{W^{k,p}}^{p-1}
\|\Phi^{(1)}-\Phi^{(2)}\|_{\mathcal {Z}^{k,p}} ) \\
 &+    C \mathbb{E}\int_0^{t\wedge \mathbbm{t}_K}\| V\|_{W^{k,p}}^{p-1}
\|\Phi^{(1)}-\Phi^{(2)}\|_{\mathcal {Z}^{k,p}} \mathrm{d} s.
 \end{split}
\end{equation}
Combining the estimates obtained in \eqref{3.11}-\eqref{3.14}, taking the supremum on $[0,t\wedge \mathbbm{t}_K]$ for \eqref{3.10} and then taking the expectation, we find that for any $K>0$,
\begin{equation}\label{3.15}
\begin{split}
\frac{3}{4}\mathbb{E}\sup_{s\in [0,t\wedge \mathbbm{t}_K]}\|V(t) \|_{W^{k,p}}^p\leq& C\int_0^{t\wedge \mathbbm{t}_K}\|\Phi^{(1)}-\Phi^{(2)}\|^{p}_{\mathcal {Z}^{k,p}} \mathrm{d} s\\
&+\frac{1}{4} \mathbb{E}\sup_{s\in [0,t\wedge \mathbbm{t}_K]}\left(\| V(s)\|_{W^{k,p}}^{p-1}
\|\Phi^{(1)}-\Phi^{(2)}\|_{\mathcal {Z}^{k,p}} \right) .
 \end{split}
\end{equation}
The scalar component $\theta_S(t)$ can be treated similar to the estimate of $u_T(t)$. In this case, we shall use the following estimate for cut-off function
\begin{equation*}
\begin{split}
\left|\varpi_R(\|(u_S^{(1)},\theta_S^{(1)})\|_{X^{1,\infty}\times W^{1,\infty}})
-\varpi_R(\|u_S^{(2)},\theta_S^{(2)}\|_{X^{1,\infty}\times W^{1,\infty}})\right|\leq  C\|\Phi^{(1)}-\Phi^{(2)}\|_{\mathcal {Z}^{k,p}}.
 \end{split}
\end{equation*}
After some calculations, we can finally obtain  for any $K>0$,
\begin{equation}\label{3.16}
\begin{split}
&\frac{3}{4}\mathbb{E}\sup_{s\in [0,t\wedge \mathbbm{t}_K]}\|Z(t) \|_{W^{k,p}}^p\\
&\quad \leq  C\mathbb{E}\int_0^{t\wedge \mathbbm{t}_K}\|\Phi^{(1)}-\Phi^{(2)}\|^{p}_{\mathcal {Z}^{k,p}} \mathrm{d} s +\frac{1}{4} \mathbb{E}\sup_{s\in [0,t\wedge \mathbbm{t}_K]}
\|\Phi^{(1)}-\Phi^{(2)}\|_{\mathcal {Z}^{k,p}}^p .
 \end{split}
\end{equation}
Putting the estimates \eqref{3.8}, \eqref{3.15} and \eqref{3.16} together, we get
\begin{equation*}
\begin{split}
 \mathbb{E}\sup_{s\in [0,t\wedge \mathbbm{t}_K]}\|\Phi^{(1)}(s)-\Phi^{(2)}(s)\|_{\mathcal {Z}^{k,p}}^p\leq 2C\mathbb{E}\int_0^{t }\mathbb{E}\sup_{r\in [0,s\wedge \mathbbm{t}_K]}\|\Phi^{(1)}(r)-\Phi^{(2)}(r)\|_{\mathcal {Z}^{k,p}}^p \mathrm{d} s,
 \end{split}
\end{equation*}
which combined with the Gronwall inequality lead to
\begin{equation*}
\begin{split}
\mathbb{E}\sup_{s\in [0,t\wedge \mathbbm{t}_K]}\|\Phi^{(1)}(s)-\Phi^{(2)}(s)\|_{\mathcal {Z}^{k,p}}^p\leq 0.
 \end{split}
\end{equation*}
By taking $K\rightarrow\infty$, we infer  that for any $T>0$,
\begin{equation*}
\begin{split}
\mathbb{E}\sup_{s\in [0,T]}\|\Phi^{(1)}(s)-\Phi^{(2)}(s)\|_{\mathcal {Z}^{k,p}}^p=0,
 \end{split}
\end{equation*}
which implies the pathwise uniqueness of the martingale solutions, and this completes the proof of Lemma \ref{3.1}.
\end{proof}

With the pathwise uniqueness result at hand, we could now establish the existence of local pathwise solutions in sufficient regular spaces, whose proof is based on the following lemma.

\begin{lemma} [\cite{47}] \label{lem:3.2}
Let $X$ be a Polish space, and $\{Y_j \}_{j\geq0}$ be a
sequence of $X$-valued random elements. We define the collection of joint
laws $\{\nu_{j,l}\}_{j,l\geq1}$ of $\{Y_j \}_{j\geq1}$ by
$$
\nu_{j,l}(E)=\mathbb{P}\{(Y_j,Y_l)\in E\},~~~E\in \mathcal {B}(X\times X).
$$
Then $\{Y_j \}_{j\geq0}$ converges in probability if and only if for every subsequence $\{\nu_{j_k,l_k}\}_{k\geq0}$, there exists a further subsequence which converges
weakly to a probability measure $\nu$ such that
$$
\nu(\{(u,v)\in X \times X:u=v\})=1.
$$
\end{lemma}

The main result in this section can be stated as follows.
\begin{lemma} [Smooth pathwise solutions] \label{lem:3.3}
Fix $k'=k+4$, $k>1+\frac{2}{p}$ and $p\geq2$.  Assume that the conditions (A1)-(A2) hold, and $(u_S^0,u_T^0,\theta_S^0)$ is a $\mathcal {Z}^{k',2}(D)$-valued $\mathcal {F}_0$-measurable random processes. Then the Cauchy problem \eqref{1.1}-\eqref{1.3} has a unique local maximum pathwise solution $(u_S,u_T,\theta_S, \{\mathbbm{t}_n\}_{n\geq1},\mathbbm{t})$ in the sense of Definition \ref{def:1.1}.
\end{lemma}

\begin{proof}
The proof of Lemma \ref{3.3} will be divided into two steps.

\textbf{\emph{Step 1.}} We shall prove that the truncated SISM admits a global pathwise solution. To this end, let $\Psi_n:=(u_S^n,u_T^n,\theta_S^n)$ be the sequence of solutions to the system in $\mathcal {Z}^{k',2}(D)$ corresponding to the stochastic basis $\mathbb{S}$ fixed in advance.

Define a sequence of measures
\begin{equation*}
\begin{split}
&\nu_{j,l}(E)=\mathbb{P}\{(\Psi_{j},\Psi_{l})\in E\},\quad \pi_{j,l}(E')=\mathbb{P}\{(\Psi_{j},\Psi_{l},\mathcal {W})\in E'\},
\end{split}
\end{equation*}
for any $E\in \mathcal {B}(\mathscr{X}_\mathbb{S}\times\mathscr{X}_\mathbb{S})$ and $E'\in \mathcal {B}(\mathscr{X}_\mathbb{S}\times\mathscr{X}_\mathbb{S}\times\mathscr{X}_\mathcal {W} )$, where
$$
\mathscr{X}_\mathbb{S}=C([0,T]; \mathcal {Z}^{k'-2,2}(D)), \quad \mathscr{X}_\mathcal {W}=C([0,T]; \mathfrak{A}_0).
$$
Applying arguments in Lemma \ref{2.2}, one can verify  that the collection $\{\pi_{j,l}\}_{j,l\geq1}$ is  weakly compact, and hence there is a subsequence (still denoted by itself) such that
$$
\pi_{j,l}\rightarrow \pi\in \mathcal {P}_r(\mathscr{X}_\mathbb{S}\times\mathscr{X}_\mathbb{S}\times\mathscr{X}_\mathcal {W}).
$$
By the Skorokhod's representation theorem, one can choose a probability space $(\widetilde{\Omega},\widetilde{\mathcal {F}},\widetilde{\mathbb{P}})$ on which the random elements $(\widetilde{\Psi}_{j},\widetilde{\Psi}_{l},\widetilde{\mathcal {W}}^{j,l})$ are defined such that $\pi_{j,l}=\widetilde{\mathbb{P}}\circ (\widetilde{\Psi}_{j},\widetilde{\Psi}_{l},\widetilde{\mathcal {W}}^{j,l})^{-1}$, and
$$
(\widetilde{\Psi}_{j},\widetilde{\Psi}_{l},\widetilde{\mathcal {W}}^{j,l})\rightarrow (\widetilde{\Psi},\widetilde{\Psi}^*,\widetilde{\mathcal {W}})\quad \mbox{in}~~~ \mathscr{X}_\mathbb{S}\times\mathscr{X}_\mathbb{S}\times\mathscr{X}_\mathcal {W},\quad\mathbb{P}\mbox{-a.s.}
$$
As a result we get the weakly convergence $\nu_{j,l}\rightharpoonup \nu=\widetilde{\mathbb{P}}\circ (\widetilde{\Psi},\widetilde{\Psi}^*)^{-1}$. Using the argument as before, we find that $\widetilde{\Psi}$ and $\widetilde{\Psi}^*$ both are global martingale solutions to the truncated SISM associated to the same stochastic basis. Since $\widetilde{\Psi}(0)=\widetilde{\Psi}^*(0)$, it follows from the pathwise uniqueness result (cf. Lemma \ref{lem:3.1}) that
$$
(\widetilde{\Psi},\widetilde{\Psi}^*)\in \mathscr{X}_\mathbb{S}\times\mathscr{X}_\mathbb{S} \quad\mbox{and}\quad \widetilde{\Psi}=\widetilde{\Psi}^*,\quad \nu\mbox{-a.s.}
$$
Thanks to the  Lemma \ref{lem:3.2}, we see that, on the original probability space $(\Omega,\mathcal {F},\mathbb{P})$,
$$
(u_S^n,u_T^n,\theta_S^n)\rightarrow (u_S,u_T,\theta_S) \quad  \mbox{in the topology of}~ \mathscr{X}_\mathbb{S},\quad \mathbb{P}\mbox{-a.s.},
$$
which together with the uniform bound in Lemma \ref{lem:2.3} indicate that $(u_S,u_T,\theta_S)$ is a pathwise solution to \eqref{1.5}.

\textbf{\emph{Step 2.}} By introducing certain stopping time and removing the cut-off function, we prove that the original SISM system has a local pathwise solution. Let us consider the following stopping time
$$
\mathbbm{t}_M=\inf\{t\geq 0;~\|(u_S,u_T,\theta_S)\|_{\mathcal {Z}^{k',2}}> M\},
$$
for any $M>0$. Let $\widehat{c}>0$ be a constant such that $\|(u_S,u_T,\theta_S)\|_{\mathcal {Z}^{1,\infty}}\leq \widehat{c}\|(u_S,u_T,\theta_S)\|_{\mathcal {Z}^{k',2}}$.

$\bullet$ If $\|(u_S,u_T,\theta_S)\|_{\mathcal {Z}^{k',2}}\leq M$ uniformly for some deterministic constant $M>0$, then $\mathbbm{t}_M$ is strictly positive almost surely, and $\|(u_S,u_T,\theta_S)\|_{\mathcal {Z}^{1,\infty}}\leq R$ for any $[0,\mathbbm{t}_M]$ with some $R>\widehat{c} M$, which indicates that
\begin{equation*}
\begin{split}
\varpi_R(\|u_S\|_{W^{1,\infty}})&=\varpi_R(\|(u_S,u_T)\|_{X^{1,\infty}\times W^{1,\infty}})
\\
&=\varpi_R(\|(u_S,\theta_S)\|_{X^{1,\infty}\times W^{1,\infty}})=\varpi_R(\|(u_S,u_T,\theta_S)\|_{\mathcal {Z}^{1,\infty}})=1,
\end{split}
\end{equation*}
for any $t\in [0,\mathbbm{t}_M]$. Thereby  the quadruple $(u_S,u_T,\theta_S,\mathbbm{t}_M)$ is a local pathwise solution to the SISM system \eqref{1.1}-\eqref{1.3}.

$\bullet$ For the general case, we define the decomposition function
$$
\chi_k(\omega):= \mathbb{I}_{\{\|(u_S^0,u_T^0,\theta_S^0)\|_{\mathcal {Z}^{k',2}}\in [k,k+1)\}}(\omega),\quad  k=0,1,2,....
$$
It follows from
previous conclusion that, for given truncated initial data $(u_S^0\chi_k,u_T^0\chi_k,\theta_S^0\chi_k)$, the SISM  \eqref{1.1}-\eqref{1.3} admits  a unique local pathwise solution $(u_S^k,u_T^k,\theta_S^k,\mathbbm{t}_M^k)$. Let us construct the following processes
$$
u_S=\sum_{k\geq 1}u_S^k\chi_k,~~u_T=\sum_{k\geq 1}u_T^k\chi_k,~~\theta_S=\sum_{k\geq 1}\theta_S^k\chi_k,~~\mathbbm{t}_M=\sum_{k\geq 1}\mathbbm{t}_M^k\chi_k.
$$
Using the similar argument as \cite{37}, we can verify that the quadruple $(u_S,u_T,\theta_S,\mathbbm{t}_M)$ is indeed a local pathwise solution to  the SISM \eqref{1.1}-\eqref{1.3} with sufficiently regular initial data in the sense of Definition \ref{def:1.1}. This completes the proof of Lemma \ref{lem:3.3}.
\end{proof}

\section{Local solutions in sharp spaces}
In previous section, we have proved the existence and uniqueness of local pathwise solutions to the SISM \eqref{1.1}-\eqref{1.3}  with initial data in the sufficiently regular spaces $\mathcal {Z}^{k',2}(D)$, where $k'=k+4$, $k>1+\frac{2}{p}$ and $p\geq2$, which is densely embedded into $\mathcal {Z}^{k,p}(D)$. In the following, by using a density and stability argument, we will show that the $W^{k,p}$-solutions to the SISM \eqref{1.1}-\eqref{1.3}  can be approximated by the ``smooth" local pathwise solutions constructed in Section 3.

To be more precise, for any $\mathcal {Z}^{k,p}(D)$-valued $\mathcal {F}_0$-measurable random variables $(u_S^0,u_T^0,\theta_S^0)$, we shall consider the following SISM with smooth initial data:
\begin{equation}\label{4.1}
\left\{
\begin{aligned}
&\mathrm{d}u_S^j+  P(u_S^j\cdot \nabla)   u_S^j\mathrm{d}t-P(u_T^j\widehat{x}) \mathrm{d}t- P(\theta _S^j\widehat{z}) \mathrm{d}t= P\sigma_1(u_S^j,u_T^j,\theta_S^j)\mathrm{d}\mathcal {W},\\
&\mathrm{d} u_T^j+ (u_S^j\cdot \nabla)   u_T^j\mathrm{d}t+u_S^j\cdot\widehat{x} \mathrm{d}t +z\mathrm{d}t=\sigma_2(u_S^j,u_T^j,\theta_S^j)\mathrm{d}\mathcal {W}, \\
&\mathrm{d}\theta_S^j+ (u_S^j\cdot \nabla ) \theta_S^j\mathrm{d}t+u_T ^j\mathrm{d}t  = \sigma_3(u_S^j,u_T^j,\theta_S^j)\mathrm{d}\mathcal {W},\\
&u_S^j(0)= J_{\frac{1}{j}}u_S^0,\quad u_T^j(0)=J_{\frac{1}{j}}u_T^0,\quad \theta_S^j(0)=J_{\frac{1}{j}}\theta_S^0,
\end{aligned}
\right.
\end{equation}
where $J_\epsilon$ denotes the smoothing operators, which can be constructed by a method suggested by Kato and Lai \cite{13}; see also \cite{41}. For example, we define
$$
J_{\frac{1}{j}} =PR'F_{\frac{1}{j}} R, \quad \mbox{for all}~j\geq 1,
$$
where $P$ is the Leray projection operator defined as before, $R$ is a linear continuous operator of extension of $D$ into $\mathbb{R}^2$ such that $\|Rf\|_{W^{k,p}(\mathbb{R}^2)}\leq C \|f\|_{W^{k,p}(D)}$ for $k\geq0$ and $p\geq 2$, $R'$ is the operator of restriction which is bounded from $W^{k,p}(\mathbb{R}^2)$ into $W^{k,p}(D)$, and $F_{1/j} $ is the Friedrichs mollifier in $\mathbb{R}^2$ defined as the multiplication by $e^{- |x|^2/j^2}$ for the Fourier transforms.

Due to the smoothing effect of the operator $J_{1/j}$, we know that
$$
(J_{\frac{1}{j}}u_S^0,J_{\frac{1}{j}}u_T^0,J_{\frac{1}{j}}\theta_S^0)\in [C^\infty(D)]^3\subset\mathcal {Z}^{k',2}(D)\hookrightarrow \mathcal {Z}^{k,p}(D),\quad \forall j\geq 1.
$$
One can conclude from Lemma \ref{lem:3.3} that the Cauchy problem \eqref{4.1} admits a sequence of local maximal pathwise solutions $ (u_S^j,u_T^j,\theta_S^j,\mathbbm{t}_j)_{j\geq 1}$ evolving continuously in $\mathcal {Z}^{k'-2,2}(D)$ which are bounded in $\mathcal {Z}^{k',2}(D)$. By the assumption of $k'$, it is clear that $\mathcal {Z}^{k',2}(D)$ is densely embedded into $\mathcal {Z}^{k,p}(D)$. In order to prove the existence of local pathwise solutions to the SISM \eqref{1.1}-\eqref{1.3} for the sharp smothness regime,  we shall further prove that $(u_S^j,u_T^j,\theta_S^j)_{j\geq 1}$ form a Cauchy sequence  in the Banach space $\mathcal {Z}^{k,p}(D)$, whose proof is based on the following basic result.

\begin{lemma} [Abstract Cauchy Theorem \cite{41}] \label{lem:4.1} Let $T>0$, we define the stopping times $\mathbbm{t}_{j,l,T}:=\mathbbm{t}_{j,T}\wedge \mathbbm{t}_{l,T}$, for any $j,l\geq1$, where
$$
\mathbbm{t}_{j,T}:=\inf\{t\geq0;\|f^j(t)\|_{\mathcal {Z}^{k,p}}\geq 2+\|f^j_0\|_{\mathcal {Z}^{k,p}}\}\wedge T,~~~\forall j\geq1.
$$
Assume that
\begin{eqnarray}\label{4.2}
 \lim_{l\rightarrow\infty}\sup_{j\geq l} \mathbb{E} \sup_{t\in [0,\mathbbm{t}_{j,l,T}]}\|f^j(t)-f^l(t)\|_{\mathcal {Z}^{k,p}}=0,
\end{eqnarray}
and
\begin{eqnarray}\label{4.3}
\lim_{\varpi\rightarrow0}\sup_{j\geq 1} \mathbb{P} \Bigg\{\sup_{t\in [0,\mathbbm{t}_{j,k,T}\wedge \varpi]}\|f^j(t)\|_{\mathcal {Z}^{k,p}}>1+\|f^j_0\|_{\mathcal {Z}^{k,p}}\Bigg\}=0.
\end{eqnarray}
Then there exits a stopping time $\mathbbm{t}$ with $\mathbb{P}\{0<\mathbbm{t}\leq T\}=1$, and a predictable process $f(\cdot)=f(\cdot\wedge \mathbbm{t})\in  C ([0,\mathbbm{t}];\mathcal {Z}^{k,p}(D))$ such that
$$
\sup_{t\in [0,\mathbbm{t}]} \|f^{j_l}(t)-f(t)\|_{\mathcal {Z}^{k,p}}\rightarrow 0, \quad \mathbb{P}\mbox{-a.s.},
$$
for some subsequence $j_l\rightarrow\infty$. Moreover, the bounds
$$
\|f^j(t)\|_{\mathcal {Z}^{k,p}}\leq 2+\sup_{j}\|f^j_0\|_{\mathcal {Z}^{k,p}},\quad \mathbb{P}\mbox{-a.s.},
$$
holds uniformly for any $t\in [0,\mathbbm{t}]$.
\end{lemma}

To prove Theorem \ref{1.3}, by making use of Lemma \ref{lem:3.3} and Lemma \ref{lem:4.1}, it suffices to verify the essential convergence \eqref{4.2} and \eqref{4.3} with respect to the functional $f^j=(u_S^j,u_T^j,\theta_S^j)$, which is defined up to a sequence of stopping time $(\mathbbm{t}_j)_{j\geq 1}$ by \eqref{4.1}. For technique reasons, by applying the decomposition method as that in the proof of Lemma \ref{lem:3.3}, one can first assume that
$$
\|(u_S^0,u_T^0,\theta_S^0)\|_{\mathcal {Z}^{k,p}}\leq M,
$$
for some deterministic $M>0$ uniformly. Moreover, we deduce from the property of $J_{1/j}$ that
\begin{equation}\label{4.4}
\begin{split}
 \sup_{j\geq 1}\|(J_{\frac{1}{j}}u_S^0,J_{\frac{1}{j}}u_T^0,J_{\frac{1}{j}}\theta_S^0)\|_{\mathcal {Z}^{k,p}}\leq C \|(u_S^0,u_T^0,\theta_S^0)\|_{\mathcal {Z}^{k,p}}\leq CM,
\end{split}
\end{equation}
which implies by the definition of $\mathbbm{t}_{j,l,T}$ that
\begin{equation}\label{4.5}
\begin{split}
 \|(u_S^j(t),u_T^j(t),\theta_S^j(t))\|_{\mathcal {Z}^{k,p}}\leq 2+ \|(u_S^{0,j},u_T^{0,j},\theta_S^{0,j})\|_{\mathcal {Z}^{k,p}}\leq 2+CM,
\end{split}
\end{equation}
for any $t\in [0,\mathbbm{t}_{j,l,T}]$, where $C>0$ is a constant depending only on $k,p$ and $D$.

\begin{lemma} \label{lem:4.2}
The convergence \eqref{4.2} holds for the sequence $(u_S^{j},u_T^{j},\theta_S^{j})_{j\geq 1}$ defined by (4.1).
\end{lemma}

\begin{proof} Setting
$$
u_S^{j,l}=u_S^j-u_S^j,\quad u_T^{j,l}=u_T^j-u_T^l,\quad \theta_S^{j,l}=\theta_S^j-\theta_S^l,
$$
and
$$
u_S^{j,l}(0)=J_{\frac{1}{j}}u_S^0-J_{\frac{1}{l}}u_S^0,\quad
u_T^{j,l}(0)=J_{\frac{1}{j}}u_T^0-J_{\frac{1}{l}}u_T^0,\quad
\theta_S^{j,l}(0)=J_{\frac{1}{j}}\theta_S^0-J_{\frac{1}{l}}\theta_S^0,\quad \forall  j,l\geq 1.
$$
In view of \eqref{4.1}, we have
\begin{equation}\label{4.6}
\left\{
\begin{aligned}
&\mathrm{d}u_S^{j,l}+  P\left((u_S^{j,l}\cdot \nabla)   u_S^j+(u_S^l\cdot \nabla)   u_S^{j,l}\right)\mathrm{d}t-P(u_T^{j,l}\widehat{x}) \mathrm{d}t- P(\theta _S^{j,l}\widehat{z}) \mathrm{d}t = P\triangle \sigma ^{1,j,l}\mathrm{d}\mathcal {W},\\
&\mathrm{d} u_T^{j,l}+ \left((u_S^{j,l}\cdot \nabla)   u_T^j+(u_S^l\cdot \nabla)   u_T^{j,l}\right)\mathrm{d}t+u_S^{j,l}\cdot\widehat{x} \mathrm{d}t  =\triangle \sigma ^{2,j,l}\mathrm{d}\mathcal {W}, \\
&\mathrm{d}\theta_S^{j,l}+ \left((u_S^{j,l}\cdot \nabla ) \theta_S^j+(u_S^l\cdot \nabla ) \theta_S^{j,l}\right)\mathrm{d}t+u_T^{j,l}\mathrm{d}t = \triangle \sigma ^{3,j,l}\mathrm{d}\mathcal {W},
\end{aligned}
\right.
\end{equation}
with
$$
\triangle \sigma ^{i,j,l}=\sigma_i(u_S^j,u_T^j,\theta_S^j)-\sigma_i(u_S^l,u_T^l,\theta_S^l),\quad \forall j,l\geq 1,\quad i=1,2,3.
$$
Applying the It\^{o}'s formula in Banach space to $\|u_S^{j,l}\|_{X^{k,p}}^p=\sum_{|\alpha|\leq k}\| \partial^\alpha  u_S^{j,l}\|_{L^p}^p$, we find
\begin{equation}\label{4.7}
\begin{split}
\mathrm{d}\| \partial^\alpha u_S^{j,l}\|_{L^p}^p=&-p\int_D| \partial^\alpha u_S^{j,l}|^{p-2} \partial^\alpha u_S^{j,l}\cdot  \partial^\alpha P\left((u_S^{j,l}\cdot \nabla)   u_S^j+(u_S^l\cdot \nabla)   u_S^{j,l}\right)\mathrm{d}V\mathrm{d}t\\
&+p\int_D| \partial^\alpha u_S^{j,l}|^{p-2} \partial^\alpha u_S^{j,l}\cdot  \partial^\alpha P\left(u_T^{j,l}\widehat{x}+\theta _S^{j,l}\widehat{z}\right) \mathrm{d}V\mathrm{d}t\\
&+\sum_{n\geq 1}\frac{p}{2}\int_D| \partial^\alpha u_S^{j,l}|^{p-2}\left( \partial^\alpha P\triangle \sigma ^{1,j,l}_n\right)^2\mathrm{d}V\mathrm{d} t\\
&+\sum_{n\geq 1}\frac{p(p-2)}{2}\int_D| \partial^\alpha u_S^{j,l}|^{p-4}\left( \partial^\alpha u_S^{j,l}\cdot  \partial^\alpha P\triangle \sigma ^{1,j,l}_n\right)^2\mathrm{d}V\mathrm{d}t\\
&+\sum_{n\geq 1}p\int_D| \partial^\alpha u_S^{j,l}|^{p-2} \partial^\alpha u_S^{j,l}\cdot  \partial^\alpha P\triangle \sigma ^{1,j,l}_n\mathrm{d}V\mathrm{d}\beta_n \\
=& (I_1^\alpha(t)+I_2^\alpha(t)+I_3^\alpha(t)+I_4^\alpha(t))\mathrm{d} t+I_5^\alpha(t)\mathrm{d} \mathcal {W} ,
\end{split}
\end{equation}
where $\triangle \sigma ^{i,j,l}_n=\triangle \sigma ^{1,j,l}e_n=\sigma_i(u_S^j,u_T^j,\theta_S^j)e_n-\sigma_i(u_S^l,u_T^l,\theta_S^l)e_n$, $n\geq 1$, $i=1,2,3,$ and $(e_n)_{n\geq 1}$ is the complete orthonormal system in the separate Hilbert space $\mathfrak{A}$.

For $I_1^\alpha(t)$, we deduce from the H\"{o}lder inequality that
\begin{equation}\label{4.8}
\begin{split}
I_1^\alpha(t)\leq& C\| \partial^\alpha u_S^{j,l}\|_{L^p}^{p-1}\| \partial^\alpha P(u_S^{j,l}\cdot \nabla)   u_S^j\|_{L^p}\\
&+p\int_D| \partial^\alpha u_S^{j,l}|^{p-2} \partial^\alpha u_S^{j,l}\cdot  \partial^\alpha P (u_S^l\cdot \nabla)   u_S^{j,l}\mathrm{d}V.
\end{split}
\end{equation}
The first nonlinear term on the R.H.S of \eqref{4.8} can be estimated as
\begin{equation}\label{4.9}
\begin{split}
&\| \partial^\alpha u_S^{j,l}\|_{L^p}^{p-1}\| \partial^\alpha P(u_S^{j,l}\cdot \nabla)   u_S^j\|_{L^p}\\
&\quad \leq  C\| \partial^\alpha u_S^{j,l}\|_{L^p}^{p-1}\left(\|u_S^{j,l}\|_{L^\infty}\|\nabla u_S^j\|_{W^{k,p}}+\|u_S^{j,l}\|_{W^{k,p}}\|\nabla u_S^j\|_{L^\infty}\right)\\
&\quad \leq  C\|u_S^{j,l}\|_{W^{k,p}}^{p-1}\left(\|u_S^{j,l}\|_{W^{k-1,p}}\|u_S^j\|_{W^{k+1,p}}+\|u_S^{j,l}\|_{W^{k,p}}\| u_S^j\|_{W^{k,p}}\right)\\
&\quad \leq   C \|u_S^{j,l}\|_{W^{k-1,p}}^p\|u_S^j\|_{W^{k+1,p}}^p+C\|u_S^{j,l}\|_{W^{k,p}}^{p}(1+\| u_S^j\|_{W^{k,p}}).
\end{split}
\end{equation}
Using the commutator estimate, the second term on the R.H.S of \eqref{4.8} can be estimated as
\begin{equation}\label{4.10}
\begin{split}
&\int_D| \partial^\alpha u_S^{j,l}|^{p-2} \partial^\alpha u_S^{j,l}\cdot  \partial^\alpha P (u_S^l\cdot \nabla)   u_S^{j,l}\mathrm{d}V\\
&\quad   =\int_D| \partial^\alpha u_S^{j,l}|^{p-2} \partial^\alpha u_S^{j,l}\cdot  \partial^\alpha (I-P)(u_S^l\cdot \nabla)   u_S^{j,l}\mathrm{d}V\\
&\quad \quad +\int_D| \partial^\alpha u_S^{j,l}|^{p-2} \partial^\alpha u_S^{j,l}\cdot  \partial^\alpha (u_S^l\cdot \nabla)   u_S^{j,l}\mathrm{d}V\\
&\quad   \leq C\| \partial^\alpha u_S^{j,l}\|_{L^p}^{p-1} \left(\|[ \partial^\alpha ,u_S^l]\cdot \nabla   u_S^{j,l}\|_{L^p}+ \| \partial^\alpha (I-P)(u_S^l\cdot \nabla)   u_S^{j,l}\|_{L^p}\right)\\
&\quad   \leq C\| u_S^{j,l}\|_{W^{k,p}}^{p-1}\Big( \| u_S^l\|_{W^{k,p}}\| \nabla u_S^{j,l}\|_{L^\infty}+  \| \nabla u_S^l\|_{L^\infty} \|u_S^{j,l}\|_{W^{k,p}} \\
&\quad\quad\quad\quad + \|u_S^l\|_{W^{1,\infty}}\| u_S^{j,l}\|_{W^{k,p}}+\|u_S^l\|_{W^{k,p}}\|u_S^{j,l}\|_{W^{1,\infty}}\Big)\\
& \quad   \leq C\|u_S^l\|_{W^{k,p}}\|u_S^{j,l}\|_{W^{k,p}}^{p}.
\end{split}
\end{equation}
Plugging \eqref{4.9} and \eqref{4.10} into \eqref{4.8}, we get
\begin{equation}\label{4.11}
I_1^\alpha(t)\leq C\|u_S^{j,l}\|_{W^{k,p}}^{p}(1+\|u_S^l\|_{W^{k,p}}+\| u_S^j\|_{W^{k,p}})+C \|u_S^{j,l}\|_{W^{k-1,p}}^p\|u_S^j\|_{W^{k+1,p}}^p.
\end{equation}
For $I_2^\alpha(t)$, it follows from the Young inequality that
\begin{equation}\label{4.12}
\begin{split}
\mathbb{E}\sup_{s\in [0,t]}\bigg|\int_0^sI_2^\alpha(r)\mathrm{d} r\bigg|&\leq \mathbb{E}\int_0^t\int_D| \partial^\alpha u_S^{j,l}(r)|^{p-1}| \partial^\alpha P(u_T^{j,l}\widehat{x}+\theta _S^{j,l}\widehat{z})| \mathrm{d}V\mathrm{d}r\cr
& \leq \mathbb{E}\Bigg(\sup_{r\in [0,t]}\| \partial^\alpha u_S^{j,l}(r)\|_{L^p}^{p-1}\int_0^t\| \partial^\alpha P(u_T^{j,l}\widehat{x}+\theta _S^{j,l}\widehat{z})\|_{L^p} \mathrm{d}r\Bigg)\cr
& \leq \frac{1}{4}\mathbb{E} \sup_{r\in [0,t]}\|u_S^{j,l}(r)\|_{X^{k,p}}^p+C\mathbb{E}\int_0^t\left(\|u_T^{j,l}\|_{W^{k,p}}^p+\|\theta _S^{j,l}\|_{W^{k,p}}\right)\mathrm{d}r.
\end{split}
\end{equation}
For $I_3^\alpha(t)$, by using the locally Lipschitz condition in (A1), we have
\begin{equation}\label{4.13}
\begin{split}
 I_3^\alpha(t)+I_4^\alpha(t)\leq &C\| \partial^\alpha u_S^{j,l}\|^{p-2}_{L^p}\bigg\|\sum_{n\geq 1}( \partial^\alpha P\triangle \sigma ^{1,j,l}_n)^2\bigg\|_{L^{p/2}}^{2/p}\\
\leq & C\| u_S^{j,l}\|^{p-2}_{X^{k,p}}\|\sigma_1(u_S^j,u_T^j,\theta_S^j)-\sigma_1(u_S^l,u_T^l,\theta_S^l)
\|_{\mathbb{X}^{k,p}}^2\\
\leq &  C\| u_S^{j,l}\|^{p-2}_{X^{k,p}}\varsigma\Bigg( \sum_{\lambda\in\{j,l\}}(\| u_S^\lambda\|_{L^\infty}+\| u_T^\lambda\|_{L^\infty}+\| \theta_S^\lambda\|_{L^\infty})\Bigg)^2\\
 &\times
\left(\|u_S^{j,l}\|_{X^{k,p}}^2+\|u_T^{j,l}\|_{W^{k,p}}^2+\|\theta_S^{j,l}\|_{W^{k,p}}^2\right)\\
\leq &  C\varsigma\Bigg( \sum_{\lambda\in\{j,l\}}(\| u_S^\lambda\|_{L^\infty}+\| u_T^\lambda\|_{L^\infty}+\| \theta_S^\lambda\|_{L^\infty})\Bigg)^2\\
 &\times
\left(\|u_S^{j,l}\|_{X^{k,p}}^p+\|u_T^{j,l}\|_{W^{k,p}}^p+\|\theta_S^{j,l}\|_{W^{k,p}}^p\right).
\end{split}
\end{equation}
For $I_5^\alpha(t)$, by applying the BDG inequality, the Minkovski inequality and using the assumption (A1), one can deduce that for any stopping time $\mathbbm{t}\geq0$ almost surely,
\begin{equation}\label{4.14}
\begin{split}
\mathbb{E}\sup_{s\in [0,\mathbbm{t}]}\bigg|\int_0^sI_5^\alpha(r)\mathrm{d} \mathcal {W}_r\bigg|\leq& C\mathbb{E}\Bigg(\int_0^\mathbbm{t}\sum_{n\geq 1}\bigg(\int_D| \partial^\alpha u_S^{j,l}|^{p-2} \partial^\alpha u_S^{j,l}\cdot  \partial^\alpha P\triangle \sigma ^{1,j,l}_n\mathrm{d}V\bigg)^2\mathrm{d} t\Bigg)^{1/2}\\
\leq& C\mathbb{E}\Bigg(\int_0^\mathbbm{t}\bigg(\int_D| \partial^\alpha u_S^{j,l}|^{p-1} \bigg(\sum_{n\geq 1}( \partial^\alpha P\triangle \sigma ^{1,j,l}_n)^2\bigg)^{1/2}\mathrm{d}V\bigg)^2\mathrm{d} t\Bigg)^{1/2}\\
\leq &C\mathbb{E}\Bigg(\sup_{s\in [0,\mathbbm{t}]}\| \partial^\alpha u_S^{j,l}(s)\|_{L^p}^{2p-2} \\
& \times\int_0^\mathbbm{t} \|\sigma_1(u_S^j,u_T^j,\theta_S^j)-\sigma_1(u_S^l,u_T^l,\theta_S^l)
\|_{\mathbb{X}^{k,p}}^2\mathrm{d} t\Bigg)^{1/2}\\
\leq& \frac{1}{4}\mathbb{E}\sup_{s\in [0,\mathbbm{t}]}\| u_S^{j,l}(s)\|_{ X^{k,p}}^{p} \\
& +C\mathbb{E}\int_0^\mathbbm{t}\varsigma\Bigg( \sum_{\lambda\in\{j,l\}}(\| u_S^\lambda\|_{L^\infty}+\| u_T^\lambda\|_{L^\infty}+\| \theta_S^\lambda\|_{L^\infty})\Bigg)^2\\
 & \times
\left(\|u_S^{j,l}\|_{X^{k,p}}^p+\|u_T^{j,l}\|_{W^{k,p}}^p+\|\theta_S^{j,l}\|_{W^{k,p}}^p\right)\mathrm{d} r.
\end{split}
\end{equation}
Combining the estimates \eqref{4.11}-\eqref{4.14}, taking supremum over $[0,\mathbbm{t}_{j,l,T}]$ and then taking the expected value, it follows from the identity \eqref{4.7}, the Sobolev embedding $\mathcal {Z}^{k,p}\hookrightarrow \mathcal {Z}^{1,\infty}$ and the definition of the stopping time $\mathbbm{t}_{j,l,T}$ that
\begin{equation}\label{4.16}
\begin{split}
&  \mathbb{E}\sup_{s\in [0,\mathbbm{t}_{j,l,T}\wedge t]}\| u_S^{j,l}(s)\|_{ X^{k,p}}^{p}\\
&\quad \leq 2\mathbb{E}\| u_S^{j,l}(0)\|_{ X^{k,p}}^{p}+C\int_0^t \mathbb{E}\sup_{s\in [0,\mathbbm{t}_{j,l,T}\wedge r]}(\|u_S^{j,l}\|_{X^{k-1,p}}^p\|u_S^j\|_{X^{k+1,p}}^p)\mathrm{d}r\\
&\quad \quad  +  C\int_0^t\mathbb{E} \sup_{s\in [0,\mathbbm{t}_{j,l,T}\wedge r]} \|(u_S^{j,l}(s),u_T^{j,l}(s),\theta_S^{j,l}(s))\|_{\mathcal {Z}^{k,p}}^p\mathrm{d} r.
\end{split}
\end{equation}
Applying the It\^{o}'s formula in the Banach space to $\| \partial^\alpha  u_T^{j,l}\|_{L^p}^p=\|u_T^{j,l}\|_{W^{k,p}}^p$ with respect to the second component in \eqref{4.1}, we find
\begin{equation}\label{4.17}
\begin{split}
\mathrm{d}\| \partial^\alpha u_T^{j,l}\|_{L^p}^p=&-p\int_D| \partial^\alpha u_T^{j,l}|^{p-2} \partial^\alpha u_T^{j,l}   \partial^\alpha \left((u_S^{j,l}\cdot \nabla)   u_T^j+(u_S^l\cdot \nabla)   u_T^{j,l}+u_S^{j,l}\widehat{x}\right)\mathrm{d}V\mathrm{d}t\\
&+\sum_{n\geq 1}\frac{p(p-1)}{2}\int_D| \partial^\alpha u_T^{j,l}|^{p-2}( \partial^\alpha \triangle \sigma ^{1,j,l}_n)^2\mathrm{d}V\mathrm{d}t\\
&+\sum_{n\geq 1}p \int_D | \partial^\alpha u_T^{j,l}|^{p-2} \partial^\alpha u_T^{j,l}  \partial^\alpha \triangle \sigma ^{1,j,l}_n\mathrm{d}V\mathrm{d}\beta_n\\
=& (J_1^\alpha(t)+J_2^\alpha(t))\mathrm{d} t+J_3^\alpha(t)\mathrm{d} \mathcal {W} .
\end{split}
\end{equation}
For $J_1^\alpha(t)$, we get by the free divergence condition $\div u_S^l =0$ that
$$
p\int_D| \partial^\alpha u_T^{j,l}|^{p-2} \partial^\alpha u_T^{j,l} (u_S^l\cdot \nabla  \partial^\alpha  u_T^{j,l})\mathrm{d}V=\int_D u_S^l\cdot\nabla | \partial^\alpha u_T^{j,l}|^p\mathrm{d}V=0.
$$
By utilizing the H\"{o}lder inequality and  the commutator estimate, we get
\begin{equation}\label{4.18}
\begin{split}
J_1^\alpha(t)\leq& C\| \partial^\alpha u_T^{j,l}\|_{L^p}^{p-1}\left(\| \partial^\alpha (u_S^{j,l}\cdot \nabla) u_T^j\|_{L^p}+\|[ \partial^\alpha ,u_S^l]\cdot \nabla u_T^{j,l}\|_{L^p}+\| \partial^\alpha u_S^{j,l}\|_{L^p}\right)\\
\leq &C\|u_T^{j,l}\|_{W^{k,p}}^{p-1}\Big(\|u_S^{j,l}\|_{L^\infty} \|\nabla u_T^j\|_{W^{k,p}}+\|u_S^{j,l}\|_{W^{k,p}}\|\nabla u_T^j\|_{L^\infty}   \\
&+\|\nabla u_S^l\|_{L^\infty}\|\nabla u_T^{j,l}\|_{W^{k-1,p}}+ \| \partial^\alpha u_S^l\|_{L^p}\|\nabla u_T^{j,l}\|_{L^\infty}+\|u_S^{j,l}\|_{X^{k,p}}\Big)\\
\leq &C\|u_T^{j,l}\|_{W^{k,p}}^{p-1}\Big(\|u_S^{j,l}\|_{W^{k-1,p}} \| u_T^j\|_{W^{k+1,p}}+\|u_T^{j,l}\|_{W^{k,p}}\|u_T^j\|_{W^{k,p}} \\
&+\|u_S^l\|_{X^{k,p}}\|u_T^{j,l}\|_{W^{k,p}}+\|u_S^{j,l}\|_{X^{k,p}}\Big)\\
\leq &C\|u_T^{j,l}\|_{W^{k,p}}^p\left(\|u_T^j\|_{W^{k,p}}+\|u_S^l\|_{X^{k,p}}+1\right) +C\|u_S^{j,l}\|_{W^{k-1,p}}^p \| u_T^j\|_{W^{k+1,p}}^p.
\end{split}
\end{equation}
Thanks to the locally Lipschitz condition in (A1), we can estimate $J_2^\alpha(t)$ as
\begin{equation}\label{4.19}
\begin{split}
|J_2^\alpha(t)|\leq& C\| \partial^\alpha u_T^{j,l}\|_{L^p}^{p-2}   \|  \partial^\alpha \sigma_2(u_S^j,u_T^j,\theta_S^j)-\sigma_2(u_S^l,u_T^l,\theta_S^l)
\|_{\mathbb{W}^{k,p}}^2\\
\leq&C \| \partial^\alpha u_T^{j,l}\|_{L^p}^{p-2}C\varsigma\Bigg( \sum_{\lambda\in\{j,l\}}(\| u_S^\lambda\|_{L^\infty}+\| u_T^\lambda\|_{L^\infty}+\| \theta_S^\lambda\|_{L^\infty})\Bigg)^2\\
 &\times
\left(\|u_S^{j,l}\|_{X^{k,p}}^2+\|u_T^{j,l}\|_{W^{k,p}}^2+\|\theta_S^{j,l}\|_{W^{k,p}}^2\right)\\
\leq& C\varsigma\Bigg( \sum_{\lambda\in\{j,l\}}(\| u_S^\lambda\|_{L^\infty}+\| u_T^\lambda\|_{L^\infty}+\| \theta_S^\lambda\|_{L^\infty})\Bigg)^2\\
 &\times
\left(\|u_S^{j,l}\|_{X^{k,p}}^p+\|u_T^{j,l}\|_{W^{k,p}}^p+\|\theta_S^{j,l}\|_{W^{k,p}}^p\right).
\end{split}
\end{equation}
For the stochastic term $J_3^\alpha(t)$, by using the BDG inequality and estimating in a similar manner to \eqref{4.19}, we find that for any stopping time $\mathbbm{t}\geq0$ almost surely,
\begin{equation}\label{4.20}
\begin{split}
&\mathbb{E}\sup_{s\in [0,\mathbbm{t}]}\bigg|\int_0^sJ_3^\alpha(r)\mathrm{d} \mathcal {W}\bigg|\\
&\quad\leq C\mathbb{E}\Bigg(\sup_{s\in [0,\mathbbm{t}]}\| \partial^\alpha u_T^{j,l}(s)\|_{L^p}^{2p-2} \int_0^\mathbbm{t} \|\sigma_2(u_S^j,u_T^j,\theta_S^j)-\sigma_2(u_S^l,u_T^l,\theta_S^l)
\|_{\mathbb{X}^{k,p}}^2\mathrm{d} t\Bigg)^{1/2}\\
&\quad\leq \frac{1}{2}\mathbb{E}\sup_{s\in [0,\mathbbm{t}]}\| u_T^{j,l}(s)\|_{ X^{k,p}}^{p} +C\mathbb{E}\int_0^\mathbbm{t}\varsigma\Bigg( \sum_{\lambda\in\{j,l\}}(\| u_S^\lambda\|_{L^\infty}+\| u_T^\lambda\|_{L^\infty}+\| \theta_S^\lambda\|_{L^\infty})\Bigg)^2\\
 &\quad \quad\times
\left(\|u_S^{j,l}\|_{X^{k,p}}^p+\|u_T^{j,l}\|_{W^{k,p}}^p+\|\theta_S^{j,l}\|_{W^{k,p}}^p\right)\mathrm{d} r.
\end{split}
\end{equation}
Taking $ \mathbbm{t}=\mathbbm{t}_{j,l,T}$ and integrating the identity \eqref{4.17} over $[0,\mathbbm{t}_{j,l,T}]$. After plugging the estimates \eqref{4.18}-\eqref{4.20} into \eqref{4.17}, we deduce from the definition of $\mathbbm{t}_{j,l,T}$ that
\begin{equation}\label{4.21}
\begin{split}
&\mathbb{E}\sup_{s\in [0,\mathbbm{t}_{j,l,T}\wedge t]}\| u_T^{j,l}(s)\|_{W^{k,p}}^p\\
&\quad \leq  2\| u_T^{j,l}(0)\|_{W^{k,p}}^p+C\int_0^ {\mathbbm{t}_{j,l,T}} \mathbb{E}(\|u_T^{j,l}\|_{W^{k-1,p}}^p \| u_T^j\|_{W^{k+1,p}}^p) \mathrm{d}r\\
&\quad\quad+C\int_0^ {\mathbbm{t}_{j,l,T}}\mathbb{E}\left(\|u_S^{j,l}\|_{X^{k,p}}^p+\|u_T^{j,l}\|_{W^{k,p}}^p+\|\theta_S^{j,l}\|_{W^{k,p}}^p\right)\mathrm{d} r\\
&\quad \leq  2\| u_T^{j,l}(0)\|_{W^{k,p}}^p+C\int_0^t\mathbb{E}\sup_{s\in [0,\mathbbm{t}_{j,l,T}\wedge r]}\left(\|u_S^{j,l}\|_{W^{k-1,p}}^p \| u_T^j\|_{W^{k+1,p}}^p\right) \mathrm{d}r\\
&\quad\quad+C\int_0^t\mathbb{E}\sup_{s\in [0,\mathbbm{t}_{j,l,T}\wedge r]}\|(u_S^{j,l},u_T^{j,l},\theta_S^{j,l})(s)\|_{\mathcal {Z}^{k,p}}^p \mathrm{d} r.
\end{split}
\end{equation}
Applying the It\^{o}'s formula in Banach space to  $\| \partial^\alpha  \theta_S^{j,l}\|_{L^p}^p=\|\theta_S^{j,l}\|_{W^{k,p}}^p$ with respect to the second scalar equation in \eqref{4.1}, we find
\begin{equation}\label{4.22}
\begin{split}
 \mathrm{d}\| \partial^\alpha \theta_S^{j,l}(t)\|_{L^p}^p=& -p \int_D| \partial^\alpha \theta_S^{j,l}|^{p-2} \partial^\alpha \theta_S^{j,l}   \partial^\alpha \left((u_S^{j,l}\cdot \nabla ) \theta_S^j +(u_S^l\cdot \nabla ) \theta_S^{j,l}+u_T^{j,l}\right)\mathrm{d}V\mathrm{d}t\\
&+\sum_{n\geq 1}\frac{p(p-1)}{2} \int_D| \partial^\alpha \theta_S^{j,l}|^{p-2}( \partial^\alpha \triangle \sigma ^{3,j,l}_n)^2\mathrm{d}V\mathrm{d}t\\
&+\sum_{n\geq 1}p  \int_D | \partial^\alpha \theta_S^{j,l}|^{p-2} \partial^\alpha \theta_S^{j,l}  \partial^\alpha \triangle \sigma ^{3,j,l}_n\mathrm{d}V\mathrm{d}\beta_n \\
=& (K_1^\alpha(t)+K_2^\alpha(t))\mathrm{d} t +K_3^\alpha(t) \mathrm{d} \mathcal {W}.
\end{split}
\end{equation}
 For $K_1^\alpha(t)$, we have
\begin{equation}\label{4.22}
\begin{split}
&\mathbb{E}\sup_{s\in [0,\mathbbm{t}_{j,l,T}\wedge t]}\bigg|\int_0^s K_1^\alpha(r)\mathrm{d} r\bigg|\\
&\quad \leq C\mathbb{E}\int_0^{\mathbbm{t}_{j,l,T}\wedge t}\| \partial^\alpha \theta_S^{j,l}\|_{L^p}^{p-1}\| \partial^\alpha (u_S^{j,l}\cdot \nabla ) \theta_S^j +[ \partial^\alpha ,u_S^l]\cdot \nabla \theta_S^{j,l}+ \partial^\alpha u_T^{j,l}\|_{L^p}\mathrm{d}t\\
&\quad \leq C\mathbb{E}\int_0^{\mathbbm{t}_{j,l,T}\wedge t}\|\theta_S^{j,l}\|_{W^{k,p}}^{p-1}\Big(\|u_S^{j,l}\|_{L^\infty}\|\nabla \theta_S^{j}\|_{W^{k,p}} +\|\nabla \theta_S^{j}\|_{L^\infty}\|u_S^{j,l}\|_{W^{k,p}} \\
&\quad\quad+\|\nabla u_S^l \|_{L^\infty}\|\Lambda^{k-1}\nabla \theta_S^{j,l} \|_{L^\infty}+\|\Lambda^{k}u_S^l \|_{L^p}\|\nabla \theta_S^{j,l} \|_{L^\infty} +\| \partial^\alpha \theta_S^{j,l}\|_{L^p}\Big)\mathrm{d}t\\
&\quad \leq C\mathbb{E}\int_0^{\mathbbm{t}_{j,l,T}\wedge t}\Big( \|u_S^{j,l}\|_{W^{k-1,p}}^p\| \theta_S^{j}\|_{W^{k+1,p}}^p +\|\theta_S^{j,l}\|_{W^{k,p}}^{p} +  \|u_S^{j,l}\|_{W^{k,p}} ^p \Big)\mathrm{d}t.
\end{split}
\end{equation}
 For $K_2^\alpha(t)$, we have
\begin{equation}\label{4.23}
\begin{split}
&\mathbb{E}\sup_{s\in [0,\mathbbm{t}_{j,l,T}\wedge t]}\bigg| \int_0^s K_2^\alpha(r)\mathrm{d} r\bigg|\\
&\quad \leq C\mathbb{E} \sum_{n\geq 1} \int_0^{\mathbbm{t}_{j,l,T}\wedge t}\int_D| \partial^\alpha \theta_S^{j,l}|^{p-2}( \partial^\alpha \triangle \sigma ^{3,j,l}_n)^2\mathrm{d}V\mathrm{d}t\\
& \quad \leq C\mathbb{E} \int_0^{\mathbbm{t}_{j,l,T}\wedge t}      \| \partial^\alpha \theta_S^{j,l}\|_{L^p}^{p-2}\Bigg(\int_D \bigg( \sum_{n\geq 1}  ( \partial^\alpha \triangle \sigma ^{3,j,l}_n)^2\bigg)^{p/2}\mathrm{d} x  \Bigg)^{2/p} \mathrm{d}r\\
& \quad \leq C\mathbb{E} \int_0^{\mathbbm{t}_{j,l,T}\wedge t}      \| \partial^\alpha \theta_S^{j,l}\|_{L^p}^{p-2}  \|\sigma_3(u_S^j,u_T^j,\theta_S^j)-\sigma_3(u_S^l,u_T^l,\theta_S^l)
\|_{\mathbb{W}^{k,p}}^2        \mathrm{d}r\\
&\quad\leq \frac{1}{4}\mathbb{E}\sup_{s\in [0,\mathbbm{t}_{j,l,T}\wedge t]}\|\theta_S^{j,l}(s)\|_{ X^{k,p}}^{p} \\
&\quad\quad+C\mathbb{E}\int_0^{\mathbbm{t}_{j,l,T}\wedge t}
\left(\|u_S^{j,l}\|_{X^{k,p}}^p+\|u_T^{j,l}\|_{W^{k,p}}^p+\|\theta_S^{j,l}\|_{W^{k,p}}^p\right)\mathrm{d} r.
\end{split}
\end{equation}
By virtue of the BDG inequality, the term $K_3^\alpha(t)$ can be estimated by
\begin{equation}\label{4.24}
\begin{split}
&\mathbb{E}\sup_{s\in [0,\mathbbm{t}_{j,l,T}\wedge t]}\bigg| \int_0^s K_3^\alpha(r)\mathrm{d} r\bigg|\\
&\quad \leq C\mathbb{E}\Bigg(  \int_0^{\mathbbm{t}_{j,l,T}\wedge s} \sum_{n\geq 1}\bigg(\int_D| \partial^\alpha \theta_S^{j,l}|^{p-2} \partial^\alpha \theta_S^{j,l}  \partial^\alpha \triangle \sigma ^{3,j,l}_n\mathrm{d}V\bigg)^2\mathrm{d}r\Bigg)^{1/2}\\
 &\quad\leq  \frac{1}{4}\mathbb{E}\sup_{s\in [0,\mathbbm{t}_{j,l,T}\wedge t]}\| \theta_S^{j,l}(s)\|_{ W^{k,p}}^{p} +C \int_0^{\mathbbm{t}_{j,l,T}\wedge t}
 \mathbb{E}\sup_{s\in [0,\mathbbm{t}_{j,l,T}\wedge r]}\|(u_S^{j,l},u_T^{j,l},\theta_S^{j,l})\|_{\mathcal {Z}^{k,p}}^p \mathrm{d} r.
\end{split}
\end{equation}
By taking supremum over $[0,\mathbbm{t}_{j,l,T}\wedge t]$ and then the expected values to  \eqref{4.22}, we get from the estimates \eqref{4.22}-\eqref{4.24} that
\begin{equation}\label{4.25}
\begin{split}
\mathbb{E}\sup_{s\in [0,\mathbbm{t}_{j,l,T}\wedge t]}\|\theta_S^{j,l}(s)\|_{W^{k,p}}^p\leq& 2  \|\theta_S^{j,l}(0)\|_{W^{k,p}}^p+\int_0^t\mathbb{E}\sup_{s\in [0,\mathbbm{t}_{j,l,T}\wedge r]}(\|u_S^{j,l}\|_{W^{k-1,p}}^p\| \theta_S^{j}\|_{W^{k+1,p}}^p)  \mathrm{d} r\\
& +C\int_0^t
\mathbb{E} \sup_{s\in [0,\mathbbm{t}_{j,l,T}\wedge r]}\|(u_S^{j,l},u_T^{j,l},\theta_S^{j,l})(s)\|_{\mathcal {Z}^{k,p}}^p \mathrm{d} r.
\end{split}
\end{equation}
Putting the estimates \eqref{4.16}, \eqref{4.21} and \eqref{4.25} together, we get
\begin{equation*}
\begin{split}
&\mathbb{E}\sup_{s\in [0,\mathbbm{t}_{j,l,T}\wedge t]}\|(u_S^{j,l}(s),u_T^{j,l}(s),\theta_S^{j,l}(s))\|_{\mathcal {Z}^{k,p}}^p\\
&\quad\leq C \|(u_S^{j,l}(0),u_T^{j,l}(0),\theta_S^{j,l}(0))\|_{\mathcal {Z}^{k,p}}^p\\
&\quad\quad +C\int_0^t \mathbb{E}\sup_{s\in [0,\mathbbm{t}_{j,l,T}\wedge r]}\left(\|u_S^{j,l}(s)\|_{W^{k-1,p}}^p\|(u_S^j(s),u_T^j(s),\theta_S^j(s))\|_{\mathcal {Z}^{k,p}}^p\right)\mathrm{d} r  \\
&\quad\quad +C \int_0^t
 \mathbb{E}\sup_{s\in [0,\mathbbm{t}_{j,l,T}\wedge r]}\|(u_S^{j,l}(s),u_T^{j,l}(s),\theta_S^{j,l}(s))\|_{\mathcal {Z}^{k,p}}^p \mathrm{d} r,
\end{split}
\end{equation*}
which together with the Gronwall inequality leads to
\begin{equation}\label{4.26}
\begin{split}
&\mathbb{E}\sup_{s\in [0,\mathbbm{t}_{j,l,T}]}\|(u_S^{j,l}(s),u_T^{j,l}(s),\theta_S^{j,l}(s))\|_{\mathcal {Z}^{k,p}}^p\leq  C \|(u_S^{j,l}(0),u_T^{j,l}(0),\theta_S^{j,l}(0))\|_{\mathcal {Z}^{k,p}}^p\\
&\quad +   C\mathbb{E}\sup_{s\in [0,\mathbbm{t}_{j,l,T}]}\left(\|u_S^{j,l}(s)\|_{W^{k-1,p}}^p\|(u_S^j(s),u_T^j(s),\theta_S^j(s))\|_{\mathcal {Z}^{k,p}}^p\right).
\end{split}
\end{equation}
By virtue of the property of the mollifiers $J_{1/j}$, we have $\|J_{1/j} u_S^{0}-u_S^{0}\|_{X^{k,p}}\rightarrow 0$ as $j\rightarrow \infty$, which implies that
$$
\lim_{j\rightarrow\infty} \sup_{l\geq j}\mathbb{E}\|u_S^{j,l}(0)\|_{X^{k,p}}=\lim_{j\rightarrow\infty} \sup_{l\geq j}\mathbb{E}\|J_{1/j}  u_0- J_{1/l} u_0\|_{X^{k,p}}=0.
$$
Similarly, we also have $\lim_{j\rightarrow\infty} \sup_{l\geq j}\mathbb{E}\|u_T^{j,l}(0)\|_{W^{k,p}}=0$ and $\lim_{j\rightarrow\infty} \sup_{l\geq j}\mathbb{E}\|\theta_S^{j,l}(0)\|_{W^{k,p}}=0$. As a consequence, we get
$$
\lim_{j\rightarrow\infty} \sup_{l\geq j}\mathbb{E}\|(u_S^{j,l}(0),u_T^{j,l}(0),\theta_S^{j,l}(0))\|_{\mathcal {Z}^{k,p}}^p=0.
$$
Therefore, in order to verify the convergence \eqref{4.2}, we have to derive some estimates for the following term:
\begin{equation}\label{333}
\begin{split}
\mathbb{E}\sup_{s\in [0,\mathbbm{t}_{j,l,T}]}\left(\|(u_S^{j,l},u_T^{j,l},\theta_S^{j,l})\|_{\mathcal {Z}^{k-1,p}}^p
\|(u_S^j,u_T^j,\theta_S^j)\|_{\mathcal {Z}^{k+1,p}}^p\right).
\end{split}
\end{equation}
To make the notation less cumbersome, we shall consider the equality \eqref{4.7} with
$k$ replaced by $k-1$, the equality \eqref{2.3} with $k$ replaced by $k+1$ and $u_S$ replaced by $u_S^j$, that is, we take
\begin{equation*}
\begin{split}
 \mathrm{d}\| u_S^j\|_{W^{k+1,p}}^p= \sum_{|\alpha|\leq k+1}\Big(({L_1^\alpha}' +{L_2^\alpha}' +{L_3^\alpha}' )\mathrm{d} t +{L_4^\alpha}' \mathrm{d} \mathcal {W}\Big),
\end{split}
\end{equation*}
and
\begin{equation*}
\begin{split}
\mathrm{d}\|u_S^{j,l}\|_{W^{k-1,p}}^p=\sum_{|\beta|\leq k-1}\Big((I_1^\beta+I_2^\beta+I_3^\beta+I_4^\beta)\mathrm{d} t+I_5^\beta \mathrm{d} \mathcal {W}\Big),
\end{split}
\end{equation*}
where $'$ denotes the derivative with respect to $t$. An application of the It\^{o}'s product rule leads to
\begin{equation}\label{4.27}
\begin{split}
 &\mathrm{d}(\|u_S^{j,l}\|_{W^{k-1,p}}^p\| u_S^{j}\|_{W^{k+1,p}}^p) \\
 &\quad = \sum_{|\alpha|\leq k+1} \|u_S^{j,l}\|_{W^{k-1,p}}^p({L_1^\alpha}' +{L_2^\alpha}' +{L_3^\alpha}' )\mathrm{d} t +\sum_{|\alpha|\leq k+1}\|u_S^{j,l}\|_{W^{k-1,p}}^p{L_4^\alpha}'\mathrm{d} \mathcal {W}\\
 &\quad\quad
+\sum_{|\beta|\leq k-1}\| u_S^{j}\|_{W^{k+1,p}}^p(I_1^\beta+I_2^\beta+I_3^\beta+I_4^\beta)\mathrm{d} t+\sum_{|\beta|\leq k-1}\| u_S^{j}\|_{W^{k+1,p}}^pI_5^\beta
 \mathrm{d} \mathcal {W}\\
 &\quad\quad + G^{j,l}(t)\mathrm{d} t,
\end{split}
\end{equation}
where the laet term $G^{j,l}(t)$ is given by
\begin{equation*}
\begin{split}
G^{j,l}(t):=&p^2\sum_{n\geq 1} \Bigg(\sum_{|\beta|\leq k-1}\int_D|\partial^\beta u_S^{j,l}|^{p-2}
\partial^\beta u_S^{j,l}\cdot \partial^\beta P\triangle \sigma ^{1,j,l}_n\mathrm{d}V \Bigg)\\
& \quad \quad \quad \quad \quad\quad \quad \quad\quad\times
\Bigg (\sum_{|\alpha|\leq k+1}\int_D|\partial^\alpha u_S^j|^{p-2} \partial^\alpha u_S^j\cdot \partial^\alpha P \sigma _{1}e_n\mathrm{d}V\Bigg).
\end{split}
\end{equation*}
In view of the estimates carried out in Section 3 and taking advantage of the condition (A1), we immediately infer  that over the interval $[0,\mathbbm{t}_{j,l,T}\wedge t]$,
\begin{equation}\label{4.28}
\begin{split}
 &\|u_S^{j,l}\|_{W^{k-1,p}}^p({L_1^\alpha}'+{L_2^\alpha}'+{L_3^\alpha}')\\
 &\quad \leq  \|\nabla u_S\| _{L^\infty}\| u_S \|_{X^{k+1,p}}^{ p}\|u_S^{j,l}\|_{W^{k-1,p}}^p+\kappa(\|u_S,u_T,\theta_S\|_{L^\infty})^2\|u_S^{j,l}\|_{W^{k-1,p}}^p
 \\
 &\quad\quad+(1+\kappa(\|u_S^j,u_T^j,\theta_S^j\|_{L^\infty})^2)\|u_S^{j,l}\|_{W^{k-1,p}}^p(\|u_S\|_{X^{k+1,p}}^p+\|u_T\|_{W^{k+1,p}}^p
 +\|\theta_S\|_{W^{k+1,p}}^p)\\
 & \quad \leq  C\|u_S^{j,l}\|_{W^{k-1,p}}^p\| u_S^j \|_{X^{k+1,p}}^{ p}
 \\
 &\quad\quad+C\|u_S^{j,l}\|_{W^{k-1,p}}^p(1+\|u_S^j\|_{X^{k+1,p}}^p+\|u_T^j\|_{W^{k+1,p}}^p
 +\|\theta_S\|_{W^{k+1,p}}^p).
\end{split}
\end{equation}
For the stochastic term associated  to ${L^\alpha_4}'$, we have
\begin{equation}\label{4.29}
\begin{split}
 &\mathbb{E}\sup_{s\in \mathbbm{t}_{j,l,T}\wedge t}\bigg|\int_0^s\|u_S^{j,l}(r)\|_{W^{k-1,p}}^p{L_4^\alpha}'\mathrm{d} \mathcal {W}\bigg|\\
 &\quad \leq C\mathbb{E}\Bigg(\int_0^{\mathbbm{t}_{j,l,T}\wedge t}\|u_S^{j,l}(r)\|_{X^{k-1,p}}^{2p}\\
 &\quad\quad  \times \sum_{n\geq1}\bigg(\int_D|\Lambda^{k+1}u_S^j
  |^{p-2}\Lambda^{k+1}u_S^j \cdot\Lambda^{k+1}P\sigma_1(u_S^j,u_T^j,\theta_S^j)e_n \mathrm{d}V\bigg)^{2}\mathrm{d} r\Bigg)^{1/2}\\
&\quad \leq C\mathbb{E}\bigg(\int_0^{\mathbbm{t}_{j,l,T}\wedge t}\|u_S^{j,l}(r)\|_{X^{k-1,p}}^{2p}
 (\|u_S^j\|_{X^{k+1,p}}^{2p}
  +\|u_T^j\|_{W^{k+1,p}}^{2p}
 +\|\theta_S\|_{W^{k+1,p}}^{2p})\mathrm{d} r\bigg)^{1/2}\\
&\quad \leq \frac{1}{5}\mathbb{E}\sup_{s\in [0,\mathbbm{t}_{j,l,T}\wedge t]} (\|u_S^{j,l}(s)\|_{W^{k-1,p}}^{p}\|(u_S^j,u_T^j,\theta_S^j)(s)\|_{\mathcal {Z}^{k+1,p}}^{p})\\
&\quad \quad +C\mathbb{E}\int_0^{\mathbbm{t}_{j,l,T}\wedge t}
 \|u_S^{j,l}(s)\|_{W^{k-1,p}}^{p} \|(u_S^j,u_T^j,\theta_S^j)(s)\|_{\mathcal {Z}^{k+1,p}}^{p}\mathrm{d} r .
 \end{split}
\end{equation}
For the stochastic term related to $I_5^\beta$, we get from the BDG inequality that
\begin{equation}\label{4.30}
\begin{split}
 &\mathbb{E}\sup_{s\in \mathbbm{t}_{j,l,T}\wedge t}\bigg|\int_0^s\|u_S^j(r)\|_{X^{k+1,p}}^pI_5^\beta\mathrm{d} \mathcal {W}\bigg|\\
 &\quad \leq C\mathbb{E}\bigg(\int_0^{\mathbbm{t}_{j,l,T}\wedge t}\|u_S^j\|_{X^{k+1,p}}^{2p}\|u_S^{j,l}
  \|_{X^{k-1,p}}^{2p-2}\\
  &\quad\quad\quad\quad\quad\times  \|P\sigma_1(u_S^j,u_T^j,\theta_S^j)-P\sigma_1(u_S^l,u_T^l,\theta_S^l)\|_{\mathbb{X}^{k-1,p}}^{2}\mathrm{d} r\bigg)^{1/2}\\
&\quad \leq  \frac{1}{4}\mathbb{E}\sup_{s\in [0,\mathbbm{t}_{j,l,T}\wedge t]} (\|u_S^{j,l}\|_{X^{k-1,p}}^{p}\|u_S^j\|_{X^{k+1,p}}^{p}) +C\mathbb{E}\int_0^{\mathbbm{t}_{j,l,T}\wedge t}
\|u_S^{j,l}\|_{X^{k-1,p}}^{p}\|u_S^j\|_{X^{k+1,p}}^{p}\mathrm{d} r.
 \end{split}
\end{equation}
Utilizing the H\"{o}lder inequality and the Minkowski inequality as well as the locally Lipschitz condition (A1),
 one can estimate the term $G^{j,l}(t)$ over $[0,\mathbbm{t}_{j,l,T}\wedge t]$ as
\begin{equation}\label{4.31}
\begin{split}
  |G^{j,l}|\leq& C\sum_{|\beta|\leq k-1}\bigg\|(\int_D|\partial^\beta u_S^{j,l}|^{p-2}
\partial^\beta u_S^{j,l}\cdot \partial^\beta P\triangle \sigma ^{1,j,l}_n\mathrm{d}V)_{n\geq1}\bigg\|_{l^2}\\
&\times \sum_{|\alpha|\leq k+1}
 \bigg\|(\int_D|\partial^\alpha u_S^j|^{p-2}\partial^\alpha u_S^j\cdot \partial^\alpha P \sigma _{1}e_n\mathrm{d}V)_{n\geq1}\bigg\|_{l^2}\\
 \leq&C\sum_{|\beta|\leq k-1}\int_D|\partial^\beta u_S^{j,l}|^{p-1}\|(\partial^\beta P\triangle \sigma ^{1,j,l}_n)_{n\geq1}\|_{l^2}\mathrm{d}V
\sum_{|\alpha|\leq k+1}\int_D|\partial^\alpha u_S^{j,l}|^{p-1}\|(\partial^\alpha P\sigma _{1}e_n)_{n\geq1}\|_{l^2}\mathrm{d}V\\
 \leq&C\sum_{|\beta|\leq k-1}\|\partial^\beta u_S^{j,l}\|^{p-1}_{L^p}\Big\|\|(\partial^\beta P\triangle \sigma ^{1,j,l}_n)_{n\geq1}\|_{l^2}\Big\|_{L^p}
\sum_{|\alpha|\leq k+1}\|\partial^\alpha u_S^j\|^{p-1}_{L^p}\Big\|\|(\partial^\alpha P\sigma _{1}e_n)_{n\geq1}\|_{l^2}\Big\|_{L^p}\\
\leq & C\|u_S^{j,l}\|^{p-1}_{X^{k-1,p}}\|u_S^j\|^{p-1}_{X^{k+1,p}}\Big(\|u_S^{j,l}\|_{X^{k-1,p}}+\|u_T^{j,l}\|_{W^{k-1,p}}+\|\theta_S^{j,l}\|_{W^{k-1,p}}\Big)
\\
&\times\Big(1+\|u_S^j\|_{X^{k+1,p}}+\|u_T^j\|_{W^{k+1,p}}+\|\theta_S\|_{W^{k+1,p}}\Big)\\
\leq & C\|(u_S^{j,l},u_T^{j,l},\theta_S^{j,l})\|_{\mathcal {Z}^{k-1,p}}^p
 (1+\|(u_S^j,u_T^j,\theta_S^j)\|_{\mathcal {Z}^{k+1,p}}^p),
\end{split}
\end{equation}
where the last two inequalities used the definition of the stopping time $\mathbbm{t}_{j,l,T}$ in Lemma \ref{lem:4.1}.

The drift term $I_1^\beta$ can be treated by following the decomposition method in \cite{41} (pp. 119), and one can deduce that for any $s\in [0,\mathbbm{t}_{j,l,T}\wedge t]$,
\begin{equation}\label{4.32}
\begin{split}
  \| u_S^{j}\|_{W^{k+1,p}}^p|I_1^\beta|\leq&  C \|u_S^{j,l}\|^{p }_{X^{k-1,p}}\| u_S^{j}\|_{X^{k+1,p}}^p(\|u_S^j\|^{p}_{X^{k,p}}+\|u_S^l\|^{p}_{X^{k,p}})\\
   \leq&  C \|u_S^{j,l}\|^{p }_{X^{k-1,p}}\| u_S^{j}\|_{X^{k+1,p}}^p .
\end{split}
\end{equation}
For $I_2^\beta$, we have
\begin{equation}\label{4.33}
\begin{split}
 \| u_S^{j}\|_{W^{k+1,p}}^p|I_2^\beta|\leq& p\| u_S^{j}\|_{W^{k+1,p}}^p \int_D|\partial^\beta u_S^{j,l}|^{p-1}|\partial^\beta P(u_T^{j,l}\widehat{x}+\theta _S^{j,l}\widehat{z}) |\mathrm{d}V\\
\leq& C\| u_S^{j}\|_{W^{k+1,p}}^p \|\partial^\beta u_S^{j,l}\|^{p-1}_{L^p}\|\partial^\beta P(u_T^{j,l}\widehat{x}
+\theta _S^{j,l}\widehat{z}) \|_{L^p}\\
\leq& C\| u_S^{j}\|_{W^{k+1,p}}^p \|u_S^{j,l}\|^{p-1}_{X^{k-1,p}}(\| u_T^{j,l} \|_{W^{k-1,p}}+
\|\theta _S^{j,l} \|_{W^{k-1,p}})\\
\leq& C  \| u_S^{j}\|_{W^{k+1,p}}^p\|(u_S^{j,l},u_T^{j,l},\theta_S^{j,l})\| _{\mathcal {Z}^{k-1,p}}^p,
\end{split}
\end{equation}
over the interval $[0,\mathbbm{t}_{j,l,T}\wedge t]$.

For $I_3^\beta$, we have for any $s\in [0,\mathbbm{t}_{j,l,T}\wedge t]$ that
\begin{equation}\label{4.34}
\begin{split}
 \| u_S^{j}\|_{X^{k+1,p}}^p|I_3^\beta+I_4^\beta|\leq&  C \| u_S^{j}\|_{X^{k+1,p}}^p\|\partial^\beta u_S^{j,l}\|_{L^p}^{p-2}
 \|\partial^\beta P\triangle \sigma ^{1,j,l}\|_{\mathbb{X}^{0,p}}^2\\
 \leq&  C \| u_S^{j}\|_{X^{k+1,p}}^p\| u_S^{j,l}\|_{X^{k-1,p}}^{p-2}
 \|\sigma_1(u_S^j,u_T^j,\theta_S^j)-\sigma_1(u_S^l,u_T^l,\theta_S^l)\|_{\mathbb{X}^{k-1,p}}^2\\
 \leq&  C\varsigma\Bigg( \sum_{\lambda\in\{j,l\}} (\| u_S^\lambda\|_{L^\infty}+\| u_T^\lambda\|_{L^\infty}+\| \theta_S^\lambda\|_{L^\infty})\Bigg)^2\| u_S^{j}\|_{X^{k+1,p}}^p\\
 &\times
\| u_S^{j,l}\|_{X^{k-1,p}}^{p-2}\left(\|u_S^{j,l}\|_{X^{k-1,p}}^2+\|u_T^{j,l}\|_{W^{k-1,p}}^2
+\|\theta_S^{j,l}\|_{W^{k-1,p}}^2\right)\\
 \leq&  C\| u_S^{j}\|_{X^{k+1,p}}^p
\|(u_S^{j,l},u_T^{j,l},\theta_S^{j,l})\| _{\mathcal {Z}^{k-1,p}}^p.
\end{split}
\end{equation}
Integrating the equality \eqref{4.27} with respect to $t$ and then taking the supremum over $[0,\mathbbm{t}_{j,l,T}\wedge t]$. After taking the expected values and applying the estimates \eqref{4.28}-\eqref{4.34}, we get
\begin{equation}\label{4.35}
\begin{split}
 & \mathbb{E}\sup_{s\in [0,\mathbbm{t}_{j,l,T}\wedge t]}\|u_S^{j,l}(s)\|_{X^{k-1,p}}^p
 \| u_S^{j}(s)\|_{X^{k+1,p}}^p \\
 &\quad\leq \mathbb{E}(\|u_S^{j,l}(0)\|_{X^{k-1,p}}^p\| u_S^{j}(0)\|_{X^{k+1,p}}^p) \\
 &\quad\quad+C\int_0^{\mathbbm{t}_{j,l,T}\wedge t}\|(u_S^{j,l},u_T^{j,l},\theta_S^{j,l})(s)\| _{\mathcal {Z}^{k-1,p}}^p
 (1+\|(u_S^j,u_T^j,\theta_S^j)(s)\|_{\mathcal {Z}^{k+1,p}}^{p})\mathrm{d} s.
 \end{split}
\end{equation}

In order to close the above inequality, we have to obtain estimates for the other terms involved in
$\|(u_S^{j,l},u_T^{j,l},\theta_S^{j,l})(s)\| _{\mathcal {Z}^{k-1,p}}^p
 \|(u_S^j,u_T^j,\theta_S^j)(s)\|_{\mathcal {Z}^{k+1,p}}^{p}$
Thanks to the equalities \eqref{2.28}, \eqref{2.34}, \eqref{4.16} and \eqref{4.21}, by replacing $k$ with $k-1$ and $k+1$ correspondingly, one can derive similar estimates for the other terms without any additional difficulties, so we shall omit the details to save the space. Finally, by summing up all these estimates we obtain
\begin{equation*}
\begin{split}
 & \mathbb{E}\sup_{s\in [0,\mathbbm{t}_{j,l,T}\wedge t]}(\|(u_S^{j,l},u_T^{j,l},\theta_S^{j,l})(s)\| _{\mathcal {Z}^{k-1,p}}^p
 \|(u_S^j,u_T^j,\theta_S^j)(s)\|_{\mathcal {Z}^{k+1,p}}^{p}) \\
 &\quad\leq C\mathbb{E}(\|(u_S^{j,l},u_T^{j,l},\theta_S^{j,l})(0)\| _{\mathcal {Z}^{k-1,p}}^p
 \|(u_S^j,u_T^j,\theta_S^j)(0)\|_{\mathcal {Z}^{k+1,p}}^{p}) \\
 &\quad\quad+C\int_0^t\Bigg(\mathbb{E}\sup_{s\in [0,\mathbbm{t}_{j,l,T}\wedge r]}(\|(u_S^{j,l},u_T^{j,l},\theta_S^{j,l})(s)
 \| _{\mathcal {Z}^{k-1,p}}^p\|(u_S^j,u_T^j,\theta_S^j)(s)\|_{\mathcal {Z}^{k+1,p}}^{p})
 \\
 &\quad\quad +\mathbb{E}\sup_{s\in [0,\mathbbm{t}_{j,l,T}\wedge r]}\|(u_S^{j,l},u_T^{j,l},\theta_S^{j,l})(s)\| _{\mathcal {Z}^{k-1,p}}^p\Bigg)\mathrm{d} r,
 \end{split}
\end{equation*}
which together with the Gronwall inequality lead to
\begin{equation}\label{4.36}
\begin{split}
 & \mathbb{E}\sup_{s\in [0,\mathbbm{t}_{j,l,T} ]}\Big(\|(u_S^{j,l},u_T^{j,l},\theta_S^{j,l})(s)\| _{\mathcal {Z}^{k-1,p}}^p
 \|(u_S^j,u_T^j,\theta_S^j)(s)\|_{\mathcal {Z}^{k+1,p}}^{p}\Big) \\
 &\quad\leq C\mathbb{E}\Big(\|(u_S^{j,l},u_T^{j,l},\theta_S^{j,l})(0)\| _{\mathcal {Z}^{k-1,p}}^p
 \|(u_S^j,u_T^j,\theta_S^j)(0)\|_{\mathcal {Z}^{k+1,p}}^{p}\Big) \\
 &\quad\quad+C \mathbb{E}\sup_{s\in [0,\mathbbm{t}_{j,l,T} ]}\|(u_S^{j,l},u_T^{j,l},\theta_S^{j,l})(s)\| _{\mathcal {Z}^{k-1,p}}^p
 . \end{split}
\end{equation}
Due to the properties of the smoothing operator, we have the convergence
$$
j\|u_S^{0,j}-u_S^0\|_{X^{k-1,p}}\rightarrow 0,\quad j\|u_T^{0,j}-u_T^0\|_{W^{k-1,p}}\rightarrow 0,
 \quad j\|\theta_S^{0,j}-\theta_S^0\|_{W^{k-1,p}}\rightarrow 0,
$$
as $j\rightarrow\infty$, and the bound
$$
\|(u_S^j,u_T^j,\theta_S^j)(0)\|_{\mathcal {Z}^{k+1,p}}^p \leq C j^p \|(u_S^j,u_T^j,\theta_S^j)(0)\|_{\mathcal {Z}^{k,p}}^p,
$$
for all $j\geq 1$. Using the last two properties, we have
\begin{equation} \label{000}
\begin{split}
&\lim_{j\rightarrow\infty}\sup_{l\geq j}\mathbb{E}(\|(u_S^{j,l},u_T^{j,l},\theta_S^{j,l})(0)\| _{\mathcal {Z}^{k-1,p}}^p
 \|(u_S^j,u_T^j,\theta_S^j)(0)\|_{\mathcal {Z}^{k+1,p}}^{p})\\
&\quad \leq C\|(u_S^j,u_T^j,\theta_S^j)(0)\|_{\mathcal {Z}^{k,p}}^p\\
&\quad\quad\times
\lim_{j\rightarrow\infty}\sup_{l\geq j}\mathbb{E}\bigg(j^p\|(u_S^{0,j}-u_S^0,u_T^{0,j}-u_T^0,\theta_S^{0,j}-\theta_S^0) \|_{\mathcal {Z}^{k-1,p}}^{p}
\\
&\quad\quad\quad\quad \quad\quad\quad+\frac{l^p}{j^p}\cdot j^p\|(u_S^{0,l}-u_S^0,u_T^{0,l}-u_T^0,\theta_S^{0,l}-\theta_S^0)
 \|_{\mathcal {Z}^{k-1,p}}^{p}\bigg)=0.
  \end{split}
\end{equation}
For the second term on the R.H.S of  \eqref{4.36}, we can use the similar estimates as those in the proof of pathwise uniqueness to show that
$$
\mathbb{E}\sup_{s\in [0,\mathbbm{t}_{j,l,T} ]}\|(u_S^{j,l},u_T^{j,l},\theta_S^{j,l})(s)\| _{\mathcal {Z}^{k-1,p}}^p\leq C\|(u_S^{j,l},u_T^{j,l},\theta_S^{j,l})(0)\| _{\mathcal {Z}^{k-1,p}}^p,
$$
where $C$ is a positive constant independent of $j$ and $l$. Moreover, since the approximate initial datum $(J_{\frac{1}{j}}u^0_S,J_{\frac{1}{j}}u^0_T,J_{\frac{1}{j}}\theta^0_S)_{j\geq1}$ formulate a Cauchy sequence in $\mathcal {Z}^{k,p}(D)$, we get  from the last inequality that
\begin{equation}\label{111}
\begin{split}
 &\lim_{j\rightarrow\infty}\sup_{l\geq j} \mathbb{E}\sup_{s\in [0,\mathbbm{t}_{j,l,T} ]}
 \|(u_S^{j,l},u_T^{j,l},\theta_S^{j,l})(s)\| _{\mathcal {Z}^{k-1,p}}^p \\
 &\quad \leq \lim_{j\rightarrow\infty}\sup_{l\geq j} \mathbb{E} \|(u_S^{j,l},u_T^{j,l},\theta_S^{j,l})(0)\| _{\mathcal {Z}^{k-1,p}}^p
 =0 .
 \end{split}
\end{equation}
With the estimates \eqref{000} and \eqref{111} at hand, we can conclude from \eqref{4.36} that
\begin{equation*}
\begin{split}
\lim_{j\rightarrow\infty}\sup_{l\geq j}  \mathbb{E}\sup_{s\in [0,\mathbbm{t}_{j,l,T} ]}(\|(u_S^{j,l},u_T^{j,l},\theta_S^{j,l})(s)\| _{\mathcal {Z}^{k-1,p}}^p
 \|(u_S^j,u_T^j,\theta_S^j)(s)\|_{\mathcal {Z}^{k+1,p}}^{p})=0,
   \end{split}
\end{equation*}
which implies the convergence \eqref{4.2}. The proof of Lemma \ref{lem:4.2} is now completed.
\end{proof}

\begin{lemma}\label{lem:4.3}
The convergence \eqref{4.3} holds for the sequence $(u_S^{j},u_T^{j},\theta_S^{j})_{j\geq 1}$ defined by (4.1).
\end{lemma}

\begin{proof}
By virtue of the inequalities as we established in Section 2, we have
\begin{equation*}
\begin{split}
& \sup_{s\in [0,\mathbbm{t}_{j,l,T}\wedge r]} \|(u_S^j,u_T^j,\theta_S^j)(s)\|_{\mathcal {Z}^{k,p}}\\
 &\quad
 \leq\|(u_S^{0,j},u_T^{0,j},\theta_S^{0,j})\|_{\mathcal {Z}^{k,p}}+\sum_{|\alpha|\leq k}\Bigg(\int_0^{\mathbbm{t}_{j,l,T}\wedge S} |{L_1^\alpha}'+{L_1^\alpha}'+{L_1^\alpha}'|\mathrm{d}s\\
 &\quad  \quad+ \int_0^{\mathbbm{t}_{j,l,T}\wedge r} |{J_1^\alpha}'+{J_2^\alpha}'+{J_3^\alpha}'+{J_4^\alpha}'|\mathrm{d}s
  + \int_0^{\mathbbm{t}_{j,l,T}\wedge r} |N_1^\alpha +N_2^\alpha+N_3^\alpha|\mathrm{d}s\Bigg)
 \\
 &\quad  \quad +\sum_{|\alpha|\leq k}\Bigg(\sup_{s\in [0,\mathbbm{t}_{j,l,T}\wedge r]}\bigg|\int_0^s {L_4^\alpha}'\mathrm{d} \mathcal {W}\bigg|+ \sup_{s\in [0,\mathbbm{t}_{j,l,T}\wedge S]}\bigg|\int_0^s {J_5^\alpha}'\mathrm{d} \mathcal {W}\bigg|
 \\
 & \quad  \quad  + \sup_{s\in [0,\mathbbm{t}_{j,l,T}\wedge r]}\bigg|\int_0^s N_4^\alpha\mathrm{d} \mathcal {W}\bigg|\Bigg),
 \end{split}
\end{equation*}
for any $j \geq 1$ and $r>0$. Thanks to above inequality, we can estimates the probability of the event in \eqref{4.3} as follows
\begin{equation}\label{4.38}
\begin{split}
&\mathbb{P} \Bigg\{\sup_{s\in [0,\mathbbm{t}_{j,k,T}\wedge r]}\|(u_S^j,u_T^j,\theta_S^j)(s)\|_{\mathcal {Z}^{k,p}}
> \|(u_S^{0,j},u_T^{0,j},\theta_S^{0,j})\|_{\mathcal {Z}^{k,p}}+1\Bigg\}\leq T_1(r)+T_2(r),
 \end{split}
\end{equation}
where
\begin{equation*}
\begin{split}
T_1(r):= &\mathbb{P} \Bigg\{\sum_{|\alpha|\leq k}\Bigg(\int_0^{\mathbbm{t}_{j,l,T}\wedge r} |{L_1^\alpha}'+{L_2^\alpha}'+{L_3^\alpha}'|\mathrm{d}s+ \int_0^{\mathbbm{t}_{j,l,T}\wedge r} |{J_1^\alpha}'+{J_2^\alpha}'+{J_3^\alpha}'+{J_4^\alpha}'|\mathrm{d}s\\
 & +  \int_0^{\mathbbm{t}_{j,l,T}\wedge r} |N_1^\alpha +N_2^\alpha+N_3^\alpha|\mathrm{d}s\Bigg) > \frac{1}{2}\Bigg\},
  \end{split}
\end{equation*}
and
\begin{equation*}
\begin{split}
 T_2(r):= & \mathbb{P} \Bigg\{\sum_{|\alpha|\leq k}\Bigg(\sup_{s\in [0,\mathbbm{t}_{j,l,T}\wedge r]}\bigg|\int_0^s {L_4^\alpha}'\mathrm{d} \mathcal {W}\bigg|+ \sup_{s\in [0,\mathbbm{t}_{j,l,T}\wedge S]}\bigg|\int_0^s {J_5^\alpha}'\mathrm{d} \mathcal {W}\bigg|\\
 &  + \sup_{s\in [0,\mathbbm{t}_{j,l,T}\wedge r]}\bigg|\int_0^s N_4^\alpha\mathrm{d} \mathcal {W}\bigg|\Bigg) > \frac{1}{2}\Bigg\} .
 \end{split}
\end{equation*}
For $T_1(r)$, we apply the estimates in \eqref{2.27}, \eqref{2.29}, \eqref{2.40} (using the definition of $\mathbbm{t}_{j,k,T}$ in this case), and then the Chebyshev inequality, we find
\begin{equation}\label{4.39}
\begin{split}
T_1(r) &\leq \mathbb{P} \bigg\{\int_0^{\mathbbm{t}_{j,l,T}\wedge r}
\Big(1+\|(\nabla u_S^j,\nabla u_T^j,\nabla\theta_S^j)\|_{\mathcal {Z}^{0,\infty}}
+\kappa(\|(u_S^j,u_T^j,\theta_S^j)\|_{\mathcal {Z}^{0,\infty}})^2\Big)\\
&\quad  \quad\quad  \quad \times \|(u_S^j,u_T^j,\theta_S^j)(s)\|_{\mathcal {Z}^{k,p}}^p \mathrm{d}s
> \frac{1}{2}\bigg\}\\
&\leq C\mathbb{E}\int_0^{\mathbbm{t}_{j,l,T}\wedge r}
\Big(1+\|(\nabla u_S^j,\nabla u_T^j,\nabla\theta_S^j)\|_{\mathcal {Z}^{0,\infty}}
+\kappa(\|(u_S^j,u_T^j,\theta_S^j)\|_{\mathcal {Z}^{0,\infty}})^2\Big)\\
&\quad  \quad\quad  \quad \times \|(u_S^j,u_T^j,\theta_S^j)(s)\|_{\mathcal {Z}^{k,p}}^p \mathrm{d}s\\
&\leq C\mathbb{E}\int_0^{\mathbbm{t}_{j,l,T}\wedge r}
\|(u_S^j,u_T^j,\theta_S^j)(s)\|_{\mathcal {Z}^{k,p}}^p \mathrm{d}s\\
&\leq C(2+CM)^pr,
 \end{split}
\end{equation}
where the positive constant $C$ is independent of $j$ and $r$. For $T_2(r)$, by using the BDG inequality, the Minkowski inequality and assumption (A1), we have
\begin{equation}\label{4.40}
\begin{split}
T_2(r) &\leq  \mathbb{P} \Bigg\{\sum_{|\alpha|\leq k}\sup_{s\in [0,\mathbbm{t}_{j,l,T}\wedge r]}
\bigg|\int_0^s {L_4^\alpha}'\mathrm{d} \mathcal {W}\bigg|> \frac{1}{6}\Bigg\}
+\mathbb{P} \Bigg\{\sum_{|\alpha|\leq k}\sup_{s\in [0,\mathbbm{t}_{j,l,T}\wedge S]}\bigg|  \int_0^s {J_5^\alpha}'\mathrm{d} \mathcal {W}\bigg|
> \frac{1}{6}\Bigg\}
 \\
 & \quad  +\mathbb{P} \Bigg\{\sum_{|\alpha|\leq k}\sup_{s\in [0,\mathbbm{t}_{j,l,T}\wedge r]}
 \bigg|\int_0^s N_4^\alpha\mathrm{d} \mathcal {W}\bigg| > \frac{1}{6}\bigg\}\\
 &\leq  36 \mathbb{E}\sum_{|\alpha|\leq k}\Bigg(
\bigg|\int_0^{\mathbbm{t}_{j,l,T}\wedge r} {L_4^\alpha}'\mathrm{d} \mathcal {W}\bigg|^2
+
\bigg|\int_0^{\mathbbm{t}_{j,l,T}\wedge r} {J_5^\alpha}'\mathrm{d} \mathcal {W}\bigg|^2
  +
\bigg|\int_0^{\mathbbm{t}_{j,l,T}\wedge r} N_4^\alpha\mathrm{d} \mathcal {W}\bigg|^2\Bigg)\\
& \leq  C \sum_{|\alpha|\leq k}\Bigg(\mathbb{E}
\int_0^{\mathbbm{t}_{j,l,T}\wedge r}\sum_{l\geq 1}\bigg(\int_D| \partial^\alpha u_S^j |^{ p-2} \partial^\alpha u_S^j \cdot  \partial^\alpha P\sigma_1(u_S^j,u_T^j,\theta_S^j)e_l \mathrm{d}V\bigg)^2\mathrm{d}t\\
&\quad\quad\quad\quad + \mathbb{E}
\int_0^{\mathbbm{t}_{j,l,T}\wedge r}\sum_{l\geq 1}\bigg(\int_D| \partial^\alpha u_T^j |^{ p-1}|  \partial^\alpha \sigma_2(u_S^j,u_T^j,\theta_S^j)e_l| \mathrm{d}V\bigg)^2\mathrm{d}t\\
&\quad \quad\quad\quad+ \mathbb{E}
\int_0^{\mathbbm{t}_{j,l,T}\wedge r}\sum_{l\geq 1}\bigg(\int_D| \partial^\alpha \theta_S^j |^{ p-1}|  \partial^\alpha \sigma_3(u_S^j,u_T^j,\theta_S^j)e_l| \mathrm{d}V\bigg)^2\mathrm{d}t\Bigg)\\
& \leq C \mathbb{E}
\int_0^{\mathbbm{t}_{j,l,T}\wedge r} \kappa(\|(u_S^j,u_T^j,\theta_S^j)\|_{\mathcal {Z}^{0,\infty}})^2\Big(1+\|(u_S^j,u_T^j,\theta_S^j)(s)\|_{\mathcal {Z}^{k,p}}^{2p}\Big) \mathrm{d}t\\
&\leq C \mathbb{E}
\int_0^{\mathbbm{t}_{j,l,T}\wedge r} \kappa(2+CM)^2\left(1+(2+CM)^{2p}\right) \mathrm{d}t\\
& \leq C \kappa(2+CM)^2\left(1+(2+CM)^{2p}\right)r.
 \end{split}
\end{equation}
where the positive constant $C$ is independent of $j$ and $r$, and the last inequality used the Sobolev embedding $\mathcal {Z}^{k,p}(D)\hookrightarrow\mathcal {Z}^{0,\infty}(D)$ and  the uniform bounds \eqref{4.4}-\eqref{4.5}. Therefore, the estimates \eqref{4.39}-\eqref{4.40} implies that
$$
T_1(r)+T_2(r)\rightarrow 0  \quad \mbox{as}~~r\rightarrow 0,
$$
which proves the desired convergence \eqref{4.3}. This completes the proof of Lemma \ref{lem:4.3}.
\end{proof}

Based on the above results, we can now give the proof of the first main result Theorem \ref{thm:1.2}.
\begin{proof}[\textbf{Proof of Theorem 1.2 (Local pathwise solution)}]
Assume that the uniform bound \eqref{4.5} holds for some deterministic constant $M>0$.
Thanks to the Lemma \ref{lem:4.2} and \ref{lem:4.3}, one can now conclude from Lemma \ref{lem:4.1} that there is
a strictly positive stopping time $\mathbbm{t}$ almost surely, a subsequence
$(u_S^{j_l},u_T^{j_l},\theta_S^{j_l})_{l\geq 1}$ of $(u_S^{j},u_T^{j},\theta_S^{j})_{j\geq 1}$
and a predictable process $(u_S,u_T,\theta_S)$ such that
\begin{equation}\label{4.41}
\begin{split}
(u_S^{j_l},u_T^{j_l},\theta_S^{j_l})_{l\geq 1} \longrightarrow (u_S,u_T,\theta_S)\quad \mbox{in}\quad
C([0,\mathbbm{t});\mathcal {Z}^{k,p}(D) ),\quad {as}~l\rightarrow\infty,\quad \mathbb{P}\mbox{-a.s.},
 \end{split}
\end{equation}
and
\begin{equation}\label{4.42}
\begin{split}
\sup_{t\in [0,\mathbbm{t})} \|(u_S,u_T,\theta_S)(t)\|_{\mathcal {Z}^{k,p}}\leq C<\infty,\quad \mathbb{P}\mbox{-a.s.}.
 \end{split}
\end{equation}
Since $(u_S^{j},u_T^{j},\theta_S^{j})$ are continuous $\mathcal {F}_t$-adapted processes with values
in $\mathcal {Z}^{k,p}(D)$, we infer that the processes $(u_S^{j},u_T^{j},\theta_S^{j})$ are $\mathcal {F}_t$-predictable.
Moreover, as the pointwise limits preserve the measurability, we thus conclude from \eqref{4.41} that the
limit process $(u_S,u_T,\theta_S)$ is also $\mathcal {F}_t$-predictable.
One may verify that $(u_S,u_T,\theta_S,\mathbbm{t})$ is a local pathwise solution
to the SISM \eqref{1.5} in the sense of Definition \ref{def:1.1}. Moreover, one can remove the restriction \eqref{4.5} by applying the
cutting argument as that in the proof of Lemma \ref{lem:3.3}. Finally, we can extend the local case to the maximal pathwise solutions
by a standard argument (cf. \cite{48}).

The proof of Theorem \ref{thm:1.2} is now completed.
\end{proof}

\section{Global pathwise solutions}

In this section, we are aiming at seeking for global pathwise solutions for the ISM perturbed by linear multiplicative stochastic forcing \eqref{1.8}.

Recalling that the potential temperature in ISM is defined by  (cf. \cite{1,2})
$$
\theta(x,y,z,t)=\theta(x,z,t)+(y-y_0)s,\quad s\in \mathbb{R},
$$
which varies linearly on the $y$-direction. In this section, we shall assume technically that $s=0$, that is, we do not consider the variation of temperature in $y$-direction. As we mentioned before, how to remove this assumption is a very challenging problem. To be more precise, we consider the following Cauchy problem:
\begin{equation}\label{5.1}
\left\{
\begin{aligned}
&\mathrm{d}u_S+  P(u_S\cdot \nabla)   u_S\mathrm{d}t-P(u_T\widehat{x}) \mathrm{d}t
- P(\theta _S\widehat{z}) \mathrm{d}t= \alpha u_S \mathrm{d}W,\\
&\mathrm{d} u_T+ (u_S\cdot \nabla)   u_T\mathrm{d}t+u_S\cdot\widehat{x} \mathrm{d}t =\alpha u_T \mathrm{d}W, \\
&\mathrm{d}\theta_S+ (u_S\cdot \nabla ) \theta_S\mathrm{d}t = \alpha\theta_S \mathrm{d}W\\
&u_S(0,x)=u_S^0,~~ u_T(0,x)=u_T^0,~~\theta_S(0,x)=\theta_S^0,
\end{aligned}
\right.
\end{equation}
where $\alpha \in \mathbb{R}$ is a real number, and $W$ is an one-dimension Brownian motion. As a special case of Theorem \ref{thm:1.2}, we conclude that the system \eqref{5.1} admits a local pathwise solution $(u_S, u_T, \theta_S,\mathbbm{t})$.

By virtue of the special structure of the system \eqref{5.1}, one may transform it into a system of random PDEs. To this end, we consider the stochastic process $\tilde{\alpha}(t)=e^{-\alpha W_t}$ and introduce new variables
\begin{equation}\label{5.2}
\begin{split}
\tilde{u}_S= \tilde{\alpha}(t) u_S,
\quad \tilde{u}_T= \tilde{\alpha}(t) u_T,
\quad \tilde{\theta}_S= \tilde{\alpha}(t) \theta_S.
 \end{split}
\end{equation}
Direct calculation shows that
$$
\mathrm{d} \tilde{\alpha}(t)=  \frac{1}{2}\alpha^2 \tilde{\alpha}(t) \mathrm{d}t - \alpha \tilde{\alpha}(t)\mathrm{d} W_t.
$$
An application of the It\^{o}'s product rule leads to
\begin{equation*}
\begin{split}
\mathrm{d}\tilde{u}_S=&\mathrm{d} (\tilde{\alpha}(t) u_S)
= u_S\mathrm{d}\tilde{\alpha}(t)+ \tilde{\alpha}(t) \mathrm{d}u_S+\mathrm{d} \tilde{\alpha}(t)\mathrm{d} u_S\\
=  & u_S\Big(\frac{1}{2}\alpha^2 \tilde{\alpha}(t) \mathrm{d}t - \alpha \tilde{\alpha}(t)\mathrm{d} W_t\Big)+\tilde{\alpha}(t) \Big(-P(u_S\cdot \nabla)   u_S\mathrm{d}t+P(u_T\widehat{x}) \mathrm{d}t\\
& + P(\theta _S\widehat{z}) \mathrm{d}t+\alpha u_S \mathrm{d}W\Big)
-  \alpha^2 \tilde{\alpha}(t)u_S \mathrm{d}t\\
=& -\tilde{\alpha}(t)P(u_S\cdot \nabla) u_S\mathrm{d}t+\tilde{\alpha}(t)P(u_T\widehat{x}) \mathrm{d}t
+ \tilde{\alpha}(t)P(\theta _S\widehat{z}) \mathrm{d}t-  \frac{1}{2}\alpha^2 \tilde{\alpha}(t)u_S \mathrm{d}t\\
=& -\tilde{\alpha}^{-1} P(\tilde{u}_S\cdot \nabla) \tilde{u}_S\mathrm{d}t+ P(\tilde{u}_T\widehat{x}) \mathrm{d}t
+   P(\tilde{\theta} _S\widehat{z}) \mathrm{d}t-  \frac{1}{2}\alpha^2  \tilde{u}_S \mathrm{d}t.
 \end{split}
\end{equation*}
By taking the similar calculation we deduce that
\begin{equation*}
\begin{split}
\mathrm{d} \tilde{u}_T=-\tilde{\alpha}(t) (u_S\cdot \nabla)   u_T\mathrm{d}t
-\tilde{\alpha}(t)u_S\cdot\widehat{x} \mathrm{d}t -  \frac{1}{2}\alpha^2  \tilde{u}_T \mathrm{d}t,
 \end{split}
\end{equation*}
and
\begin{equation*}
\begin{split}
\mathrm{d}\tilde{\theta}_S=-\tilde{\alpha}(t)(u_S\cdot \nabla ) \theta_S\mathrm{d}t
 -  \frac{1}{2}\alpha^2  \tilde{\theta}_S \mathrm{d}t.
 \end{split}
\end{equation*}
Therefore, the system \eqref{5.1} can be formulated as
\begin{equation}\label{5.3}
\left\{
\begin{aligned}
&\partial_t\tilde{u}_S+\frac{1}{2}\alpha^2  \tilde{u}_S+  \tilde{\alpha}^{-1} P(\tilde{u}_S\cdot \nabla)   \tilde{u}_S
- P(\tilde{u}_T\widehat{x})
-  P(\tilde{\theta} _S\widehat{z}) = 0,\\
&\partial_t\tilde{u}_T+\frac{1}{2}\alpha^2  \tilde{u}_T+ \tilde{\alpha}^{-1} (\tilde{u}_S\cdot \nabla)   \tilde{u}_T
+ \tilde{u}_S\cdot\widehat{x} =0, \\
&\partial_t\tilde{\theta}_S+\frac{1}{2}\alpha^2  \tilde{\theta}_S+ \tilde{\alpha}^{-1} (\tilde{u}_S\cdot \nabla ) \tilde{\theta}_S
 = 0,\\
&\tilde{u}_S(0,x)=u_S^0,~~ \tilde{u}_T(0,x)=u_T^0,~~\tilde{\theta}_S(0,x)=\theta_S^0.
\end{aligned}
\right.
\end{equation}
System (5.3) can be regarded as a damping perturbed ISM with random coefficients. In the following, we shall establish the global result by introducing suitable stopping times and carefully analysing the nonlinear estimates.

\begin{proof} [\textbf{Proof of Theorem \ref{thm:1.3} (Global pathwise solution)}.] The proof will be divided into several steps.

\emph{\textbf{Step 1 (Estimation for $\|(\tilde{u}_S ,\tilde{u}_T ,\tilde{\theta}_S )\|_{\mathcal {Z}^{k,p}}$).}}  We first apply  the operator $ \partial^\alpha $ to the first equation of \eqref{5.3} to get
\begin{equation}\label{5.4}
\begin{split}
&\partial_t \partial^\alpha \tilde{u}_S+\frac{1}{2}\alpha^2  \partial^\alpha \tilde{u}_S+ \tilde{\alpha}^{-1}  \partial^\alpha P(\tilde{u}_S\cdot \nabla)   \tilde{u}_S
-  \partial^\alpha P(\tilde{u}_T\widehat{x})  -   \partial^\alpha P(\tilde{\theta} _S\widehat{z}) = 0.
 \end{split}
\end{equation}
Multiplying \eqref{5.4} by $ \partial^\alpha \tilde{u}_S| \partial^\alpha \tilde{u}_S|^{p-2}$ and integrating on $D$, we get
\begin{equation*}
\begin{split}
&\frac{1}{p}\frac{\mathrm{d}}{\mathrm{d} t}\| \partial^\alpha \tilde{u}_S(t) \|_{L^p}^p+\frac{\alpha^2}{2}  \| \partial^\alpha \tilde{u}_S\|_{L^p}^p
+ \tilde{\alpha}^{-1} \int_D| \partial^\alpha \tilde{u}_S|^{p-2} \partial^\alpha \tilde{u}_S\cdot \partial^\alpha P(\tilde{u}_S\cdot \nabla)   \tilde{u}_S\mathrm{d}V\mathrm{d}t
\\
&\quad- \int_D| \partial^\alpha \tilde{u}_S|^{p-2} \partial^\alpha \tilde{u}_S\cdot \partial^\alpha P(\tilde{u}_T\widehat{x})\mathrm{d}V \mathrm{d}t
\\
&\quad- \int_D| \partial^\alpha \tilde{u}_S|^{p-2} \partial^\alpha \tilde{u}_S\cdot \partial^\alpha P(\tilde{\theta} _S\widehat{z})\mathrm{d}V \mathrm{d}t= 0.
 \end{split}
\end{equation*}
It then follows from the estimates \eqref{2.6}, \eqref{2.8} and the H\"{o}lder inequality  that
\begin{equation}\label{5.5}
\begin{split}
&\frac{\mathrm{d}}{\mathrm{d} t}\| \partial^\alpha \tilde{u}_S(t) \|_{L^p}^p+\frac{\alpha^2}{2}  \| \partial^\alpha \tilde{u}_S\|_{L^p}^p\\
&\quad \leq   C\tilde{\alpha}^{-1}  \|\nabla \tilde{u}_S\| _{L^\infty}\| \partial^\alpha  \tilde{u}_S \|_{L^p}^{ p}
+ C (\| \partial^\alpha  \tilde{u}_S \|_{L^p}^{ p}+ \| \partial^\alpha  \tilde{u}_T \|_{L^p}^{ p} )
\\
&\quad\quad+ C (\| \partial^\alpha  \tilde{u}_S \|_{L^p}^{ p}
+ \| \partial^\alpha  \tilde{\theta}_S\|_{L^p}^{ p} )\\
&\quad\leq  C(1+\tilde{\alpha}^{-1} \|\nabla \tilde{u}_S\| _{L^\infty}) (\| \partial^\alpha  \tilde{u}_S \|_{L^p}^{ p}+
 \| \partial^\alpha  \tilde{u}_T \|_{L^p}^{ p}+ \| \partial^\alpha  \tilde{\theta}_S\|_{L^p}^{ p} ).
\end{split}
\end{equation}
Second, by applying the operator $ \partial^\alpha $ to the second equation of \eqref{5.3}, multiplying the resulted equation by
 $ \partial^\alpha \tilde{u}_T| \partial^\alpha \tilde{u}_T|^{p-2}$
and then integrating by parts on $D$, we get
\begin{equation}\label{5.6}
\begin{split}
&\frac{1}{p}\frac{\mathrm{d}}{\mathrm{d} t}\| \partial^\alpha \tilde{u}_T(t) \|_{L^p}^p +\frac{\alpha^2}{2}  \| \partial^\alpha \tilde{u}_T\|_{L^p}^p\\
&
\quad\leq \tilde{\alpha}^{-1} \bigg|\int_D| \partial^\alpha \tilde{u}_T|^{p-2}
 \partial^\alpha \tilde{u}_T[ \partial^\alpha ,\tilde{u}_S]\cdot \nabla \tilde{u}_T\mathrm{d}V \bigg|
+ \bigg|\int_D| \partial^\alpha \tilde{u}_T|^{p-2}
 \partial^\alpha \tilde{u}_T \partial^\alpha (\tilde{u}_S\cdot\widehat{x})\mathrm{d}V  \bigg|\\
  &\quad\leq \tilde{\alpha}^{-1}   \| \partial^\alpha  \tilde{u}_T \|_{L^p}^{ p-1}\|[ \partial^\alpha ,\tilde{u}_S]\cdot \nabla \tilde{u}_T\|_{L^p} + \| \partial^\alpha  \tilde{u}_T \|_{L^p}^{ p-1}\| \partial^\alpha (\tilde{u}_S\cdot\widehat{x})\|_{L^p} \\
&\quad\leq C\tilde{\alpha}^{-1} \| \partial^\alpha  \tilde{u}_T \|_{L^p}^{ p-1}(\|\nabla \tilde{u}_S\|_{L^\infty}\|\Lambda^{k-1} \nabla\tilde{u}_T \|_{L^p}+\|\Lambda^{k}\tilde{u}_S \|_{L^p}\|\nabla \tilde{u}_T\|_{L^\infty} )\\
&\quad\quad +  \| \partial^\alpha  \tilde{u}_T \|_{L^p} ^p+\| \partial^\alpha  \tilde{u}_S \|_{L^p}^p \\
&\quad\leq C(1+\tilde{\alpha}^{-1}(\|\nabla \tilde{u}_S\|_{L^\infty}+\|\nabla \tilde{u}_T\|_{L^\infty}) ) (\| \partial^\alpha  \tilde{u}_S \|_{L^p}^p+\| \partial^\alpha  \tilde{u}_T \|_{L^p}^p).
 \end{split}
\end{equation}
Finally, by applying the operator $ \partial^\alpha $ to the third equation of \eqref{5.3}, multiplying the resulted equation by
 $ \partial^\alpha \tilde{\theta}_S| \partial^\alpha \tilde{\theta}_S|^{p-2}$
and then integrating on $D$. Thanks to the free divergence condition for $\tilde{u}_S$ and the commutator estimate, we have
\begin{equation}\label{5.7}
\begin{split}
&\frac{1}{p}\frac{\mathrm{d}}{\mathrm{d} t}\| \partial^\alpha \tilde{\theta}_S(t) \|_{L^p}^p +\frac{\alpha^2}{2}  \| \partial^\alpha \tilde{\theta}_S\|_{L^p}^p\\
&\quad\leq \tilde{\alpha}^{-1} \bigg|\int_D| \partial^\alpha \tilde{\theta}_S|^{p-2}
 \partial^\alpha \tilde{\theta}_S[ \partial^\alpha ,\tilde{u}_S]\cdot \nabla \tilde{\theta}_S\mathrm{d}V \bigg|
+  \bigg|\int_D| \partial^\alpha \tilde{\theta}_S|^{p-2}
 \partial^\alpha \tilde{\theta}_S\tilde{u}_T\mathrm{d}V\bigg|    \\
 &\quad\leq   C\tilde{\alpha}^{-1} \| \partial^\alpha  \tilde{\theta}_S \|_{L^p}^{p-1}(\|\nabla \tilde{u}_S\|_{L^\infty}\|\Lambda^{k-1} \nabla\tilde{\theta}_S \|_{L^p}+\|\Lambda^{k}\tilde{u}_S \|_{L^p}\|\nabla \tilde{\theta}_S\|_{L^\infty} )\\
 &\quad\quad+ \|\Lambda^{k}\tilde{\theta}_S \|_{L^p}^p+\|\Lambda^{k}\tilde{u}_T \|_{L^p}^p\\
&\quad\leq C (1+\tilde{\alpha}^{-1}(\|\nabla \tilde{\theta}_S\|_{L^\infty}+\|\nabla \tilde{u}_S\|_{L^\infty})) (\|\Lambda^{k}\tilde{u}_S \|_{L^p}+\|\Lambda^{k}\tilde{\theta}_S \|_{L^p}).
 \end{split}
\end{equation}
Combining the inequalities \eqref{5.5}-\eqref{5.7}, we deduce that
\begin{equation}\label{5.8}
\begin{split}
 &\frac{1}{p}\frac{\mathrm{d}}{\mathrm{d} t} \|(\tilde{u}_S ,\tilde{u}_T ,\tilde{\theta}_S )(t)\|_{\mathcal {Z}^{k,p}}^p+\frac{\alpha^2}{2}\|(\tilde{u}_S ,\tilde{u}_T ,\tilde{\theta}_S )(t)\|_{\mathcal {Z}^{k,p}}^p \\
 &\quad\leq  C_1((1+\tilde{\alpha}^{-1}(\|\nabla \tilde{u}_S\|_{L^\infty}+\|\nabla \tilde{\theta}_S\|_{L^\infty}+\|\nabla \tilde{u}_T\|_{L^\infty}))\|(\tilde{u}_S ,\tilde{u}_T ,\tilde{\theta}_S )(t)\|_{\mathcal {Z}^{k,p}}^p,
 \end{split}
\end{equation}
which implies that
\begin{equation}\label{5.9}
\begin{split}
 &\frac{\mathrm{d}}{\mathrm{d} t} \|(\tilde{u}_S ,\tilde{u}_T ,\tilde{\theta}_S )(t)\|_{\mathcal {Z}^{k,p}}+\frac{\alpha^2}{2}\|(\tilde{u}_S ,\tilde{u}_T ,\tilde{\theta}_S )(t)\|_{\mathcal {Z}^{k,p}} \\
 &\quad\leq  C_1(1+\tilde{\alpha}^{-1}(\|\nabla \tilde{u}_S\|_{L^\infty}+\|\nabla \tilde{\theta}_S\|_{L^\infty}+\|\nabla \tilde{u}_T\|_{L^\infty}))\|(\tilde{u}_S ,\tilde{u}_T ,\tilde{\theta}_S )(t)\|_{\mathcal {Z}^{k,p}}.
 \end{split}
\end{equation}
Note that in order to derive a bound for  $\|(\tilde{u}_S ,\tilde{u}_T ,\tilde{\theta}_S )(t)\|_{\mathcal {Z}^{k,p}}$, one have to derive properly estimates for the gradients $\nabla\tilde{u}_S$, $\nabla\tilde{u}_T $ and $\nabla\tilde{\theta}_S$.

\emph{\textbf{Step 2 (Estimation for $\|(\omega_S ,\nabla\tilde{u}_T ,\nabla\tilde{\theta}_S )(t)\|_{\mathcal {Z}^{0,\infty}}$).}}  In order to treat the norm $\|\nabla \tilde{u}_S\|_{L^\infty}$ on the R.H.S of \eqref{5.9}, we use the Beale-Kato-Majda-type inequality to obtain
\begin{equation}\label{5.10}
\begin{split}
\|\nabla\tilde{u}_S\|_{L^{\infty}}\leq C_2\|\tilde{u}_S\|_{L^2}+C_2 \|\mbox{curl}~ \tilde{u}_S\|_{L^{\infty}}\bigg(1+\log^+\bigg(\frac{\|\tilde{u}_S\|_{W^{k,p}}}{\|\mbox{curl}~\tilde{u}_S\|_{L^{\infty}}}\bigg)   \bigg),
 \end{split}
\end{equation}
which leads us to consider the equation satisfied by the vorticity $\omega_S=\mbox{curl}~\tilde{u}_S$ and the $L^2$-estimate for $\tilde{u}_S$. To this end, we apply the vorticity operator $\nabla^\bot\cdot=(-\partial_z,\partial_x)\cdot$ to the first equation of \eqref{5.3} to deduce that
\begin{equation}\label{5.11}
\begin{split}
\partial_t \omega_S+\frac{1}{2}\alpha^2  \omega_S+  \tilde{\alpha}^{-1} P(\tilde{u}_S\cdot \nabla)   \omega_S
- P\partial_x\tilde{u}_T +P \partial_z\tilde{\theta} _S  = 0.
 \end{split}
\end{equation}
Multiplying both sides of \eqref{5.11} by $\omega_S|\omega_S|^{p-2}$; after integration by parts on $D$ and using the identity $(P(\tilde{u}_S\cdot \nabla) \omega_S,\omega_S|\omega_S|^{p-2})_{L^2}=0$, we get
\begin{equation}\label{5.12}
\begin{split}
\frac{1}{p}\frac{\mathrm{d}}{\mathrm{d} t} \|\omega_S\|_{L^p}^p+\frac{1}{2}\alpha^2  \|\omega_S\|_{L^p}^p\leq&  \int_D \left||\omega_S|^{p-2}\omega_S(P\partial_x\tilde{u}_T +P \partial_z\tilde{\theta} _S)\right| \mathrm{d} V\\
\leq& C(\|\omega_S\|_{L^p}^p+\|\nabla \tilde{u}_T\|_{L^p}^p+\|\nabla\tilde{\theta}_S\|_{L^p}^p).
 \end{split}
\end{equation}
To estimate the term $\|\nabla \tilde{u}_T\|_{L^p}$ involved in \eqref{5.9}, we apply the operator to $\nabla^\bot$ to the second equation of \eqref{5.3} to obtain
\begin{equation}\label{5.13}
\begin{split}
&\partial_t\nabla^\bot\tilde{u}_T+\frac{1}{2}\alpha^2 \nabla^\bot \tilde{u}_T+ \tilde{\alpha}^{-1} \tilde{u}_S\cdot \nabla \nabla^\bot \tilde{u}_T
+ \nabla^\bot(\tilde{u}_S\cdot\widehat{x}) =\tilde{\alpha}^{-1} \nabla\tilde{u}_S\cdot \nabla    ^\bot\tilde{u}_T.
 \end{split}
\end{equation}
Multiplying both sides of \eqref{5.13} by $\nabla^\bot\tilde{u}_T|\nabla^\bot\tilde{u}_T|^{p-2}$, integrating on $D$ and using the Biot-Savart law for the vorticity $\omega_S$, we deduce that
\begin{equation}\label{5.14}
\begin{split}
&\frac{1}{p}\frac{\mathrm{d}}{\mathrm{d} t} \|\nabla \tilde{u}_T\|_{L^p}^p+\frac{1}{2}\alpha^2  \|\nabla \tilde{u}_T\|_{L^p}^p\\
&\quad \leq \int_D |\nabla^\bot(\tilde{u}_S\cdot\widehat{x})-\tilde{\alpha}^{-1} \nabla\tilde{u}_S\cdot \nabla ^\bot   \tilde{u}_T ||\nabla^\bot\tilde{u}_T|^{p-1}\mathrm{d}  V\\
& \quad \leq\|\omega_S\|_{L^p}^p+\tilde{\alpha}^{-1} \|\nabla\tilde{u}_S\|_{L^\infty} \|\nabla \tilde{u}_T\|_{L^p}^p + \|\nabla \tilde{u}_T\|_{L^p}^p.
 \end{split}
\end{equation}
To estimate the term $\|\nabla \tilde{\theta}_S\|_{L^p}$ involved in \eqref{5.9}, we apply the operator $\nabla^\bot$ to the third equation of \eqref{5.3} to obtain
\begin{equation}\label{5.15}
\begin{split}
\partial_t\nabla^\bot\tilde{\theta}_S+\frac{1}{2}\alpha^2  \nabla^\bot\tilde{\theta}_S+ \tilde{\alpha}^{-1}  \tilde{u}_S\cdot \nabla   \nabla^\bot\tilde{\theta}_S= \tilde{\alpha}^{-1}  \nabla\tilde{u}_S\cdot \nabla^\bot\tilde{\theta}_S.
 \end{split}
\end{equation}
Multiplying both sides of \eqref{5.15} by $\nabla^\bot\tilde{\theta}_S|\nabla^\bot\tilde{\theta}_S|^{p-2}$; after integration by parts and using the divergence-free nature of $\tilde{u}_S$, one can deduce that
\begin{equation}\label{5.16}
\begin{split}
\frac{1}{p}\frac{\mathrm{d}}{\mathrm{d} t} \|\nabla \tilde{\theta}_S\|_{L^p}^p+\frac{1}{2}\alpha^2  \|\nabla \tilde{\theta}_S\|_{L^p}^p &\leq  \tilde{\alpha}^{-1}\bigg|\int_D |\nabla^\bot\tilde{\theta}_S|^{p-2}\nabla^\bot\tilde{\theta}_S  \nabla\tilde{u}_S \nabla^\bot\tilde{\theta}_S   \mathrm{d} V\bigg| \\
&\leq  \tilde{\alpha}^{-1}\|\nabla\tilde{u}_S\|_{L^\infty} \|\nabla\tilde{\theta}_S\|_{L^p}^p .
 \end{split}
\end{equation}
Putting the estimates \eqref{5.12}, \eqref{5.14} and \eqref{5.16} together leads to
\begin{equation*}
\begin{split}
&\frac{1}{p}\frac{\mathrm{d}}{\mathrm{d} t} (\|\omega_S\|_{L^p}^p+ \|\nabla \tilde{u}_T\|_{L^p}^p+\|\nabla \tilde{\theta}_S\|_{L^p}^p)+\frac{\alpha^2 }{2}  (\|\omega_S\|_{L^p}^p+ \|\nabla \tilde{u}_T\|_{L^p}^p+\|\nabla \tilde{\theta}_S\|_{L^p}^p)\\
&\quad \leq C (1+\tilde{\alpha}^{-1}\|\nabla\tilde{u}_S\|_{L^\infty} ) (\|\omega_S\|_{L^p}^p+\|\nabla \tilde{u}_T\|_{L^p}^p+\|\nabla\tilde{\theta}_S\|_{L^p}^p).
 \end{split}
\end{equation*}
Thanks to the equivalence $(a+b+c)^p\approx a^p+b^p+c^p$, for any $a,b,c >0$ and positive integer $p$, the last inequality implies that
\begin{equation*}
\begin{split}
&\frac{\mathrm{d}}{\mathrm{d} t} \|(\omega_S,\nabla \tilde{u}_T,\nabla \tilde{\theta}_S)\|_{L^p}+\frac{\alpha^2 }{2 } \|(\omega_S,\nabla \tilde{u}_T,\nabla \tilde{\theta}_S)\|_{L^p}\\
&\quad \leq C_3 (1+\tilde{\alpha}^{-1}\|\nabla\tilde{u}_S\|_{L^\infty} ) \|(\omega_S,\nabla \tilde{u}_T,\nabla \tilde{\theta}_S)\|_{L^p}.
 \end{split}
\end{equation*}
Sending $p\rightarrow\infty$ in above inequality, we obtain
\begin{equation}\label{5.17}
\begin{split}
&\frac{\mathrm{d}}{\mathrm{d} t} \|(\omega_S,\nabla \tilde{u}_T,\nabla \tilde{\theta}_S)\|_{L^\infty}+\frac{\alpha^2 }{2 } \|(\omega_S,\nabla \tilde{u}_T,\nabla \tilde{\theta}_S)\|_{L^\infty}\\
&\quad \leq C_3 (1+\tilde{\alpha}^{-1}\|\nabla\tilde{u}_S\|_{L^\infty} ) \|(\omega_S,\nabla \tilde{u}_T,\nabla \tilde{\theta}_S)\|_{L^\infty}.
 \end{split}
\end{equation}
Finally, let us deal with the $L^2$-estimation appeared in \eqref{5.9}. Due to the cancellation property $(\tilde{\alpha}^{-1} P(\tilde{u}_S\cdot \nabla)   \tilde{u}_S,\tilde{u}_S)=0$, after multiplying the first equation \eqref{5.3} by $\tilde{u}_S$ and integrating on $D$, we have
\begin{equation}\label{5.18}
\begin{split}
\frac{\mathrm{d}}{\mathrm{d} t}\|\tilde{u}_S \|_{L^2}^2 +\frac{\alpha^2}{2}\|\tilde{u}_S \|_{L^2}^2
\leq& \bigg|\int_D \tilde{u}_S P(\tilde{u}_T\widehat{x}) \mathrm{d} V\bigg|
+ \bigg|\int_D\tilde{u}_S P(\tilde{\theta} _S\widehat{z})\mathrm{d} V \bigg|\\
\leq& C( \|\tilde{u}_S \|_{L^2}\|\tilde{u}_T \|_{L^2}
+ \|\tilde{u}_S \|_{L^2}\|\tilde{\theta}_S \|_{L^2})\\
\leq& C(\|\tilde{u}_S \|_{L^2}^2
+\|\tilde{u}_T \|_{L^2}^2+\|\tilde{\theta}_S \|_{L^2}^2).
 \end{split}
\end{equation}
Similarly, by using the properties $(\tilde{\alpha}^{-1} P(\tilde{u}_S\cdot \nabla)   \tilde{u}_T,\tilde{u}_T)=0$ and $(\tilde{\alpha}^{-1} P(\tilde{u}_S\cdot \nabla)   \tilde{\theta}_S,\tilde{\theta}_S)=0$, one can deduce that
\begin{equation}\label{5.19}
\begin{split}
\frac{\mathrm{d}}{\mathrm{d} t}\|\tilde{u}_T \|_{L^2}^2 +\frac{\alpha^2}{2}\|\tilde{u}_T \|_{L^2}^2
&\leq C(\|\tilde{u}_S \|_{L^2}^2+\|\tilde{u}_T \|_{L^2}^2),\\
\frac{\mathrm{d}}{\mathrm{d} t}\|\tilde{\theta}_S \|_{L^2}^2 +\frac{\alpha^2}{2}\|\tilde{\theta}_S \|_{L^2}^2
&\leq 0.
 \end{split}
\end{equation}
Therefore, we get by putting the estimates \eqref{5.18}-\eqref{5.19} together
\begin{equation*}
\begin{split}
&\frac{\mathrm{d}}{\mathrm{d} t}(\|\tilde{u}_S \|_{L^2}^2+\|\tilde{u}_T \|_{L^2}^2+\|\tilde{\theta}_S \|_{L^2}^2) +\frac{\alpha^2}{2}(\|\tilde{u}_S \|_{L^2}^2+\|\tilde{u}_T \|_{L^2}^2+\|\tilde{\theta}_S \|_{L^2}^2)\\
&
\quad \leq C (\|\tilde{u}_S \|_{L^2}^2+\|\tilde{u}_T \|_{L^2}^2+\|\tilde{\theta}_S \|_{L^2}^2),
 \end{split}
\end{equation*}
where the positive constant $C$ depends only on $p$ and $D$. By choosing  $|\alpha|$ large enough such that $\frac{\alpha^2}{2}-C\geq \frac{\alpha^2}{4}$, one can use the Gronwall inequality to last estimates to get
\begin{equation}\label{5.20}
\begin{split}
 \|\tilde{u}_S \|_{L^2}^2+\|\tilde{u}_T \|_{L^2}^2+\|\tilde{\theta}_S \|_{L^2}^2\leq e^{-\frac{\alpha^2 t}{4}}(\|u_S^0 \|_{L^2}^2+\|u_T^0 \|_{L^2}^2+\|\theta_S ^0\|_{L^2}^2),
 \end{split}
\end{equation}
for all $t>0$.

For simplicity, we also fix the Sobolev embedding constant $C_3$ as
\begin{equation}\label{5.21}
\begin{split}
\|(\omega_S,\nabla \tilde{u}_T,\nabla \tilde{\theta}_S)\|_{\mathcal {Z}^{0,2}}+\|(\tilde{u}_S, \tilde{u}_T, \tilde{\theta}_S)\|_{\mathcal {Z}^{1,\infty}}\leq C_3 \|(\tilde{u}_S, \tilde{u}_T, \tilde{\theta}_S)\|_{\mathcal {Z}^{k,p}}.
 \end{split}
\end{equation}
and set $\tilde{C}$ as
$$
\tilde{C}:= 1+C_1+C_3+2C_1C_2+C_1C_2C_3.
$$

\emph{\textbf{Step 3 (Bound for $\|(u_S, u_T, \theta_S)(t)\|_{\mathcal {Z}^{k,p}}$).}} Let us introduce two processes
\begin{equation*}
\begin{split}
\psi(t):= \|(\omega_S,\nabla \tilde{u}_T,\nabla \tilde{\theta}_S)(t)\|_{L^\infty}\exp\{\frac{\alpha^2 t}{8}\},
 \end{split}
\end{equation*}
and
\begin{equation*}
\begin{split}
\varphi(t):= \|(\tilde{u}_S, \tilde{u}_T, \tilde{\theta}_S)(t)\|_{\mathcal {Z}^{k,p}}\exp\{\frac{\alpha^2 t}{8}\}.
 \end{split}
\end{equation*}
Moreover, in order to exploiting the damping in \eqref{5.9}, we consider the stopping time
\begin{equation*}
\begin{split}
 \tilde{\mathbbm{t}}:=&\inf_{t\geq0}\bigg\{t;~~ 1+\tilde{\alpha}^{-1}(\|\tilde{u}_S\|_{X^{k,p}}+\| \tilde{\theta}_S\|_{X^{k,p}}+\|\tilde{u}_T\|_{X^{k,p}})\geq \frac{|\alpha|}{8\tilde{C}}\bigg\}\\
=&\inf_{t\geq0}\bigg\{t;~~ 1+ \|u_S\|_{X^{k,p}}+\|u_T\|_{X^{k,p}}+\| \theta_S\|_{X^{k,p}}\geq \frac{|\alpha|}{8\tilde{C}}\bigg\}.
 \end{split}
\end{equation*}
Due  to the Sobolev embedding theorem, we have for any $t\in [0,\tilde{\mathbbm{t}}]$
\begin{equation}\label{5.22}
\begin{split}
1+\tilde{\alpha}^{-1}\|(\omega_S,\nabla \tilde{u}_T,\nabla \tilde{\theta}_S)\|_{L^\infty} \leq 1+ \tilde{\alpha}^{-1}\|(\tilde{u}_S, \tilde{u}_T, \tilde{\theta}_S)\|_{\mathcal {Z}^{1,\infty}}\leq \frac{|\alpha|}{8}.
 \end{split}
\end{equation}
Thereby it follows from the estimates \eqref{5.9} and \eqref{5.22} that
\begin{equation}\label{5.23}
\begin{split}
\frac{\mathrm{d}}{\mathrm{d} t} \|(\omega_S,\nabla \tilde{u}_T,\nabla \tilde{\theta}_S)\|_{L^\infty}+(\frac{ \alpha^2 }{2}-\frac{|\alpha|}{8}) \|(\omega_S,\nabla \tilde{u}_T,\nabla \tilde{\theta}_S)\|_{L^\infty}\leq 0,
 \end{split}
\end{equation}
for all $t\in [0,\tilde{\mathbbm{t}}]$.
After choosing $|\alpha|>1$ large enough such that $\frac{ \alpha^2 }{2}-\frac{|\alpha|}{8}\geq \frac{ \alpha^2 }{8}$, one can apply the Gronwall lemma to \eqref{5.23} and use the fact of $\tilde{\alpha}^{-1}(0)=1$ to get
\begin{equation}\label{5.24}
\begin{split}
 \psi(t)\leq \psi(0)= \|(\omega_S^0,\nabla u_T^0,\nabla \theta_S^0)\|_{L^\infty}\leq \frac{|\alpha| }{8}-1\leq \frac{\alpha^2 }{8}.
 \end{split}
\end{equation}
By virtue of the condition
$$
\|(u_S^0, u_T^0, \theta_S^0)\|_{\mathcal {Z}^{k,p}}\leq \frac{\alpha^2}{8\tilde{C}},
$$
and  applying the estimates \eqref{5.22}-\eqref{5.24}, we obtain
\begin{equation}\label{5.25}
\begin{split}
 \frac{\mathrm{d} \varphi}{\mathrm{d} t}\leq&C_1\bigg(1+\tilde{\alpha}^{-1}\|\nabla \tilde{\theta}_S\|_{L^\infty}+\tilde{\alpha}^{-1}\|\nabla \tilde{u}_T\|_{L^\infty}\\
 &+ C_2\tilde{\alpha}^{-1} \|\omega_S\|_{L^{\infty}}+ C_2\tilde{\alpha}^{-1}\|\tilde{u}_S\|_{L^2}+C_2\tilde{\alpha}^{-1} \|\omega_S\|_{L^{\infty}}\log^+\bigg(\frac{\varphi}{\psi}\bigg)  \bigg)\varphi\\
\leq&\tilde{C}\bigg (\frac{|\alpha|}{8} +\tilde{\alpha}^{-1}\exp\{-\frac{\alpha^2 t}{8}\}  \psi  +\tilde{\alpha}^{-1}\exp\{-\frac{\alpha^2 t}{8}\}\frac{\alpha^2}{8 }\\
&+\tilde{\alpha}^{-1} \exp\{-\frac{\alpha^2 t}{8}\}\psi\log^+\bigg(\frac{\varphi}{\psi}\bigg)  \bigg)\varphi.
 \end{split}
\end{equation}
To get a better understanding of \eqref{5.5}, we first introduce the following transformations:
$$
\tilde{\varphi}(t)=\exp\{-\frac{\tilde{C}|\alpha|t}{8}\}\varphi(t), \quad \tilde{\psi}(t)=\exp\{-\frac{\tilde{C}|\alpha|t}{8}\}\psi(t).
$$
and choose  $|\alpha|>1$ sufficiently large such that
\begin{equation*}
\begin{split}
\frac{\alpha^2}{8}-\frac{\tilde{C}|\alpha|}{8}\geq \frac{\alpha^2}{16}.
 \end{split}
\end{equation*}
Then, it follows from \eqref{5.25} and the definition of $\tilde{\alpha}(t)$ that
\begin{equation}\label{5.26}
\begin{split}
 \frac{\mathrm{d}\tilde{ \varphi}}{\mathrm{d} t}&\leq \tilde{C}\tilde{\alpha}^{-1}\exp\{-\frac{\alpha^2 t}{8}\} \bigg ( \frac{\alpha^2}{8 }\exp\{-\frac{\tilde{C}|\alpha|t}{8}\} + \tilde{\psi}  + \tilde{ \psi}\log^+\bigg(\frac{\tilde{\varphi}}{\tilde{\psi}}\bigg)  \bigg)\varphi\\
 &\leq \tilde{C}\exp\{\alpha W_t-\frac{\alpha^2 t}{16}\} \bigg ( \frac{\alpha^2}{8} + \tilde{ \psi}  +  \tilde{ \psi} \log^+\bigg(\frac{\tilde{\varphi}}{\tilde{\psi}}\bigg)  \bigg)\tilde{\varphi}.
 \end{split}
\end{equation}
Considering the function
$$
F(\tilde{ \psi}):=\frac{\alpha^2}{8} + \tilde{ \psi}  +  \tilde{ \psi} \log^+(\frac{\tilde{\varphi}}{\tilde{\psi}}),
$$
which is actually involved  on the R.H.S. of \eqref{5.26}. It is clear that the variable of functional  $F(t)$ has the constrains $0<\tilde{ \psi}(t)\leq \psi\leq \frac{\alpha^2}{8}$ and $\tilde{\psi}< \tilde{C}\tilde{\varphi}$. We shall distinguish two cases to provide an estimation for $F(t)$.

\underline{Case 1:}  If $0<\tilde{\psi}< \tilde{\varphi}$, then $\frac{\tilde{\varphi}}{\tilde{\psi}}>1$ and we have
\begin{equation*}
\begin{split}
\tilde{\psi}\log^+(\frac{\tilde{\varphi}}{\tilde{\psi}})
=\tilde{\psi}\log(\frac{\tilde{\varphi}}{\tilde{\psi}})&=\tilde{\psi}\log( \tilde{\varphi} )-\mathbb{I}_{\tilde{\psi}\in(0,1]}\tilde{\psi}\log( \tilde{\psi} )-\mathbb{I}_{\tilde{\psi}\in[1,\infty)}\tilde{\psi}\log( \tilde{\psi} )\\
&\leq \tilde{\psi}\log( \tilde{\varphi} )-\mathbb{I}_{\tilde{\psi}\in(0,1]}\tilde{\psi}\log( \tilde{\psi} ) \\
&\leq \tilde{\psi}\log( \tilde{\varphi} )+ \frac{1}{e},
 \end{split}
\end{equation*}
which implies that
$$
F(\tilde{ \psi})\leq \frac{\alpha^2}{8}+\tilde{\psi} +\tilde{\psi}\log( \tilde{\varphi} )+ \frac{1}{e}\leq \frac{\alpha^2}{4}+ \tilde{C} +\tilde{\psi}\log( \tilde{\varphi} ).
$$

\underline{Case 2:} If $ \tilde{\psi}\geq\tilde{\varphi}$, then $\frac{\tilde{\psi}}{\tilde{C} }\leq\tilde{\varphi}\leq \tilde{\psi} $, which indicates that $\log^+(\frac{\tilde{\varphi}}{\tilde{\psi}})=0$ and  $F(\tilde{ \psi}) =\frac{\alpha^2}{8} + \tilde{ \psi} $. Since $x^2e^{-x}\leq 4e^{-2}< 1$ for all $x\geq1$, we have
$$
 \tilde{\psi}\log( \tilde{\varphi} )\geq \tilde{\psi}\log(\frac{\tilde{\psi}}{\tilde{C} })\geq -\frac{\tilde{C}^2}{e^{\tilde{C}}}\geq - \frac{4}{e^2},\quad \mbox{for all}~\tilde{\varphi}> 0.
$$
Thereby,
$$
F(\tilde{ \psi}) =\frac{\alpha^2}{8} + \tilde{ \psi} \leq \frac{\alpha^2}{4} \leq \frac{\alpha^2}{4}+\frac{4}{e^2}+\tilde{\psi}\log( \tilde{\varphi} )\leq \frac{\alpha^2}{4}+\tilde{C}+\tilde{\psi}\log( \tilde{\varphi} ).
$$

Taking the last two conclusions into consideration, the R.H.S. of \eqref{5.26} can be estimated by
\begin{equation}\label{5.27}
\begin{split}
 \frac{\mathrm{d}\tilde{ \varphi}}{\mathrm{d} t}\leq \tilde{C}\exp\{\alpha W_t-\frac{\alpha^2 t}{16}\} \bigg (\frac{\alpha^2}{4}+\tilde{C}+\tilde{\psi}\log( \tilde{\varphi} )\bigg)\tilde{\varphi}.
 \end{split}
\end{equation}
Furthermore, by virtue of the identity $\mathrm{d}(\log\tilde{ \varphi})= \tilde{ \varphi}^{-1}\mathrm{d}\tilde{ \varphi}$ and the decomposition $\exp\{\alpha W_t-\frac{\alpha^2 t}{16}\} =\exp\{\alpha W_t-\frac{\alpha^2 t}{32}\} \exp\{-\frac{\alpha^2 t}{32}\} $, we deduce  from \eqref{5.27} that
\begin{equation}\label{5.28}
\begin{split}
 \frac{\mathrm{d}(\log\tilde{ \varphi})}{\mathrm{d} t}\leq  \tilde{C}r(\frac{\alpha^2}{4}+\tilde{C})\exp\{-\frac{\alpha^2 t}{32}\}+\tilde{C}r\exp\{-\frac{\alpha^2 t}{32}\}\tilde{\psi}\log( \tilde{\varphi} ) ,\quad \forall t\in [0,\mathbbm{t}_r],
 \end{split}
\end{equation}
where $\tilde{\mathbbm{t}}_r$ is a stopping time defined by
$$
\tilde{\mathbbm{t}}_r:=\inf\bigg\{t\geq0; ~\exp\{\alpha W_t-\frac{\alpha^2 t}{32}\}\geq r\bigg\},\quad \forall r\geq 1.
$$
Solving the inequality \eqref{5.28} and using the fact of $0<\tilde{ \psi}(t)\leq \frac{\alpha^2}{8}$, we obtain
\begin{equation*}
\begin{split}
 \log(\tilde{ \varphi}(t))
 \leq& \log(\tilde{ \varphi}(0))\exp\bigg\{ \int_0^t \tilde{C}r\exp\{-\frac{\alpha^2 s}{32}\}\tilde{\psi}  \mathrm{d}s\bigg\} \\
 &+\int_0^t\tilde{C}r(\frac{\alpha^2}{4}+\tilde{C})\exp\{-\frac{\alpha^2 \tau}{32}\}     \exp\bigg\{ \int_\tau^t \tilde{C}r\exp\{-\frac{\alpha^2 s}{32}\}\tilde{\psi}  \mathrm{d}s\bigg\} \mathrm{d} \tau\\
\leq & \log(\tilde{ \varphi}(0))\exp\{4 \tilde{C}r\} +  \tilde{C}r\exp\{ 4\tilde{C}r\}\bigg(8 +\frac{32\tilde{C}}{\alpha^2}\bigg) ,
 \end{split}
\end{equation*}
which leads to
\begin{equation}\label{5.29}
\begin{split}
\tilde{ \varphi}(t)
\leq  \tilde{ \varphi}(0)\Big(1+ \tilde{ \varphi}(0)^{\exp\{4 \tilde{C}r\}-1} \Big) \exp\bigg\{\tilde{C}r\exp\{ 4\tilde{C}r\}\bigg(8 +\frac{32\tilde{C}}{\alpha^2}\bigg)\bigg\}.
 \end{split}
\end{equation}
In order to derive a bound for $\tilde{ \varphi}$ which depends only on $\alpha$ and $r$, we assume that $|\alpha|$ is sufficiently large such that $|\alpha|> 16\tilde{C}$ and
\begin{equation}\label{5.30}
\begin{split}
\tilde{ \varphi}(0)=\varphi(0)\leq \frac{|\alpha|}{16\tilde{C}A(|\alpha|,r)}:=\tilde{A}(|\alpha|,r),
 \end{split}
\end{equation}
where
$$
A(|\alpha|,r):= 2r \Bigg(1+\bigg(\frac{|\alpha|}{16\tilde{C}}\bigg)^{1-\frac{1}{2(\exp\{4\tilde{C}r\}-1)}}\Bigg)\exp\bigg\{\tilde{C}r\exp\{ 4\tilde{C}r\}\bigg(8 +\frac{32\tilde{C}}{\alpha^2}\bigg)\bigg\}.
$$
From the definition of $A(|\alpha|,r)$, it is clear that
$$
1<\tilde{A}(|\alpha|,r)\leq \frac{|\alpha|}{16\tilde{C}},
$$
and
\begin{equation*}
\begin{split}
 \lim_{r\rightarrow+\infty} \tilde{A}(|\alpha|,r)&=0,\quad \mbox{for fixed}~ \alpha ;\\
 \lim_{|\alpha|\rightarrow\infty} \tilde{A}(|\alpha|,r)&=\infty,\quad \mbox{for fixed}~ r ;\\
 \lim_{|\alpha|\rightarrow0} \tilde{A}(|\alpha|,r)&=0,\quad \mbox{for fixed}~ r .
 \end{split}
\end{equation*}
Observing that, if the parameter $|\alpha|$ is large enough such that $|\alpha|>16\tilde{C}$, then we have
\begin{equation}\label{5.31}
\begin{split}
16\tilde{C}A(|\alpha|,r) &\geq 16\tilde{C} \Bigg(1+\bigg(\frac{|\alpha|}{16\tilde{C}}\bigg)^{1-\frac{1}{2(\exp\{4\tilde{C}r\}-1)}}\Bigg)\\
&\geq ( 16\tilde{C} )^{ \frac{1}{2(\exp\{4\tilde{C}r\}-1)}} |\alpha| ^{1-\frac{1}{2(\exp\{4\tilde{C}r\}-1)}}.
 \end{split}
\end{equation}
Moreover, by using the fact of $ \exp\{4\tilde{C}r\}>2$ we get
\begin{equation}\label{5.32}
\begin{split}
\left(\frac{|\alpha|}{ ( 16\tilde{C} )^{ \frac{1}{2(\exp\{4\tilde{C}r\}-1)}} |\alpha| ^{1-\frac{1}{2(\exp\{4\tilde{C}r\}-1)}}}\right)^{\exp\{4\tilde{C}r\}-1}&=\sqrt{\frac{|\alpha|}{16\tilde{C}}}\\
&\leq \bigg( \frac{|\alpha|}{16\tilde{C}} \bigg)^{1-\frac{1}{2(\exp\{4\tilde{C}r\}-1)}}.
\end{split}
\end{equation}
Thereby, we deduce from \eqref{5.31} and  \eqref{5.32} that
\begin{equation}\label{5.33}
\begin{split}
 \bigg(\frac{|\alpha|}{16\tilde{C}A(\alpha,r)}\bigg)^{\exp\{4 \tilde{C}r\}-1}  \leq \bigg( \frac{|\alpha|}{16\tilde{C}} \bigg)^{1-\frac{1}{2(\exp\{4\tilde{C}r\}-1)}}.
 \end{split}
\end{equation}
Combining the estimates \eqref{5.29}-\eqref{5.30} and \eqref{5.33}, we obtain
\begin{equation*}
\begin{split}
\tilde{ \varphi}(t)&\leq \tilde{ \varphi}(0)\Bigg(1+ \bigg(\frac{|\alpha|}{16\tilde{C}A(\alpha,r)}\bigg)^{\exp\{4 \tilde{C}r\}-1} \Bigg) \exp\bigg\{\tilde{C}r\exp\{ 4\tilde{C}r\}\bigg(8 +\frac{32\tilde{C}}{\alpha^2}\bigg)\bigg\}\\
&\leq  \tilde{ \varphi}(0) \frac{A(\alpha,r)}{2r}\\
&\leq \frac{|\alpha|}{16\tilde{C}A(\alpha,r)} \frac{A(\alpha,r)}{2r}=\frac{|\alpha|}{32r\tilde{C}}.
 \end{split}
\end{equation*}
Recalling the definitions for $\tilde{ \varphi},\varphi$ and the stopping time $\tilde{\mathbbm{t}}_r$, the last inequality implies that
\begin{equation}\label{5.34}
\begin{split}
 \|(u_S, u_T, \theta_S)(t)\|_{\mathcal {Z}^{k,p}}&\leq \exp\{\alpha W_t+\frac{\tilde{C}|\alpha| t}{8}-\frac{\alpha^2 t}{8}\}\tilde{ \varphi}(t)\\
 &\leq  \frac{|\alpha|}{32r\tilde{C}}\exp\{\alpha W_t -\frac{\alpha^2 t}{16}\}  \\
 &= \frac{|\alpha|}{32r\tilde{C}}\exp\{\alpha W_t-\frac{\alpha^2t}{32}\} \exp\{ -\frac{\alpha^2t}{32}\}\\
 & \leq \frac{|\alpha|}{32 \tilde{C}},\quad \forall t\in[0,\tilde{\mathbbm{t}}\wedge\tilde{\mathbbm{t}}_r],
 \end{split}
\end{equation}
where the sufficiently large $|\alpha|$ is chosen as above.

\emph{\textbf{Step 4 (Global existence).}} Thanks to the estimate \eqref{5.34}, we have $\tilde{\mathbbm{t}}\wedge \tilde{\mathbbm{t}}_r=\tilde{\mathbbm{t}}_r\leq \tilde{\mathbbm{t}}$, and
$$
\sup_{t\in [0, \tilde{\mathbbm{t}}_r]}\|(u_S, u_T, \theta_S)(t)\|_{\mathcal {Z}^{1,\infty}}\leq \tilde{C}  \sup_{t\in [0, \tilde{\mathbbm{t}}_r]}\|(u_S, u_T, \theta_S)(t)\|_{\mathcal {Z}^{k,p}}\leq \frac{|\alpha|}{32}.
$$
which implies that $\tilde{\mathbbm{t}}_r\leq\mathbbm{t}$, where $\mathbbm{t}$ is the maximum existence time for the local pathwise solution $(u_S,u_T,\theta_S)$.   To obtain an estimate for $\mathbb{P}\{\tilde{\mathbbm{t}}_r=\infty\}$, we first note that on the event $\{\tilde{\mathbbm{t}}_r=\infty\}$,
$$
\Lambda(t):=\exp\{\alpha W_t-\frac{\alpha^2t}{32}\}\leq r, \quad \mbox{for all}~  t\geq 0,
$$
and $\Lambda(t)$ solves the following SDE
\begin{equation*}
\begin{split}
 \Lambda(t) =1+\frac{15\alpha ^2}{32}\int_0^t\Lambda(s) \mathrm{d}s+ \alpha \int_0^t \Lambda(t)\mathrm{d}W_s.
 \end{split}
\end{equation*}
Then, since $\Lambda(t)>0$ for all $t\in \mathbb{R}$, we can apply the  It\^{o}'s formula to $(\Lambda(t))^{\frac{1}{16}}$, and integrate the resulted equality over $[0,t\wedge \tilde{\mathbbm{t}}_r]$ to get
\begin{equation}\label{5.35}
\begin{split}
(\Lambda(t\wedge \tilde{\mathbbm{t}}_r))^{\frac{1}{16}}=1+ \int_0^{t\wedge \tilde{\mathbbm{t}}_r}(\Lambda(s))^{\frac{1}{16}}\mathrm{d}W_s.
 \end{split}
\end{equation}
We get by taking the expectation to \eqref{5.35} to obtain
\begin{equation}\label{5.36}
\begin{split}
\mathbb{E}\left[(\Lambda(t\wedge \tilde{\mathbbm{t}}_r))^{\frac{1}{16}}\right]=1.
 \end{split}
\end{equation}
Observing that if $j\leq\tilde{\mathbbm{t}}_r$, then we get from the definition of $\tilde{\mathbbm{t}}_r$ that $\Lambda(j\wedge \tilde{\mathbbm{t}}_r)=\Lambda(j)\geq r$. By virtue of \eqref{5.36} and the continuity of the measure, we have
\begin{equation*}
\begin{split}
\mathbb{P}\{\mathbbm{t}=\infty\}&\geq \mathbb{P}\{\tilde{\mathbbm{t}}_r=\infty\}=\mathbb{P}\left\{\bigcap_{j\geq1}\{\tilde{\mathbbm{t}}_r>j\right\}\\
&=\lim_{j\rightarrow\infty}\mathbb{P}\{\tilde{\mathbbm{t}}_r>j\}\\
&\geq1-\lim_{j\rightarrow\infty}\mathbb{P}\{ \Lambda(j\wedge \tilde{\mathbbm{t}}_r)\geq r\}\\
&\geq1-\lim_{j\rightarrow\infty}\mathbb{E}\left(\frac{(\Lambda(j\wedge \tilde{\mathbbm{t}}_r))^{\frac{1}{16}}}{r^{\frac{1}{16}}}\right)=1-  \frac{1}{r^{\frac{1}{16}}},
 \end{split}
\end{equation*}
which implies that the pathwise solution $(u_S,u_T,\theta_S)$ exists globally.

The proof of the Theorem \ref{thm:1.3} is now completed.
\end{proof}

\bibliographystyle{apa}
\bibliography{zhangl}

\end{document}